\documentclass[11pt,a4paper]{article}
\pdfoutput=1

\usepackage{abraces}
\usepackage[utf8]{inputenc}
\usepackage{amssymb}
\usepackage{amsmath}
\usepackage{amsfonts}
\usepackage{bbm}
\usepackage{amsthm}
\usepackage{mathrsfs}
\usepackage{hyperref}
\usepackage{color}
\usepackage[margin=2.41cm]{geometry}
\usepackage[all,cmtip]{xy}
\usepackage{graphicx}
\usepackage{float}
\usepackage{varwidth}
\usepackage{comment}
\usepackage{enumitem}
\usepackage{cancel}

\usepackage{bm}
\usepackage{mathtools}

\usepackage[inkscapelatex=false]{svg}

\usepackage{upgreek}
\usepackage{rotating}

\usepackage{tikz}
\usetikzlibrary{shapes.geometric}

\usepackage{tikz}
\usetikzlibrary{cd,nfold,matrix,arrows,calc,decorations.pathmorphing,fit,shapes.geometric,decorations.pathreplacing,positioning}
\tikzset{Rightarrow/.style={double equal sign distance,>={Implies},->},
triple/.style={-,preaction={draw,Rightarrow}},
quadruple/.style={preaction={draw,Rightarrow,shorten >=0pt},shorten >=1pt,-,double,double
distance=0.2pt}}

\usepackage{quiver}

\mathchardef\mhyphen="2D

\definecolor{darkred}{rgb}{0.8,0.1,0.1}
\hypersetup{
	colorlinks=true,         
	linkcolor=darkred,
	citecolor=blue,
}

\theoremstyle{plain}
\newtheorem{theo}{Theorem}[section]
\newtheorem{lem}[theo]{Lemma}
\newtheorem{propo}[theo]{Proposition}
\newtheorem{cor}[theo]{Corollary}

\theoremstyle{definition}
\newtheorem{defi}[theo]{Definition}

\newtheorem{exx}[theo]{Example}
\newenvironment{ex}
{\pushQED{\qed}\exx}
{\popQED\endexx}

\newtheorem{remm}[theo]{Remark}
\newenvironment{rem}
{\pushQED{\qed}\remm}
{\popQED\endremm}

\newtheorem{constrr}[theo]{Construction}
\newenvironment{constr}
{\pushQED{\qed}\constrr}
{\popQED\endconstrr}

\numberwithin{equation}{section}

\def\nn{\nonumber}

\def\bbK{\mathbb{K}}
\def\bbR{\mathbb{R}}
\def\bbC{\mathbb{C}}
\def\bbN{\mathbb{N}}

\def\h{\mathfrak{h}}
\def\g{\mathfrak{g}}

\def\PN{\mathrm{PsNat}}
\def\Mod{\mathbf{Mod}}
\def\KZ{\mathrm{KZ}}
\def\CM{\mathrm{CM}}
\def\Li{\mathrm{Li}}

\def\hom{\underline{\mathrm{hom}}}
\def\End{\mathrm{End}}

\def\Sym{\mathrm{Sym}}

\def\ps{\mathrm{ps}}
\def\id{\mathrm{id}}
\def\Id{\mathrm{Id}}
\def\ID{\mathrm{ID}}

\def\dd{\mathrm{d}}

\def\op{\mathrm{op}}

\def\Vec{\mathbf{Vec}}

\def\Ch{\mathbf{Ch}}
\def\dgMod{\mathbf{dgMod}}

\def\p{\mathbf{p}}
\def\q{\mathbf{q}}
\def\j{\mathbf{j}}
\def\k{\mathbf{k}}
\def\BB{\mathbf{B}}
\def\CC{\mathbf{C}}
\def\DD{\mathbf{D}}

\def\LL{\mathbf{L}}
\def\RR{\mathbf{R}}

\def\dgCat{\mathbf{dgCat}}

\def\ad{\mathrm{ad}}

\def\gl{\mathfrak{gl}}

\def\A{\mathcal{A}}
\def\B{\mathcal{B}}
\def\C{\mathcal{C}}

\def\F{\mathcal{F}}
\def\G{\mathcal{G}}
\def\H{\mathcal{H}}
\def\I{\mathcal{I}}
\def\L{\mathcal{L}}
\def\M{\mathcal{M}}

\def\P{\mathcal{P}}
\def\Q{\mathcal{Q}}
\def\R{\mathcal{R}}
\def\T{\mathsf{T}}
\def\U{\mathcal{U}}
\def\V{\mathcal{V}}

\makeatletter
\newcommand{\xRrightarrow}[2][]{\ext@arrow 0359\Rrightarrowfill@{#1}{#2}}
\newcommand{\Rrightarrowfill@}{\arrowfill@\equiv\equiv\Rrightarrow}
\newcommand{\xLleftarrow}[2][]{\ext@arrow 3095\Lleftarrowfill@{#1}{#2}}
\newcommand{\Lleftarrowfill@}{\arrowfill@\Lleftarrow\equiv\equiv}
\makeatother

\def\sk{\vspace{2mm}}

\makeatletter
\let\@fnsymbol\@alph
\makeatother

\title{%
Cartier integration of infinitesimal 2-braidings via 2-holonomy of the CMKZ 2-connection, I: Hexagonators and the Breen polytope
}

\author{%
Cameron James Deverall Kemp\vspace{4mm}\\
{\small School of Mathematical Sciences, University of Nottingham,}\\
{\small University Park, Nottingham NG7 2RD, United Kingdom.}\vspace{4mm}\\
{\small \begin{tabular}{ll}
Email: & \href{mailto:cameron.kemp@nottingham.ac.uk}{\texttt{cameron.kemp@nottingham.ac.uk}}
\vspace{2mm}
\end{tabular}
}
}

\date{June 2025}

\begin{document}

\maketitle

\begin{abstract}
\noindent This paper follows on from ``Syllepses from 3-shifted Poisson structures and second-order integration of infinitesimal 2-braidings". Given a symmetric strict infinitesimal 2-braiding on a symmetric strict monoidal cochain 2-category, we construct a candidate hexagonator series as a 2-holonomy with respect to the Cirio-Martins-Knizhnik-Zamolodchikov (CMKZ) fake flat 2-connection over the configuration space of 3 distinguishable particles on the complex line, $Y_3$. The second-order term of the hexagonator series is computed and found to agree with the ``infinitesimal hexagonator" from the aforementioned work. Finally, we assume a coherent totally symmetric strict infinitesimal 2-braiding and prove that the Breen polytope axiom is satisfied by translating it into a 2-loop in $Y_3$.
\end{abstract}
\vspace{-3mm}

\vspace{-2mm}

\vspace{-2mm}

\subsubsection*{Conventions and notations}
\begin{itemize}
\item Throughout this paper our base algebra $R$ is a generic commutative associative unital $\bbC$-algebra.
\item Boldface is used to denote a category or highlight the term being defined whereas italics are used for emphasis, e.g., a \textbf{cochain 2-category} is a $\Ch_R^{[-1,0]}$-category and infinitesimal 2-braidings are defined to be \textit{pseudo}natural rather than 2-natural.
\item We sometimes use the notation $\equiv$ to stress that the equality holds between pseudonatural transformations.
\item We suppress factors of $\hbar:=\frac{h}{2i\pi}$ in Sections \ref{sec:Constructing Hex} and \ref{sec:Breen} and use notations such as $\varepsilon^{t_{23}}:=e^{\ln(\varepsilon)t_{23}}$ and $\Phi_{ijk}:=\Phi_\KZ(t_{ij},t_{jk})$ to keep expressions compact.
\end{itemize}

\newpage
\tableofcontents
\newpage
\section{Introduction}
\subsection{A brief historical overview}
The Yang-Baxter equation (YBE) first arose in the context of the (1+1)-dimensional many-body problem \cite{Yang} and exactly solvable lattice models \cite{Baxter}. Specifically, given a vector space $V$, an automorphism $R\in\mathrm{Aut}(V\otimes V)$ is said to be an \textbf{$R$-matrix} if it solves the equation
\begin{equation}\label{eq:YBE}
R_{12}R_{23}R_{12}:=(R\otimes\id)(\id\otimes R)(R\otimes\id)=(\id\otimes R)(R\otimes\id)(\id\otimes R)=:R_{23}R_{12}R_{23}
\end{equation}
in $\mathrm{Aut}(V\otimes V\otimes V)$. The equation \eqref{eq:YBE} expresses a coherency condition for the scattering matrix $R$ to factorise to that of the two-body problem, see \cite{Jimbo2}. The YBE played a central r\^{o}le in the Leningrad school's development of the quantum inverse scattering method (QISM) \cite{QISM,Sklyanin}, a technique for reconstructing the potential from given scattering data of a (1+1)-dimensional quantum integrable system. The fact that the algebraic machinery of the QISM looked like a deformed symmetry algebra motivated Kulish and Reshetikhin \cite{Kulish} to construct an explicit $R$-matrix acting on higher spin representations of $\mathfrak{sl}(2,\bbC)$.
\sk

In parallel, the Knizhnik-Zamolodchikov (KZ) system of complex linear first-order partial differential equations were derived in \cite{KZ} from the SU(2) Wess–Zumino–Witten model in two-dimensional conformal field theory. In particular, the equations govern the correlation functions of the theory and it was demonstrated that the equations of the system are compatible with one another, i.e., the system is consistent.
\sk

Around this time, Drinfeld \cite{Drinfeld85} and Jimbo \cite{Jimbo} independently generalised the Kulish-Reshetikhin result and discovered the existence of nontrivial $R$-matrices on non-cocommutative Hopf algebra deformations $U_q\g$ of the universal enveloping algebra of a semisimple/simple, respectively, complex Lie algebra $\g$. In \cite[Section 10]{Drinfeld86}, Drinfeld abstracted the notion of a \textbf{quasitriangular} Hopf algebra $(H,\mu,\eta,\Delta,\varepsilon,S,R)$ as one equipped with an invertible element $R\in H\otimes H$ which controls the potential lack of cocommutativity,
\begin{equation}\label{eq:R measures quasi-cocommutativity}
\Delta^\op(x)=R\Delta(x)R^{-1}\quad,
\end{equation}
and satisfies the \textbf{hexagon} conditions:
\begin{equation}\label{eq:hexagon for quasitriangular}
[\Delta\otimes\id_H](R)=R_{13}R_{23}\qquad,\qquad[\id_H\otimes\Delta](R)=R_{13}R_{12}\quad.
\end{equation}
He then notes that the quasitriangular structure $R$ satisfies the YBE \eqref{eq:YBE} and provides, for $H$-modules $U,V\in{}_H\Mod$, an $H$-linear isomorphism from $U\otimes_HV$ to $V\otimes_HU$. For these reasons, it is sometimes called a \textit{universal} $R$-matrix.
\sk

Motivated by the form of the KZ equations (and the consistency of the system), Kohno \cite[Subsection 1.1]{Kohno} defined, from the canonical symmetric invariant 2-tensor $t\in(\Sym^2\g)^\g$ of a simple complex Lie algebra $\g$ and a generic finite-dimensional representation $\rho:\g\to\End(V)$, a flat connection on a trivial principal $G$-bundle over the configuration space of $n\in\bbN^*$ distinguishable particles on the complex line $Y_n:=\{(z_1,\ldots,z_n)\in\bbC^n|\mathrm{~if~}i\neq j\mathrm{~then~}z_i\neq z_j\}$. Such a flat connection is manifestly invariant under pullback by the natural action of $\mathrm{S}_n$ on $Y_n$ thus descends to a flat connection over the configuration space of $n$ indistinguishable particles on the complex line $X_n:=Y_n/S_n$. This, in turn, provides a monodromy representation $\rho_n^\KZ:\pi_1(X_n)\to\mathrm{Aut}(V^{\otimes n})$ which is important because the fundamental group $\pi_1(X_n)$ was already well-known\footnote{See \cite{Fox}.} to be \textbf{Artin's braid group on $n$ strands} \cite{Artin}, i.e., $\{\sigma_i\}_{1\leq i\leq n-1}$ such that $\sigma_i\sigma_j=\sigma_j\sigma_i$ if $|i-j|\geq2$ and $\sigma_i\sigma_{i+1}\sigma_i=\sigma_{i+1}\sigma_i\sigma_{i+1}$. Kohno then showed in \cite[Subsection 1.2]{Kohno} that this monodromy representation is equivalent to the one induced by the $R$-matrices of the aforementioned Drinfeld-Jimbo quantum groups $U_q\g$ which allows for a concrete description \cite[Theorem 1.2.8]{Kohno} via the explicit form in, say, \cite{Jimbo1}.  

Given a bialgebra $(A,\mu,\eta,\Delta,\epsilon)$, the category of modules ${}_A\Mod$ is equipped with a canonical monoidal structure where the monoidal product comes from the coproduct $\Delta$ and its coassociativity allows one to use the usual associativity constraint from the category of vector spaces. With a view to generalising the associativity constraint, Drinfeld introduced \cite[(1.1)-(1.4)]{Drinfeld89} the notion of a \textbf{quasibialgebra} $(A,\mu,\eta,\Delta,\varepsilon,\Phi)$ where: $(A,\mu,\eta)$ is a unital associative algebra, $\Delta:A\to A\otimes A$ is an algebra homomorphism, $\varepsilon:A\to\bbK$ is an algebra homomorphism which acts as a counit for the coproduct $\Delta$ and the \textbf{Drinfeld associator} $\Phi\in A\otimes A\otimes A$ is an invertible element such that:
\begin{subequations}\label{subeq:Drinfeld associator axioms}
\begin{alignat}{3}
[\id\otimes\Delta]\big(\Delta(a)\big)=&\,\Phi[\Delta\otimes\id]\big(\Delta(a)\big)\Phi^{-1}\quad&&,\\
[\id\otimes\id\otimes\Delta](\Phi)[\Delta\otimes\id\otimes\id](\Phi)=&\,\Phi_{234}[\id\otimes\Delta\otimes\id](\Phi)\Phi_{123}\quad&&,\\
[\id\otimes\varepsilon\otimes\id](\Phi)=&\,1\quad&&.
\end{alignat}
\end{subequations}
Even though the notion of a braided monoidal category had not been defined yet, he anticipated the need to define a quasitriangular structure $R$ on a quasibialgebra by weakening the hexagon conditions \eqref{eq:hexagon for quasitriangular} as in \cite[(3.9)]{Drinfeld89}:
\begin{equation}\label{subeq:hexagon for quasiHopf}
[\Delta\otimes\id_H](R)=\Phi_{312}R_{13}\Phi_{132}^{-1}R_{23}\Phi\qquad,\qquad[\id_H\otimes\Delta](R)=\Phi_{231}^{-1}R_{13}\Phi_{213}R_{12}\Phi^{-1}\quad.
\end{equation}
Actually, the relaxed hexagon conditions \eqref{subeq:hexagon for quasiHopf} allow for the following simple concrete ansatz for a nontrivial $R$-matrix. Suppose we are given a generic complex Lie algebra $\g$ and a symmetric invariant 2-tensor $t\in(\Sym^2)^\g$, extend the cocommutative Hopf algebra structure of $U\g$ to $(U\g)[[h]]$ in the trivial way; the ansatz $R=e^{\frac{h}{2}t}$ is clearly invertible in $(U\g)[[h]]\otimes(U\g)[[h]]$ and satisfies \eqref{eq:R measures quasi-cocommutativity}. Furthermore, in general, $e^{\frac{h}{2}(t_{13}+t_{23})}\neq e^{\frac{h}{2}t_{13}}e^{\frac{h}{2}t_{23}}$ and $e^{\frac{h}{2}(t_{13}+t_{12})}\neq e^{\frac{h}{2}t_{13}}e^{\frac{h}{2}t_{12}}$. Inspired by Kohno's results, Drinfeld \cite[Section 2]{Drinfeld90} constructed an invertible formal power series $\Phi_\KZ(\mathfrak{a},\mathfrak{b})$ via analysis of the KZ differential equations. He then showed that $\Phi:=\Phi_\KZ(t_{12},t_{23})$ is indeed a Drinfeld associator, i.e., satisfies \eqref{subeq:Drinfeld associator axioms} and \eqref{subeq:hexagon for quasiHopf}. 
\sk

Part of the motivation for Joyal and Street to introduce the concept of braided (strict) monoidal category (\cite[Section 1 of Chapter 3]{Street})\cite[Definition 2.1]{Joyal} was to formalise the intuition that the modules over a quasitriangular Hopf algebra should form such a category with the braiding being given by $\sigma:=\mathrm{flip}\circ R$. Cartier \cite{Cartier} abstracted from Drinfeld's algebraic deformation quantisation results \cite[Section 2]{Drinfeld90} the problem of integrating an infinitesimal braiding on a symmetric strict monoidal $\bbK$-linear category to construct a braided monoidal $\bbK[[h]]$-category with braiding and associator as above, see \cite[Theorem XX.6.1]{Kassel}. It was proved in \cite[Proposition 2.1]{Joyal} that braided monoidal categories satisfy a categorified version of the YBE \eqref{eq:YBE}. There are actually two ways to prove this, i.e., the following two diagrams commute: 
\begin{equation}\label{eq:2 proofs of YBE}
\begin{tikzcd}[column sep=0.2in,row sep=0.2in]
	& {(UV)W} \\
	{(VU)W} && {U(VW)} \\
	{V(UW)} && {U(WV)} \\
	{V(WU)} && {(UW)V} \\
	{(VW)U} && {(WU)V} \\
	{(WV)U} && {W(UV)} \\
	& {W(VU)}
	\arrow["{\sigma_{UV}\otimes1_W}"', from=1-2, to=2-1]
	\arrow["{\alpha_{UVW}}", from=1-2, to=2-3]
	\arrow["{\alpha_{VUW}}"', from=2-1, to=3-1]
	\arrow["{1_U\otimes\sigma_{VW}}", from=2-3, to=3-3]
	\arrow["{\sigma_{U(VW)}}"{description}, from=2-3, to=5-1]
	\arrow["{1_V\otimes\sigma_{UW}}"', from=3-1, to=4-1]
	\arrow["{\alpha^{-1}_{UWV}}", from=3-3, to=4-3]
	\arrow["{\sigma_{U(WV)}}"{description}, from=3-3, to=6-1]
	\arrow["{\alpha^{-1}_{VWU}}"', from=4-1, to=5-1]
	\arrow["{\sigma_{UW}\otimes1_V}", from=4-3, to=5-3]
	\arrow["{\sigma_{VW}\otimes1_U}"', from=5-1, to=6-1]
	\arrow["{\alpha_{WUV}}", from=5-3, to=6-3]
	\arrow["{\alpha_{WVU}}"', from=6-1, to=7-2]
	\arrow["{1_W\otimes\sigma_{UV}}", from=6-3, to=7-2]
\end{tikzcd}\quad
\begin{tikzcd}[column sep=0.2in,row sep=0.2in]
	& {(UV)W} \\
	{(VU)W} && {U(VW)} \\
	{V(UW)} && {U(WV)} \\
	{V(WU)} && {(UW)V} \\
	{(VW)U} && {(WU)V} \\
	{(WV)U} && {W(UV)} \\
	& {W(VU)}
	\arrow["{\sigma_{UV}\otimes1_W}"', from=1-2, to=2-1]
	\arrow["{\alpha_{UVW}}", from=1-2, to=2-3]
	\arrow["{\sigma_{(UV)W}}"{description}, curve={height=20pt}, from=1-2, to=6-3]
	\arrow["{\alpha_{VUW}}"', from=2-1, to=3-1]
	\arrow["{\sigma_{(VU)W}}"{description}, curve={height=-30pt}, from=2-1, to=7-2]
	\arrow["{1_U\otimes\sigma_{VW}}", from=2-3, to=3-3]
	\arrow["{1_V\otimes\sigma_{UW}}"', from=3-1, to=4-1]
	\arrow["{\alpha^{-1}_{UWV}}", from=3-3, to=4-3]
	\arrow["{\alpha^{-1}_{VWU}}"', from=4-1, to=5-1]
	\arrow["{\sigma_{UW}\otimes1_V}", from=4-3, to=5-3]
	\arrow["{\sigma_{VW}\otimes1_U}"', from=5-1, to=6-1]
	\arrow["{\alpha_{WUV}}", from=5-3, to=6-3]
	\arrow["{\alpha_{WVU}}"', from=6-1, to=7-2]
	\arrow["{1_W\otimes\sigma_{UV}}", from=6-3, to=7-2]
\end{tikzcd}
\end{equation}
The axiomatisation of braided monoidal 2-categories began with the work of Kapranov and Voevodsky \cite{Kapranov,Kapranov1}; the history of the development is an incredibly involved topic and we refer the interested reader to \cite[Section 2.1]{Schommer}. For our purposes, we briefly mention that the braiding/associator are \textit{pseudo}natural and the hexagon axioms are replaced by hexagonator modifications which satisfy higher coherence conditions. In particular, Breen \cite{Breen} recognised the necessity of the following higher coherence condition: one can fill the two diagrams \eqref{eq:2 proofs of YBE} with the components of the hexagonators and the pseudonaturality constraints of the braiding. The Breen polytope axiom then asserts that the two pasting diagrams are equal as in \cite[Figures C.13 and C.14]{Schommer}, \cite[(2.40)]{Me} and \eqref{eqn:Breen axiom}. 
\sk

As mentioned in \cite[Introduction]{Me}, the notion of an infinitesimal 2-braiding and the associated 2-categorical analogue of the KZ connection is due to Cirio and Faria Martins (CM) \cite{Joao1,Joao,Joao2}. The present paper constructs a hexagonator series and proves that it satisfies the Breen polytope axiom by studying the 2-holonomy of the CMKZ 2-connection as explained in the following plan. 

\subsection{Plan of the paper}\label{subsec:plan}
Section \ref{sec:Recap} begins by providing a recap of the essential notions in \cite[Section 2]{Me}. For example, we recall only the \textit{vertical} composition $\circ$ of pseudonatural transformations $\xi:F\Rightarrow G:\CC\to\DD$ and modifications $\Xi:\xi\Rrightarrow\xi'$ because we will not use the horizontal composition $\ast$ a single time in this paper. For the same reason, we will not recall the specific formulae for the symmetric monoidal structure $(\boxtimes,\tau)$ on $\dgCat_R^{[-1,0],\ps}$ nor the pentagonator $\Pi$ (and its axioms) of a braided strictly-unital monoidal $\Ch_R^{[-1,0]}$-category $(\CC,\otimes,I,\alpha,\sigma,\H^L,\H^R)$. We close Section \ref{sec:Recap} by quickly recalling definitions surrounding an infinitesimal 2-braiding $t:\otimes\Rightarrow\otimes:\CC\boxtimes\CC\to\CC$ from \cite[Sections 4 and 5]{Me}. In particular, Remark \ref{rem:index notation for t} recalls Cirio and Martins' notions of coherency \cite[Definition 17]{Joao} and total symmetry \cite[Definition 18]{Joao} which will play a prominent role in Section \ref{sec:Breen}.
\sk

Section \ref{sec:Prelims on 2hol} provides a primer on 2-holonomy and starts out by following the presentation of ideas in \cite[Sections 2 and 3]{Baez}. Subsection \ref{subsec:Lie/diff cross mods} recalls the 2-categorical generalisations of Lie groups $G$, Lie algebras $\g$ and connections $\A$ on trivial principal fibre bundles, i.e., Lie crossed modules $\mathcal{X}:=\big(H\xrightarrow{\mathcal{X}} G\,,\,G\xrightarrow{\rhd}\mathrm{Aut}(H)\big)$, differential crossed modules $\chi:=\big(\h\xrightarrow{\chi}\g\,,\,\g\xrightarrow{\rhd}T_e\mathrm{Aut}(H)\big)$ and fake flat 2-connections $(\A,\B)$, respectively. In order to talk about the 2-holonomy of a fake flat 2-connection we need to translate the underlying Lie crossed module $\mathcal{X}$ into a Lie 2-group $\mathrm{B}\mathcal{X}$ hence Subsubsection \ref{subsub:translation between cross mod and 2-group} presents our dictionary. It is well known that a connection on a trivial principal fibre bundle is equivalent to a $\g$-valued 1-form globally defined over the base space manifold $M$ and, furthermore, parallel transport with respect to such a connection is given by the path-ordered exponential which itself is smoothly functorial in a specific sense, see Theorem \ref{theo:hol as smooth functor} and Construction \ref{con:pathexp}. It is this latter viewpoint of emphasising smooth functoriality that forms the launch pad for the abstraction up towards 2-holonomy of a fake flat 2-connection, i.e., it is defined as a smooth 2-functor into the associated Lie 2-group, see Theorem \ref{theo:2hol characterisation}. 
\sk

Subsection \ref{subsec:Lie/diff cross mod on 2-functor} specifies these general concepts to the following case: any given $\Ch_R^{[-1,0]}$-functor $F:\BB\to\CC$ admits a Lie 2-group $\mathbf{Aut}_F$ of pseudonatural automorphisms and automodifications as in Remark \ref{rem:Lie 2-group of pseudoautos and automods}. Most importantly, because the Lie 2-group and its associated differential crossed module $\gl_F$ are of the general linear type, the path-ordered exponential specifies to an iterated integral series \eqref{eq:formal ODE solution as iterated integrals} and we can follow Cirio and Martins' method \cite[Theorem 8]{Joao2} of constructing a very explicit formula for the 2-holonomy \eqref{eq:CM def of 2hol for gl_F}. Subsection \ref{subsec:KZ 2-connection} makes a further specification to the case where we are given an infinitesimally 2-braided symmetric strict monoidal $\Ch_R^{[-1,0]}$-category $(\CC,\otimes,\gamma,t)$ and: the $\Ch_R^{[-1,0]}$-functor is given by the $n$-fold iterated monoidal product $F=\otimes^n:\CC^{\,\boxtimes(n+1)}\to\CC$, the base space manifold is given by the configuration space of $n+1$ distinguishable particles on the complex line $M=Y_{n+1}$ and the $\gl_{\otimes^n}$-valued 2-connection is given by the Cirio-Martins-Knizhnik-Zamolodchikov 2-connection of Definition \ref{def:CMKZ 2-connection pre-formal}. The rest of the paper is concerned with the case $n=2$ hence all (2-)paths will be in $Y_3$.
\sk

It can not be overstated how much the material in Section \ref{sec:Constructing Hex} relies on the exposition in \cite[Subsection 2.3]{BRW}. Subsection \ref{subsec:BRW 1-paths} begins by denoting $\bbC^\times:=\bbC\backslash\{0\}$ and $\bbC^{\times\times}:=\bbC\backslash\{0,1\}$ then uses the diffeomorphism $\vartheta:\bbC^{\times\times}\times\bbC^\times\times\bbC\cong Y_3$ \cite[(2.3)]{BRW} to pullback the CMKZ 2-connection. It is shown in \cite[Section 2.3]{BRW} that the hexagon axioms are satisfied by constructing a six-sided contractible loop which varies only in the first factor $\bbC^{\times\times}$. The corresponding parallel transports give rise to terms such as $\Phi_\KZ(t_{12},-t_{12}-t_{23})$ but, crucially, they can rewrite this term (into one appearing in the hexagon axiom) as $\Phi_\KZ(t_{12},t_{13})$ because an infinitesimal braiding satisfies the four-term relations $[t_{12},t_{13}+t_{23}]=0=[t_{23},t_{12}+t_{13}]$. An infinitesimal 2-braiding does not generally satisfy such relations but instead there are four-term relationator modifications hence we expect there to be nontrivial modifications $\Phi_\KZ(t_{12},-t_{12}-t_{23})\Rrightarrow\Phi_\KZ(t_{12},t_{13})$. 
\sk

There are two ways to produce such modifications: the first is by taking the difference of each series formula as in Construction \ref{con:mod from Phi_213} whereas the second is by finding 1-paths which produce each term then constructing what we call a ``vertically-interpolative" 2-path as in Subsection \ref{subsec:interpolative 2-paths}. The first method produces an explicit series formula as in \eqref{eq:mod from phi_213} whereas the second method produces an iterated integral series as in Proposition \ref{propo:solved 1st 2hol}. The fact that these two methods must yield the same modification\footnote{Because of the generators and relations of the Drinfeld-Kohno differential crossed module of Example \ref{ex: diff cross mod n=2}.} is a \textit{result}, not a problem. In other words, we learn what such iterated integral series evaluate to; we are not obligated to directly compute their limit. That being said, we do show in Constructions \ref{con:lim 2hol of 3rd verinterpol} and \ref{con:limit of 4th 2hol} and Proposition \ref{propo:left cong 2-hol} that such direct computations indeed reproduce the requisite series formulae; these are the feasible cases and we provide these computations for clarity's sake.
\sk

Section \ref{sec:Breen} shows that the Breen polytope axiom \eqref{eqn:Breen axiom} is satisfied by our hexagonator series if the infinitesimal 2-braiding $t$ is totally symmetric, strict and coherent. In particular, we show in Construction \ref{con:derivation of Breen equation} that total symmetry allows us to strip away instances of the symmetric braiding in the Breen polytope axiom to leave an equation purely in terms of $t$ whereas Subsubsection \ref{subsub:S_3 acts coherently} shows that coherency allows us to regard right pre-hexagonators with permuted indices as coming from  pullback 2-holonomies. Finally, Construction \ref{con:the Breen 2-loop} builds the contractible six-faced 2-loop in $\bbC^{\times\times}\times\bbC^\times$.

\newpage
\section{Recap of infinitesimal 2-braidings and hexagonators}\label{sec:Recap}
\begin{defi}\label{def:pseudonatural}
A \textbf{pseudonatural
transformation} $\xi:F\Rightarrow G:\CC\to\DD$ consists of, for each 0-cell $U\in \CC$, a 1-cell $\xi_U\in\DD[F(U),G(U)]^0$ and, for each pair of 0-cells $U,U'\in \CC$, a homotopy $\xi_{(-)}:\CC[U,U']\rightarrow\DD\left[F(U),G(U')\right][-1]$ such that, for all $f\in\CC[U,U']$ and $f'\in \CC[U',U'']$:
\begin{subequations}
\begin{alignat}{2}
G(f)\, \xi_U - \xi_{U'}\,F(f)\,&=\partial\xi_f+\xi_{\partial f}\quad&&,\label{eq:dubindex is homotopy}\\
\xi_{f' f}\,&=\,\,\xi_{f'}\, F(f) + G(f')\,\xi_f\quad&&.\label{eqn:dubindex splits prods}
\end{alignat}
\end{subequations}
A \textbf{pseudonatural endomorphism} $\upsilon\in\PN_F$ is one of the form $\upsilon : F\Rightarrow F$.
\end{defi}
Definition \ref{def:pseudonatural} is noticeably more compact than \cite[Definition 2.1]{Me}. The $\Ch_R^{[-1,0]}$-functors, $F$ and $G$, preserve units hence plugging $f'=f=1_U$ into \eqref{eqn:dubindex splits prods} gives us
\begin{equation}\label{eqn: homotopy kills unit}
\xi_{1_U}=2\xi_{1_U}\qquad\implies\qquad\xi_{1_U}=0\quad.
\end{equation}
In other words, the lax unity axiom \cite[(2.3a)]{Me} is redundant given the lax associativity axiom \cite[(2.3b)]{Me}. Furthermore, a homotopy $\xi_{(-)}:\CC[U,U']\rightarrow\DD\left[F(U),G(U')\right][-1]$ between truncated homs automatically gives $\xi_k=0$ for $k\in\CC[U,U']^{-1}$ so plugging $f=k$ into \eqref{eq:dubindex is homotopy} gives
\begin{subequations}
\begin{equation}\label{eq:naturality of homotopy components}
G(k)\, \xi_U - \xi_{U'}\,F(k)\,=\, \xi_{\partial k}\quad,
\end{equation}
whereas plugging $f=g\in\CC[U,U']^0$ into \eqref{eq:dubindex is homotopy} yields
\begin{equation}\label{eqn: dubindex ind as homotopy}
G(g)\, \xi_U - \xi_{U'}\,F(g)\,=\partial\xi_g\quad.
\end{equation}
\end{subequations}
\begin{constr}\label{con:vert comp pseudo}
Given $F\xRightarrow{\xi}G\xRightarrow{\theta}H$, the \textbf{vertical composite pseudonatural transformation} $\theta\xi:F\Rightarrow H$
is defined by:
\begin{equation}\label{eq:vercomp pseudos}
(\theta\xi)_U\,:=\, \theta_U\,\xi_U\qquad,\qquad(\theta \xi)_f\,:=\, \theta_f\,\xi_U + \theta_{U'}\,\xi_f\quad.
\end{equation}
This composition
is associative, unital with respect to the \textbf{identity pseudonatural transformation}
$1:=\Id_F : F\Rightarrow F$, defined by $(\Id_F)_U:=1_{F(U)}$ and
$(\Id_{F})_f:=0$, and functorial with respect to the \textbf{sub/postscript} notation. Explicitly, denoting:
\begin{equation}\label{eq:whiskering pseudo by functor}
(\xi'_F)_U:=\xi'_{F(U)}\quad,\quad(\xi'_F)_f:=\xi'_{F(f)}\quad\qquad,\qquad\quad F'(\xi)_U:=F'(\xi_U)\quad,\quad F'(\xi)_f:=F'(\xi_f)\,,
\end{equation} 
by $\xi'_F$ and $F'(\xi)$, respectively, it is straightforward to check:
\begin{subequations}
\begin{alignat}{4}
(\theta'\xi')_F&=\theta'_F\xi'_F\qquad&&,\qquad F'(\theta\xi)=F'(\theta)F'(\xi)\quad&&&,\\
1_F&=1\qquad&&,\qquad\,\,\,F'(1)=1\quad&&&.
\end{alignat}
\end{subequations}
\end{constr}
A pseudonatural transformation $\xi:F\Rightarrow G:\CC\rightarrow\DD$ is a \textbf{pseudonatural isomorphism} if there exists another pseudonatural transformation $\xi^{-1}:G\Rightarrow F:\CC\rightarrow\DD$ such that $\xi^{-1}\xi=\Id_F$ and $\xi\,\xi^{-1}=\Id_G$. A \textbf{pseudonatural automorphism} $\upsilon\in\PN_F^\times$ is a pseudonatural endomorphism which is also a pseudonatural isomorphism.
\begin{rem}\label{rem:pseudonatural endo is auto iff}
The formulae for the vertical composition and unit in Construction \ref{con:vert comp pseudo} tells us that a pseudonatural transformation $\xi:F\Rightarrow G:\CC\rightarrow\DD$ is a pseudonatural isomorphism if and only if the cochain components $\xi_U\in\DD[F(U),G(U)]^0$ are isomorphisms $\xi_U:F(U)\cong G(U)$. In other words, it can be easily checked that setting:
\begin{equation}
\xi^{-1}_U:=(\xi_U)^{-1}\qquad,\qquad\xi^{-1}_f:=-(\xi_{U'})^{-1}\xi_f(\xi_U)^{-1}\quad,
\end{equation}
produces a well-defined \textbf{inverse pseudonatural transformation} $\xi^{-1}:G\Rightarrow F:\CC\rightarrow\DD$.
\end{rem}

\begin{defi}\label{defi: modification}
A \textbf{modification} $\Xi:\xi\Rrightarrow\xi':F\Rightarrow G:\CC\rightarrow\DD$ consists of, for each 0-cell $U\in\CC$, a 2-cell $\Xi_U:\xi_U\Rightarrow\xi'_U:F(U)\rightarrow G(U)$, i.e.,
\begin{subequations}
\begin{equation}
\partial\Xi_U=\xi_U-\xi'_U
\end{equation}
such that, for every $f\in\CC[U,U']$,
\begin{equation}\label{eqn:mod single condition}
\Xi_{U'}F(f)+\xi_f=\xi'_f+G(f)\Xi_U\quad.
\end{equation}
\end{subequations}
An \textbf{endomodification} $\Xi\in\Mod_F$ is a modification of the form $\Xi:\xi\Rrightarrow\xi':F\Rightarrow F$.
\end{defi}
\begin{constr}\label{con:lat comp mods}
Given $\xi\overset{\Xi}{\Rrightarrow}\xi'\overset{\Xi'}{\Rrightarrow}\xi''$, the \textbf{lateral composite} $\Xi'\cdot\Xi:\xi\Rrightarrow\xi''$ is defined by
\begin{subequations}
\begin{flalign}
(\Xi'\cdot\Xi)_U \,:=\, \Xi'_U + \Xi_U\quad.
\end{flalign}
This composition is associative, unital
with respect to the \textbf{zero modification} $\mathbf{0}:=\ID_\xi:\xi\Rrightarrow\xi$, defined by $(\ID_\xi)_U:=0$, and invertible with respect to the \textbf{reverse} $\overleftarrow{\Xi}:\xi'\Rrightarrow\xi$, defined by
\begin{equation}\label{interior inverse modifications}
\overleftarrow{\Xi}_U:=-\Xi_U\quad.
\end{equation}
\end{subequations}
\end{constr}
\begin{constr}\label{con:vertcomp mods}
Given $\Xi:\xi\Rrightarrow\xi':F\Rightarrow G$ and $\Theta:\theta\Rrightarrow\theta':G\Rightarrow H$, the \textbf{vertical composite modification} $\Theta\Xi:\theta\xi\Rrightarrow\theta'\xi':F\Rightarrow H$ is defined by
\begin{flalign}\label{eq:vertcomp mods}
(\Theta \Xi)_U\,:=\, \Theta_U\xi'_U + \theta_U\Xi_U\quad.
\end{flalign}
This composition is associative, unital
with respect to the particular identities
$\ID_{\Id_F}:\Id_F\Rrightarrow\Id_F$ and functorial, i.e., the exchange law between vertical and lateral composition is upheld,
\begin{equation}\label{eq:vert and lat exchange of mods}
(\Theta' \Xi')\cdot(\Theta \Xi)=(\Theta'\cdot\Theta) (\Xi'\cdot\Xi)\quad,
\end{equation}
and the vertical composition of two lateral unit modifications is again a lateral unit,
\begin{equation}\label{eq:vercomp of lat unit mods is lat unit}
\ID_\theta\ID_\xi=\ID_{\theta\xi}\quad.
\end{equation}
Using that $\partial\Xi_U=\xi_U-\xi'_U$ and $\partial\Theta_U=\theta_U-\theta'_U$ together with the truncation of the hom-complexes allows us to equally write the vertical composition of modifications \eqref{eq:vertcomp mods} as\footnote{This fact is crucial to prove that the Peiffer identity holds for the differential crossed module of pseudonatural endomorphisms and endomodifications as in \eqref{eq:Peiffer identity holds for gl_F}.}
\begin{equation}\label{eq:alt vertcompo mod}
(\Theta \Xi)_U=\Theta_U\left(\xi_U-\partial\Xi_U\right)+\left(\theta'_U+\partial\Theta_U\right)\Xi_U=\Theta_U\xi_U+\theta'_U\Xi_U\quad.
\end{equation} 
We denote the \textbf{pre/post-whiskering}, respectively, by
\begin{equation}
\Theta\xi':=\Theta\,\ID_{\xi'}\qquad,\qquad\theta\Xi:=\ID_\theta\,\Xi\quad.
\end{equation}
Using this notation, we can rewrite \eqref{eq:vertcomp mods} and \eqref{eq:alt vertcompo mod}, respectively, as  
\begin{equation}\label{eq:nice vercomp mods}
\Theta\Xi=\Theta\xi'\cdot\theta\Xi=\theta'\Xi\cdot\Theta\xi\quad.
\end{equation}
\end{constr}
If $\xi,\xi':F\Rightarrow G$ are pseudonatural isomorphisms then we call $\Xi:\xi\Rrightarrow\xi':F\Rightarrow G$ an \textbf{isomodification}, in which case 
\begin{equation}\label{eq:inverse mod}
\Xi^{-1}:=\xi^{-1}\overleftarrow{\Xi}\xi'^{-1}:\xi^{-1}\Rrightarrow\xi'^{-1}:G\Rightarrow F
\end{equation}
clearly defines the \textbf{inverse modification}. An \textbf{automodification} $\Xi\in\Mod_F^\times$ is a modification between pseudonatural automorphisms.
\begin{constr}\label{con:add mods}
Suppose we have modifications $\Xi:\xi\Rrightarrow\xi':F\Rightarrow G$ and $\Upsilon:\upsilon\Rrightarrow\upsilon':F\Rightarrow G$ then, for $r,r'\in R$, we define $r\Xi+r'\Upsilon:r\xi+r'\upsilon\Rrightarrow z\xi'+w\upsilon':F\Rightarrow G$ by
\begin{equation}\label{eq:add pseudos and mods}
(r\xi+r'\upsilon)_U:=r\xi_U+r'\upsilon_U\,\,,\,\,
(r\xi+r'\upsilon)_f:=r\xi_f+r'\upsilon_f\qquad,\qquad(r\Xi+r'\Upsilon)_U:=r\Xi_U+r'\Upsilon_U\,.
\end{equation}
Such addition/scaling operations are evidently Abelian and $\circ$ of pseudonatural transformations \eqref{eq:vercomp pseudos}/modifications \eqref{eq:vertcomp mods} is $R$-bilinear. We define the \textbf{coboundary} of $\Xi:\xi\Rrightarrow\xi'$ as
\begin{equation}\label{eq:coboundary pseudonat}
\partial\Xi:\equiv\xi-\xi'\qquad.
\end{equation}
\end{constr}
\begin{rem}\label{rem:index notation for t}
The tricategory $\dgCat_R^{[-1,0],\ps}$ admits a canonical symmetric monoidal structure $(\boxtimes,\tau)$ which allows us to reproduce the usual definition of a symmetric strict monoidal $\Ch_R^{[-1,0]}$-category $(\CC,\otimes,\gamma)$, an \textbf{infinitesimal 2-braiding} is then defined as a pseudonatural endomorphism of the form $t:\otimes\Rightarrow\otimes:\CC^{\boxtimes2}\rightarrow\CC$. The associativity of $\boxtimes$ and $\otimes$ allows us to unambiguously define $\otimes^n:\CC^{\boxtimes(n+1)}\rightarrow\CC$. In particular, for $n=2$, we can define: 
\begin{equation}
t_{12}:\equiv\otimes(t\boxtimes1)\quad,\quad t_{23}:\equiv\otimes(1\boxtimes t)\quad,\quad t_{1(23)}:\equiv t_{\id_\CC\,\boxtimes\,\otimes}\quad,\quad  t_{(12)3}:\equiv t_{\otimes\,\boxtimes\,\id_\CC}\,.
\end{equation}
The fact that $\gamma$ is natural, involutive and satisfies the hexagon axiom means that we have
\begin{equation}
t_{13}:\equiv\otimes\left([\gamma^{-1}\boxtimes1](1\boxtimes t)_{\tau_{\CC,\CC}\,\boxtimes\,\id_\CC}[\gamma\boxtimes1]\right)\equiv\otimes\left([1\boxtimes\gamma^{-1}](t\boxtimes1)_{\id_\CC\,\boxtimes\,\tau_{\CC,\CC}}[1\boxtimes\gamma]\right)\,.
\end{equation}
We then say that $t$ is \textbf{strict} if it satisfies:
\begin{equation}\label{eq:strict t as index}
t_{1(23)}\equiv t_{12}+t_{13}\qquad,\qquad t_{(12)3}\equiv t_{23}+t_{13}\quad,
\end{equation}
which then means that we have, more generally, for $k\neq i\neq j\neq k\,$: 
\begin{equation}\label{eq:strict t notation}
t_{k(ij)}\equiv t_{ki}+t_{kj}\qquad,\qquad t_{(ij)k}\equiv t_{ik}+t_{jk}\quad.
\end{equation}
A \textbf{symmetric} $t$ satisfies $t\equiv\gamma^{-1}t_{\tau_{\CC,\CC}}\gamma$ which then means that we have, for $k\neq i\neq j\neq k\,$: 
\begin{equation}\label{eq:sym t notation}
t_{ij}\equiv t_{ji}\qquad,\qquad t_{(ij)k}\equiv t_{k(ij)}\quad.
\end{equation} 
A \textbf{totally symmetric} $t$ is a symmetric $t$ which further satisfies\footnote{See \cite[Remarks 5.18 and 5.20]{Me} for how the latter is derived.}
\begin{equation}\label{left total symmetry}
t_{\gamma_{12}}=0\qquad\implies\qquad t_{(ij)k}=t_{(ji)k}\quad.
\end{equation}
By\footnote{See \cite[Definition 5.1 to (5.7)]{Me}.} $\L:=t_{t_{12}}:[t_{12},t_{(12)3}]\Rrightarrow0$, we mean the \textbf{left 4-term relationator} given by 
\begin{equation}
\L_{UVW}:=t_{t_{UV}\boxtimes1_W}
\end{equation}
and likewise for the \textbf{right 4-term relationator} $\R:=t_{t_{23}}:[t_{23},t_{1(23)}]\Rrightarrow0$. If 
\begin{equation}
-\R_{213}=\L+\R=-\L_{132}
\end{equation}
then we say $t$ is \textbf{coherent}.
\end{rem}
We choose Drinfeld's KZ series \eqref{subeq:BRW's explicit Phi formula} as our ansatz associator $\alpha\equiv\Phi:\equiv\Phi(t_{12},t_{23})$ and the braiding as $\sigma\equiv\gamma\,e^{i\pi\hbar t}$. We denote $\Phi_{ijk}:=\Phi(t_{ij},t_{jk})$ so that, for a symmetric $t$, the \textbf{left pre-hexagonator} is of the form
\begin{subequations}\label{subeq:pre symstr pre-hexes}
\begin{equation}\label{eq:ansatz left hexagonator}
\LL:\Phi_{231}e^{i\pi\hbar t_{1(23)}}\Phi \Rrightarrow e^{i\pi\hbar t_{13}}\Phi_{213}e^{i\pi\hbar t_{12}}
\end{equation}
and the \textbf{right pre-hexagonator} is of the form
\begin{equation}\label{eq:sym t right pre-hex}
\RR:\Phi_{213}e^{i\pi\hbar t_{(12)3}}\Phi_{321}\Rrightarrow e^{i\pi\hbar t_{13}}\Phi_{231}e^{i\pi\hbar t_{23}}\quad.
\end{equation}
\end{subequations}
The left/right hexagonators are given, respectively, by $\H^L=\gamma_{1(23)}\LL$ and $\H^R=\gamma_{(12)3}\RR$. The relevant axiom we are concerned with in this paper is the \textbf{Breen polytope axiom},
\begin{align}
&\H^L_{UWV}\alpha^{-1}_{UWV}(1_U\otimes\sigma_{VW})\alpha_{UVW}+\alpha_{WVU}\sigma_{1_U\,\boxtimes \,\sigma_{VW}}\alpha_{UVW}-\alpha_{WVU}(\sigma_{VW}\otimes1_U)\alpha^{-1}_{VWU}\H^L_{UVW}\nn\\&=(1_W\otimes\sigma_{UV})\alpha_{WUV}\H^R_{UVW}\alpha_{UVW}-\sigma_{\sigma_{UV}\,\boxtimes \,1_W}-\alpha_{WVU}\H^R_{VUW}\alpha_{VUW}(\sigma_{UV}\otimes1_W)\quad.\label{eqn:Breen axiom}
\end{align}

\newpage
\section{Preliminaries on 2-holonomy}\label{sec:Prelims on 2hol}
In this section we assume that every manifold and Lie algebra is complex finite-dimensional.

\subsection{Lie crossed modules and their differential crossed modules}\label{subsec:Lie/diff cross mods}
We recall the notions of Lie/differential crossed module \cite[Definition 1.3/4]{On2dhol}; both notions will be appealed to when defining the 2-holonomy of a fake flat 2-connection.
\begin{defi}\label{def:crossed mod}
For groups $G$ and $H$, a \textbf{crossed module} $\mathcal{X}:=\big(H\xrightarrow{\mathcal{X}} G\,,\,G\xrightarrow{\rhd}\mathrm{Aut}(H)\big)$ consists of group homomorphisms $\mathcal{X}$ and $\rhd$ such that the \textbf{coboundary} $\mathcal{X}:H\rightarrow G$ is equivariant with respect to the conjugation action of $G$ on itself, 
\begin{subequations}
\begin{equation}\label{eq:conjugation equivariance}
\mathcal{X}(g\rhd h)=g\mathcal{X}(h)g^{-1}\quad,
\end{equation}
and the \textbf{Peiffer identity} is upheld,
\begin{align}\label{eq:crossmod Peiffer}
\mathcal{X}(h)\rhd h'=h\,h'\,h^{-1}\quad.
\end{align}
\end{subequations}
A \textbf{Lie crossed module} is one such that $G$ and $H$ are Lie groups and $\mathcal{X}$ and $\rhd$ are smooth.
\end{defi}
If we label the identity of every group by $e$ and denote the Lie algebra of Lie derivations of $\h:=T_eH$ by $\T_\h$ then the functor $T_e:\mathbf{Grp(Man^\mathrm{fd}_\bbC)}\rightarrow\mathbf{Lie(\Vec^\mathrm{fd}_\bbC)}$ is such that $T_e\mathrm{Aut}(H)\cong\T_\h$ for a simply connected Lie group $H$. Now we `derive' the following definition.
\begin{defi}\label{def:differential crossed module}
For Lie algebras $\g$ and $\h$, we say $\chi:=\big(\h\xrightarrow{\chi}\g\,,\,\g\xrightarrow{\rhd}\T_\h\big)$ is a \textbf{differential crossed module} if it consists of Lie algebra homomorphisms $\chi$ and $\rhd$ such that the coboundary $\chi$ is equivariant with respect to the adjoint action of $\g$ on itself, 
\begin{subequations}
\begin{equation}\label{eq:derivation equivariance}
\chi(g\rhd h)=[g,\chi(h)]\quad,
\end{equation}
and the \textbf{Peiffer identity} is upheld,
\begin{equation}\label{eq:diffcrossmod Peiffer}
\chi(h)\rhd h'=[h,h']\quad.
\end{equation}
\end{subequations}
\end{defi}
\begin{rem}\label{rem:121 Lie correspondence}
The one-to-one correspondence between simply-connected Lie groups and Lie algebras means that we can exponentiate a differential crossed module $\chi$ to a unique (up to canonical isomorphism) Lie crossed module $\mathcal{X}$ of simply-connected Lie groups, see \cite[Theorem 1.102]{Knapp} for this standard practice of exponentiating Lie algebras and their homomorphisms. 
\end{rem}
\begin{defi}\label{def:2-connection}
Given a manifold $M$ and differential crossed module $\chi=\big(\h\xrightarrow{\chi}\g\,,\,\g\xrightarrow{\rhd}\T_\h\big)$, a \textbf{$\chi$-valued 2-connection on $M$} is given by a pair $(\A,\B)$ where $\A$ is a globally defined $\g$-valued 1-form on $M$ and $\B$ is a globally defined $\h$-valued 2-form on $M$. It is called a \textbf{fake flat} 2-connection if the \textbf{$\g$-valued curvature 2-form} vanishes up to the coboundary of $\B$, i.e.,
\begin{subequations}
\begin{equation}\label{eqn:fake flat}
\F_\A:=\dd \A+\A\wedge^{[\cdot,\cdot]}\A=\chi(\B)\quad.
\end{equation}
It is called a \textbf{2-flat} 2-connection if the \textbf{$\h$-valued 2-curvature 3-form} vanishes, i.e., 
\begin{equation}\label{eqn:2-flatness}
\G_{(\A,\B)}:=\dd \B+\A\wedge^\rhd\B=0\quad.
\end{equation}
\end{subequations}
\end{defi}

\subsubsection{The translation between crossed modules and strict 2-groups}\label{subsub:translation between cross mod and 2-group}
First we recall the definition of a strict 2-group and for this we make use of the following lemma.
\begin{lem}\label{lem:hor 2-inv from ver 2-inv and 1-inv}
Given a 2-category $\C$ in which every 1-morphism $\xi:F\rightarrow G$ admits an inverse, any 2-morphism $\Xi:\xi\Rightarrow\xi':F\rightarrow G$ admits a reverse\footnote{That is, $\overleftarrow{\Xi}:\xi'\Rightarrow\xi:F\rightarrow G$ such that $\overleftarrow{\Xi}\cdot\Xi=\ID_\xi$ and $\Xi\cdot\overleftarrow{\Xi}=\ID_{\xi'}$} if and only if it admits an inverse\footnote{That is, $\Xi^{-1}:\xi^{-1}\Rightarrow\xi'^{-1}:G\rightarrow F$ such that $\Xi^{-1}\Xi=\ID_{\Id_F}$ and $\Xi\,\Xi^{-1}=\ID_{\Id_G}$}.
\end{lem}
\begin{proof}
Let us assume that $\Xi$ admits a reverse $\overleftarrow{\Xi}:\xi'\Rightarrow\xi$ and define 
\begin{subequations}
\begin{equation}
\Xi^{-1}:=\ID_{\xi^{-1}}\overleftarrow{\Xi}\ID_{\xi'^{-1}}=\overleftarrow{\ID_{\xi^{-1}}\Xi}\,\ID_{\xi'^{-1}}=\ID_{\xi^{-1}}\overleftarrow{\Xi\,\ID_{\xi'^{-1}}}:\xi^{-1}\Rightarrow\xi'^{-1}\quad,
\end{equation}
where the latter two equalities come from the functoriality of horizontal composition. We have:
\begin{alignat}{2}
\Xi^{-1}\Xi&=\left(\big[\overleftarrow{\ID_{\xi^{-1}}\Xi}\,\ID_{\xi'^{-1}}\big]\cdot\ID_{\xi^{-1}}\right)\left(\ID_{\xi'}\cdot\Xi\right)=\left(\overleftarrow{\ID_{\xi^{-1}}\Xi}\,\ID_{\xi'^{-1}}\ID_{\xi'}\right)\cdot\left(\ID_{\xi^{-1}}\Xi\right)=\ID_{\Id_F}&&,\\
\Xi\,\Xi^{-1}&=(\Xi\cdot\ID_\xi)\left(\ID_{\xi'^{-1}}\cdot\big[\ID_{\xi^{-1}}\overleftarrow{\Xi\,\ID_{\xi'^{-1}}}\big]\right)=(\Xi\,\ID_{\xi'^{-1}})\cdot\left(\ID_\xi\ID_{\xi^{-1}}\overleftarrow{\Xi\,\ID_{\xi'^{-1}}}\right)=\ID_{\Id_G}&&.
\end{alignat}
Now assume an inverse $\Xi^{-1}:\xi^{-1}\Rightarrow\xi'^{-1}$ and choose $\overleftarrow{\Xi}:=\ID_\xi\Xi^{-1}\ID_{\xi'}:\xi'\Rightarrow\xi$ so that:
\begin{alignat}{2}
\Xi\cdot\overleftarrow{\Xi}&=\left(\Xi\,\ID_{\Id_F}\right)\cdot\left(\ID_\xi\Xi^{-1}\ID_{\xi'}\right)=(\Xi\cdot\ID_\xi)\left(\ID_{\Id_F}\cdot[\Xi^{-1}\ID_{\xi'}]\right)=\Xi\,\Xi^{-1}\ID_{\xi'}=\ID_{\xi'}\quad&&,\\
\overleftarrow{\Xi}\cdot\Xi&=(\ID_\xi\Xi^{-1}\ID_{\xi'})\cdot\left(\ID_{\Id_G}\Xi\right)=([\ID_\xi\Xi^{-1}]\cdot\ID_{\Id_G})\left(\ID_{\xi'}\cdot\Xi\right)=\ID_\xi\Xi^{-1}\Xi=\ID_\xi\quad&&.
\end{alignat}
\end{subequations}
\end{proof}
\begin{defi}\label{def:strict 2-group}
A \textbf{2-groupoid} is a 2-category satisfying the bi-implication of Lemma \ref{lem:hor 2-inv from ver 2-inv and 1-inv}, i.e., every 1-morphism and 2-morphism is invertible. 
A \textbf{strict 2-group} is a single object 2-groupoid.
\end{defi}
Our translation between a crossed module and a strict 2-group uses slightly different conventions to that in \cite[Theorem 2, Subsection 3.3]{Baez}, e.g., for us the group $H$ is identified with the 2-morphisms whose \textit{target} is the identity 1-morphism $g\xRightarrow{h}e_G$. The reason we do this relates to the choice $\partial\Xi_U=\xi_U-\xi'_U$ for $\Xi_U:\xi_U\Rightarrow\xi'_U$.
\begin{constr}\label{con:translation between 2-group and crossmod}
Given a crossed module $\mathcal{X}=\left(H\xrightarrow{\mathcal{X}} G\,,\,G\xrightarrow{\rhd}\mathrm{Aut}(H)\right)$, the strict 2-group B$\mathcal{X}$ has a single object $\bullet$ with:
\begin{enumerate}
\item[(i)]1-morphisms as $g\in G$ and the composition of $\bullet\xrightarrow{g}\bullet\xrightarrow{g'}\bullet$ given by the product $\bullet\xrightarrow{g'g}\bullet$. 
\item[(ii)]2-morphisms $g\xRightarrow{h}g'$ as elements $h\in H$ such that $\mathcal{X}(h)=g(g')^{-1}$. 
\item[(iii)]Vertical composition of $g\xRightarrow{h}g'\xRightarrow{h'}g''$ as the opposite product\footnote{This makes sense as $\mathcal{X}$ is a group homomorphism hence $\mathcal{X}(hh')=\mathcal{X}(h)\mathcal{X}(h')=g(g')^{-1}g'(g'')^{-1}=g(g'')^{-1}$.} $g\xRightarrow{hh'}g''$. This composition is thus associative, invertible and unital with respect to $e_H$\footnote{Again, $H\xrightarrow{\mathcal{X}}G$ is a group homomorphism hence $\mathcal{X}(e_H)=e_G=g\,g^{-1}$ for every $g\in G$.}. 
\item[(iv)]Horizontal composition of $g\xRightarrow{h}g''$ and $g'\xRightarrow{h'}g'''$ as\footnote{This makes sense given that $\mathcal{X}(h'[g'''\rhd h])=\mathcal{X}(h')\mathcal{X}(g'''\rhd h)=\mathcal{X}(h')g'''\mathcal{X}(h)(g''')^{-1}=g'g(g'')^{-1}(g''')^{-1}$.} $g'g\xRightarrow{h'(g'''\rhd h)}g'''g''$. This composition is thus associative, functorial, unital with respect to $e_G\xRightarrow{e_H}e_G$ and invertible with respect to\footnote{This makes sense because $\mathcal{X}\left(\left[(g'')^{-1}\rhd h\right]^{-1}\right)=\left[\mathcal{X}\big((g'')^{-1}\rhd h\big)\right]^{-1}=\left[(g'')^{-1}\mathcal{X}(h)g''\right]^{-1}=g^{-1}g''$. This is obviously a left inverse but it can be shown to also be a right inverse by noticing $[(g'')^{-1}\rhd h]^{-1}=(g'')^{-1}\rhd h^{-1}$.} $g^{-1}\xRightarrow{\left[(g'')^{-1}\rhd h\right]^{-1}}(g'')^{-1}$. 
\end{enumerate}
Conversely, given a strict 2-group $\mathbf{G}$, the crossed module $\mathcal{X}_\mathbf{G}$ is constructed as follows:
\begin{enumerate}
\item[(i)] The group $G$ is the set of 1-morphisms with the product law given by the composition law.
\item[(ii)] The group $H$ is the set of 2-morphisms of the form $g'\xRightarrow{h}e_G$ with product law given by the horizontal composition of such 2-morphisms. We define $\mathcal{X}:H\longrightarrow G$ by $\mathcal{X}(h):=g'$.
\item[(iii)] Given $g'\xRightarrow{h}e_G$ and $g\in G$, we define $g\rhd h\in H$ through \textbf{whiskered-conjugation}, 
\begin{equation}
\begin{tikzcd}
	\bullet & \bullet & \bullet & \bullet
	\arrow[""{name=0, anchor=center, inner sep=0}, "{g^{-1}}"', curve={height=30pt}, from=1-1, to=1-2]
	\arrow[""{name=1, anchor=center, inner sep=0}, "{g^{-1}}", curve={height=-30pt}, from=1-1, to=1-2]
	\arrow[""{name=2, anchor=center, inner sep=0}, "{e_G}"', curve={height=30pt}, from=1-2, to=1-3]
	\arrow[""{name=3, anchor=center, inner sep=0}, "{g'}", curve={height=-30pt}, from=1-2, to=1-3]
	\arrow[""{name=4, anchor=center, inner sep=0}, "g"', curve={height=30pt}, from=1-3, to=1-4]
	\arrow[""{name=5, anchor=center, inner sep=0}, "g", curve={height=-30pt}, from=1-3, to=1-4]
	\arrow["{\Id_{g^{-1}}}"{description}, Rightarrow, from=1, to=0]
	\arrow["h"', Rightarrow, from=3, to=2]
	\arrow["{\Id_g}"{description}, Rightarrow, from=5, to=4]
\end{tikzcd}~~\stackrel{\rhd}{\longmapsto}~~\begin{tikzcd}
	\bullet &&& \bullet
	\arrow[""{name=0, anchor=center, inner sep=0}, "{e_G}"', curve={height=30pt}, from=1-1, to=1-4]
	\arrow[""{name=1, anchor=center, inner sep=0}, "{g\,g'g^{-1}}", curve={height=-30pt}, from=1-1, to=1-4]
	\arrow["{g\rhd h:=\Id_gh\Id_{g^{-1}}}"{description}, Rightarrow, from=1, to=0]
\end{tikzcd}\quad,
\end{equation}
which automatically gives us:
\begin{subequations}
\begin{alignat}{5}
\mathcal{X}(g\rhd h)&=g\,g'g^{-1}=g\mathcal{X}(h)g^{-1}\quad&&,\\e_G\rhd h&=\Id_{e_G}h\Id_{e_G^{-1}}=h\quad&&,\\(g_1g_2)\rhd h&=\Id_{g_1g_2}h\Id_{g_2^{-1}g_1^{-1}}=\Id_{g_1}\Id_{g_2}h\Id_{g_2^{-1}}\Id_{g_1^{-1}}=g_1\rhd(g_2\rhd h)\quad&&,\\
g\rhd\Id_{e_G}&=\Id_g\Id_{e_G}\Id_{g^{-1}}=\Id_{e_G}\quad&&,\\
g\rhd(h_1h_2)&=\Id_gh_1h_2\Id_{g^{-1}}=\Id_gh_1\Id_{g^{-1}}\Id_gh_2\Id_{g^{-1}}=(g\rhd h_1)(g\rhd h_2)\quad&&.
\end{alignat}
\end{subequations}
\item[(iv)] The Peiffer identity \eqref{eq:crossmod Peiffer} is secured through the pasting diagram,
\begin{equation}
\begin{tikzcd}
	\bullet &&& \bullet &&& \bullet &&& \bullet
	\arrow[""{name=0, anchor=center, inner sep=0}, "{g^{-1}}"{description}, from=1-1, to=1-4]
	\arrow[""{name=1, anchor=center, inner sep=0}, "{g^{-1}}", curve={height=-30pt}, from=1-1, to=1-4]
	\arrow[""{name=2, anchor=center, inner sep=0}, "{e_G}"', curve={height=30pt}, from=1-1, to=1-4]
	\arrow[""{name=3, anchor=center, inner sep=0}, "{e_G}"{description}, from=1-4, to=1-7]
	\arrow[""{name=4, anchor=center, inner sep=0}, "{g'}", curve={height=-30pt}, from=1-4, to=1-7]
	\arrow[""{name=5, anchor=center, inner sep=0}, "{e_G}"', curve={height=30pt}, from=1-4, to=1-7]
	\arrow[""{name=6, anchor=center, inner sep=0}, "g"{description}, from=1-7, to=1-10]
	\arrow[""{name=7, anchor=center, inner sep=0}, "g", curve={height=-30pt}, from=1-7, to=1-10]
	\arrow[""{name=8, anchor=center, inner sep=0}, "{e_G}"', curve={height=30pt}, from=1-7, to=1-10]
	\arrow["{\Id_{g^{-1}}}"', shorten >=4pt, Rightarrow, from=1, to=0]
	\arrow["{h^{-1}}"', shorten <=4pt, Rightarrow, from=0, to=2]
	\arrow["{h'}"', shorten >=4pt, Rightarrow, from=4, to=3]
	\arrow["{e_H}"', shorten <=4pt, Rightarrow, from=3, to=5]
	\arrow["{\Id_g}"', shorten >=4pt, Rightarrow, from=7, to=6]
	\arrow["h"', shorten <=4pt, Rightarrow, from=6, to=8]
\end{tikzcd}\quad,
\end{equation}
i.e., $\mathcal{X}(h)\rhd h'=g\rhd h'=\left(\Id_gh'\Id_{g^{-1}}\right)\cdot\left(he_Hh^{-1}\right)=(\Id_g\cdot h)(h'\cdot e_H)(\Id_{g^{-1}}\cdot h^{-1})=h\,h'h^{-1}$.
\end{enumerate}
\end{constr}

\subsection{2-Holonomy as a smooth 2-functor into the Lie 2-group}
First we recall the axiomatisation of parallel transport with respect to a connection on a \textit{trivial}\footnote{See Schreiber and Waldorf \cite{SW} for the generalisation to nontrivial principal $G$-bundles with connection.} principal $G$-bundle over $M$ as a smooth functor from the path groupoid $\p_1(M)$ to B$G$ as in \cite[Theorem 1]{Baez}. Rather than working with \textbf{lazy paths} (i.e., smooth maps $p:[0,1]\to M$ such that $p$ is constant in neighbourhoods of 0 and 1) we prefer to work with the notion of a \textbf{1-path} $x\xrightarrow{p}y$ (i.e., a continuous map $p:[0,1]\to M$ which is piecewise smooth with respect to some partition of the unit interval $[0,1]$ into finitely many subintervals and such that $p(0)=x$ and $p(1)=y$). The role of each of these notions is merely to secure a well-defined composition law however all the paths we will later encounter in Section \ref{sec:Constructing Hex} will be of the second sort. Likewise, rather than working with ``bigons" as in \cite[Definition 3.1]{SW1}, we prefer to work with 2-paths as in \cite[subsubsection 2.3.2]{Joao2}.

\begin{defi}\label{def:2-path}
For 1-paths $p$ and $p'$ from $x$ to $y$, a \textbf{2-path} $p\xRightarrow{\P}p'$ is a continuous map $\P:[0,1]^2\to M$ which is piecewise smooth with respect to some partition of the unit square $[0,1]^2$ by a finite number of polygons such that:
\begin{enumerate}
\item $\P(s,0)=x$ and $\P(s,1)=y$,
\item $\P(0,r)=p(r)$ and $\P(1,r)=p'(r)$. 
\end{enumerate} 
For $s\in[0,1]$, the \textbf{cross-sectional 1-path} $x\xrightarrow{\P^s}y$ defined by $\P^s:=\P(s,\cdot):[0,1]\to M$ is indeed a continuous piecewise smooth map from $x$ to $y$ because of the above conditions on $\P$ itself. The picture we have in mind is as follows, 
\begin{figure}[H]
\centering
\scalebox{0.4}{\includesvg[width=500pt]{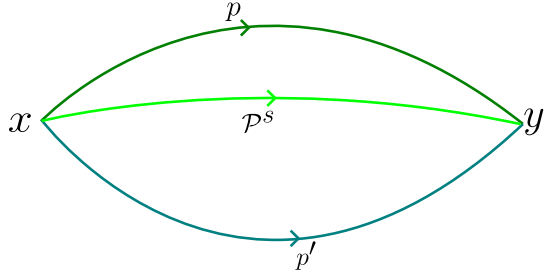}}
\caption{A generic 2-path and its cross-sectional 1-path}
\label{fig:cross-sectional 1-path}
\end{figure}
We say the 2-path $\P$ is \textbf{thin}\footnote{The name comes from the fact that thin 2-paths would sweep out zero area if there were a volume form on $M$.} if the rank of the differential $\dd\P$ is less than 2 at every point in the unit square $[0,1]^2$ for which it is defined.
\end{defi} 
Note that 2-paths $p\xRightarrow{\P}p'\xRightarrow{\P'}p''$ can be \textbf{vertically concatenated} $p\xRightarrow{\P'\cdot\P}p''$ as 
\begin{equation}\label{eqn:vert concat}
(\P'\cdot\P)(s,r):=\begin{cases}
     \P(2s,r)\,, & 0\leq s\leq\frac{1}{2}\\
      \P'(2s-1,r)\,, & \frac{1}{2}\leq s\leq1
    \end{cases}\quad.
\end{equation}
This obviously defines a new 2-path and if both $\P$ and $\P'$ are thin then so too will $\P'\cdot\P$ be. It is thus clear that thin 2-paths establish an equivalence relation on 1-paths where: transitivity comes vertical concatenation, reflexivity comes from constant 2-paths $1_p$ and symmetry comes from \textbf{reverse 2-paths} $p\xLeftarrow{\overleftarrow{\P}}p'$, i.e., $\overleftarrow{\P}:=\P\circ(\iota\times\id)$ where $\iota(s):=1-s$ is the interval inversion from \eqref{eq:interval inversion}.
\begin{rem}\label{rem:path groupoid}
Our thin 2-paths play the same role as thin homotopy equivalences in Schreiber and Waldorf, see \cite[Definition 2.1]{SW1}. Quotienting by the equivalence relation induced by thin 2-paths will guarantee the associativity, unitality and invertibility of composition. The associativity and unitality axioms could be guaranteed by working with the equivalence relation induced by smooth reparametrisation but the invertibility axiom would still be lacking. If we were only working with flat connections then it would suffice to quotient by general 2-paths. 
\end{rem}

\begin{defi}The objects of the \textbf{path groupoid} $\p_1(M)$ are simply the points of the underlying manifold $x\in M$. A morphism from $x$ to $y$ is a thin 2-path class of 1-paths, $[p]$. The composition law is given by the thin 2-path class of the concatenation of representatives, i.e., for $x\xrightarrow{p}y\xrightarrow{q}z$ we define $x\xrightarrow{qp}z$ as follows:
\begin{equation}
(qp)(r):=\begin{cases}
      p(2r)\,, & 0\leq r\leq\frac{1}{2}\\
      q(2r-1)\,, & \frac{1}{2}\leq r\leq1
    \end{cases}\quad.
\end{equation}
The identity on $x$ is given by the thin 2-path class of the constant loop at $x$, $\id_x:=[1_x]$. The inverse of a morphism is given by the thin 2-path class of the \textbf{opposite 1-path}, $[p]^{-1}:=[p\circ\iota]$.
\end{defi}
See \cite[Theorem 1]{Baez} for the following characterisation.
\begin{theo}\label{theo:hol as smooth functor}
Given a manifold $M$ and Lie group $G$ with Lie algebra $\g:=T_eG$, there is a one-to-one correspondence between the following:
\begin{enumerate}
\item A connection on a trivial principal $G$-bundle over $M$.
\item A globally defined $\g$-valued 1-form.
\item\label{item:hol} A smooth functor $W:\p_1(M)\rightarrow\mathrm{B}G$.
\end{enumerate}
\end{theo}
\begin{rem}\label{rem:hol axioms}
Item \ref{item:hol} of Theorem \ref{theo:hol as smooth functor} is called \textbf{parallel transport}. The map on morphisms $[p]$ is denoted by $W^p\in G$ and functoriality implies that:
\begin{subequations}
\begin{alignat}{2}
W^{qp}&=W^qW^p\quad&&,\\
W^{1_x}&=e_G\quad&&.
\end{alignat}
\end{subequations}
Smoothness of $W$ means the following: given a smooth function $f:[0,1]^d\times[0,1]\to M$ such that $\forall u\in[0,1]^d$ the following is a 1-path $p_u:=f(u,\cdot):[0,1]\to M$, the induced function 
\begin{subequations}
\begin{alignat}{2}
[0,1]^d&\to G\quad&&,\\
u&\mapsto W^{p_u}\quad&&,
\end{alignat}
\end{subequations}
is itself smooth.
\end{rem}
\begin{constr}\label{con:pathexp}
Exploiting the one-to-one correspondence between connections on trivial principal $G$-bundles over $M$ and globally defined $\g$-valued 1-forms $\A\in\Omega^1(M;\g)$, parallel transport is given by the \textbf{path-ordered exponential},
\begin{equation}\label{eq:path-ordered exponential}
W^p=\mathrm{P}\exp\left(\int_p\A\right)\quad.
\end{equation}
Specifically, this arises by considering the parallel transport equation in $G$, 
\begin{equation}
\frac{\dd }{\dd r}W^p_{r0}=\A\left[\frac{\dd p}{\dd r}\right]W^p_{r0}
\end{equation}
with initial condition $W^p_{00}=e_G$. We then set $W^p:=W_{10}^p$. Throughout this paper we will use formal connections and general linear Lie algebras as in Construction \ref{con:formal linear ODEs in general} hence the path-ordered exponential becomes the iterated integral series \eqref{eq:formal ODE solution as iterated integrals}. See, for example, \cite[Proposition 2]{BRW} for a proof that such iterated integral series indeed satisfy the smooth functoriality requirements of parallel transport.
\end{constr}

\subsubsection*{The abstraction up towards 2-holonomy}
We have already described the vertical concatenation of 2-paths in \eqref{eqn:vert concat} and now we also require their horizontal concatenation if we wish to describe the 2-path 2-groupoid $\p_2(M)$. We define the \textbf{horizontal concatenation},
\begin{equation}
\begin{tikzcd}
	{x} && {y} && {z}
	\arrow[""{name=0, anchor=center, inner sep=0}, "p", curve={height=-32pt}, from=1-1, to=1-3]
	\arrow[""{name=1, anchor=center, inner sep=0}, "p'"', curve={height=32pt}, from=1-1, to=1-3]
	\arrow[""{name=2, anchor=center, inner sep=0}, "{q}", curve={height=-32pt}, from=1-3, to=1-5]
	\arrow[""{name=3, anchor=center, inner sep=0}, "{q'}"', curve={height=32pt}, from=1-3, to=1-5]
	\arrow["\P\,"', shorten <=6pt, shorten >=6pt, Rightarrow, from=0, to=1]
	\arrow["{\Q\,}"', shorten <=6pt, shorten >=6pt, Rightarrow, from=2, to=3]
\end{tikzcd}
~~\longmapsto~~
\begin{tikzcd}
	{x} && {z}
	\arrow[""{name=0, anchor=center, inner sep=0}, "qp", curve={height=-32pt}, from=1-1, to=1-3]
	\arrow[""{name=1, anchor=center, inner sep=0}, "q'p'"', curve={height=32pt}, from=1-1, to=1-3]
	\arrow["\Q\P\,"', shorten <=6pt, shorten >=6pt, Rightarrow, from=0, to=1]
\end{tikzcd}\quad,
\end{equation}
as 
\begin{equation}\label{eqn:hor concat}
(\Q\P)(s,r):=\begin{cases}
     \P(s,2r)\,, & 0\leq r\leq\frac{1}{2}\\
      \Q(s,2r-1)\,, & \frac{1}{2}\leq r\leq1
    \end{cases}\quad.
\end{equation}
\begin{rem}
Again, this horizontal concatenation is well-defined because of the piecewise smoothness condition of 2-paths. However, as before we need to work with an appropriate equivalence relation on 2-paths that will guarantee the associativity, unitality and invertibility of \textit{both} concatenations as well as respect the thin 2-path equivalence on source/target 1-paths and, furthermore, enforce functoriality of horizontal composition (all this without restricting to only working with 2-flat fake flat 2-connections). The details of such ``rank-2 homotopies" are contained in \cite[Definition 1.19]{On2dhol} but note that what Faria Martins and Picken call a 2-path is the same as what Schreiber and Waldorf call a bigon. The proof that one eventually gets a 2-groupoid is contained in \cite[Appendix: Technical Lemmas]{On2dhol}.
\end{rem}
\begin{defi}
Given a Lie crossed module $\mathcal{X}=\big(H\xrightarrow{\mathcal{X}} G\,,\,G\xrightarrow{\rhd}\mathrm{Aut}(H)\big)$ and its associated differential crossed module $\chi=\big(\h\xrightarrow{\chi}\g\,,\,\g\xrightarrow{\rhd}\T_\h\big)$, a \textbf{2-connection on a trivial principal $\mathcal{X}$-bundle over $M$} is a fake flat $\chi$-valued 2-connection on $M$ as in Definition \ref{def:2-connection}.
\end{defi}
Now we are ready for the axiomatic approach to the 2-holonomy of a 2-connection on a trivial principal $\mathcal{X}$-bundle over $M$ as in \cite[Theorem 3]{Baez}.
\begin{theo}\label{theo:2hol characterisation}
Given a manifold $M$ and Lie crossed module $\mathcal{X}$ with associated differential crossed module $\chi$, there is a one-to-one correspondence between the following:
\begin{enumerate}
\item A 2-connection on a trivial principal $\mathcal{X}$-bundle over $M$.
\item\label{item:2hol as smooth 2functor} A smooth 2-functor $W:\p_2(M)\rightarrow\mathrm{B}\mathcal{X}$.
\end{enumerate}
\end{theo}
\begin{rem}\label{rem:2hol axioms}
Item \ref{item:2hol as smooth 2functor} of Theorem \ref{theo:2hol characterisation} is called \textbf{2-holonomy} and the map on 2-morphisms $[p]\xRightarrow{[\P]}[p']$ is denoted by 
\begin{equation}
W^\P:W^p\Rightarrow W^{p'}\quad.
\end{equation}
In addition to Remark \ref{rem:hol axioms}, we also have the 2-functoriality implications and for these we need our dictionary between a crossed module $\mathcal{X}$ and strict 2-group B$\mathcal{X}$ as in Construction \ref{con:translation between 2-group and crossmod}. First off, $W^\P\in H$ is such that $\mathcal{X}(W^\P)=W^p\big(W^{p'}\big)^{-1}$. For vertical concatenation of 2-paths as in \eqref{eqn:vert concat}, we require:
\begin{alignat}{2}
W^{\P'\cdot\P}&=W^{\P'}\cdot W^\P=W^\P W^{\P'}\quad&&,\\
W^{1_p}&=e_H\quad&&.\label{eq:2hol preserves ver2units}
\end{alignat}
For horizontal concatenation of 2-paths as in \eqref{eqn:hor concat}, we require:
\begin{alignat}{2}
W^{\Q\P}&=W^\Q\left(W^{q'}\rhd W^\P\right)\quad&&,\\
W^{1_{1_x}}&=e_H\quad&&.\label{eq:2hol preserves hor2units}
\end{alignat}
Actually, preserving horizontal units \eqref{eq:2hol preserves hor2units} is already achieved by preserving vertical units \eqref{eq:2hol preserves ver2units}. Smoothness means, in addition, that for a smoothly parametrised family of 2-paths $\P_s$ (where $s\in[0,1]^d$), the 2-morphism $W^{\P_s}$ depends smoothly on $s$.
\end{rem}
\begin{constr}\label{con:surface-ordered exponential}
Analogous to the path-ordered exponential of Construction \ref{con:pathexp}, there is a ``surface-ordered exponential". Given $[p]\xRightarrow{[\P]}[\delta]$, consider the differential equation in $H$, 
\begin{equation}
\frac{\partial}{\partial s}W^\P_{s0,r0}=\int_0^rW^{\P^s}_{r'0}\rhd\B\left[\frac{\partial\P^s}{\partial r'},\frac{\partial\P^s}{\partial s}\right]\dd r'W^\P_{s0,r0}
\end{equation}
with initial condition $W^\P_{00,r0}=e_H$. Once again, we consider the term $\B\left[\frac{\partial\P^s}{\partial r'},\frac{\partial\P^s}{\partial s}\right]$ as an element in $H$ through the exponential map in order for this differential equation to make sense. The \textbf{surface-ordered exponential} is the $(s=1,r=1)$-evaluation of the unique solution of this differential equation and the proof that it (together with the path-ordered exponential) indeed provides a 2-holonomy is given, for instance, in \cite[Subsubsection 2.4.5]{On2dhol}.
\end{constr}

\subsection{The general linear Lie/differential crossed module on any cochain 2-functor}\label{subsec:Lie/diff cross mod on 2-functor}
We now restrict attention to categories which are enriched over $\Ch_{R,\mathrm{fd}}^{[-1,0]}$, i.e., the category of finite-dimensional cochain complexes over $R$ concentrated in degrees $[-1,0]$. For this assumption, have in mind the homotopy $2$-category $\mathsf{Ho}_2\big({}_A\dgMod_R^{\mathrm{fgsf}}\big)$ of the dg-category of finitely generated semi-free $R$-dg-modules over a \textit{strictly} positively\footnote{The reason for restricting to the case where $A$ has finitely many generators and all of strictly positive degree is because it is the simplest way to enforce the following fact. Given any two finitely generated semi-free $A$-dg-modules
$\U,\V\in {}_A\dgMod^{\mathrm{fgsf}}$ with 
$\U^\sharp \cong A^\sharp\otimes U^\sharp$ and $\V^\sharp\cong A^\sharp\otimes V^\sharp$, 
we have the following isomorphism in ${}_{A^\sharp}\mathbf{gMod}$,
\begin{equation}
\hom_A(\U,\V)^\sharp\,\cong\,A^\sharp\otimes\big(V^\sharp\otimes (U^\sharp)^\ast\big)\quad.
\end{equation}
Thus, the vector space of $A$-linear cochain maps is finite-dimensional as well as that of the $A$-linear cochain homotopies hence the, soon to be defined, general linear differential crossed module 
\begin{equation}
\gl(\V):=\big(\hom_A(\V,\V)^{-1}\xrightarrow{\partial}\hom_A(\V,\V)^0,\rhd\big)\quad,
\end{equation}
is finite-dimensional.}, finitely generated semi-free $R$-CDGA $A$.
\subsubsection{A strict 2-group of pseudonatural automorphisms and automodifications}
\begin{rem}\label{rem:Lie 2-group of pseudoautos and automods}
We already know that $\dgCat_R^{[-1,0],\ps}[\BB,\CC]$ is a (2,1)-category;
\begin{enumerate}
\item We know $\dgCat_R^{[-1,0],\ps}[\BB,\CC](F,G)$ is a groupoid, see the lateral composition of modifications as defined in Construction \ref{con:lat comp mods}.
\item The vertical composition of pseudonatural transformations in Construction \ref{con:vert comp pseudo} and that of modifications in Construction \ref{con:vertcomp mods}, is such that the vertical composition
\begin{equation*}
\circ:\dgCat_R^{[-1,0],\ps}[\BB,\CC](G,H)\times\dgCat_R^{[-1,0],\ps}[\BB,\CC](F,G)\rightarrow\dgCat_R^{[-1,0],\ps}[\BB,\CC](F,H)
\end{equation*}
is a strictly associative functor and admits the functorial unit
\begin{equation}
\Id_F:\mathbf{1}\rightarrow\dgCat_R^{[-1,0],\ps}[\BB,\CC](F,F)\qquad\quad,\quad\qquad\bullet\longmapsto\Id_F\quad,\quad1_\bullet\longmapsto\ID_{\Id_F}\,\,.
\end{equation} 
\end{enumerate}
Given a $\Ch_R^{[-1,0]}$-functor $F:\BB\to\CC$, we recover a single object (2,1)-category $\mathbf{End}_F$ by restricting to pseudonatural endomorphisms $\xi\in\PN_F$ and endomodifications $\Xi\in\Mod_F$. In order to recover a strict 2-group $\mathbf{Aut}_F$, Lemma \ref{lem:hor 2-inv from ver 2-inv and 1-inv} tells us to restrict to pseudonatural automorphisms $\xi\in\PN_F^\times$ and automodifications $\Xi\in\Mod_F^\times$.
\end{rem}

\subsubsection*{The corresponding Lie crossed module}
Let us use Construction \ref{con:translation between 2-group and crossmod} to translate the strict 2-group $\mathbf{Aut}_F$ into the \textbf{crossed module of pseudonatural automorphisms $\xi:F\Rightarrow F$ and automodifications-to-1} $\Xi:\xi\Rrightarrow\Id_F$,
\begin{equation}
\mathrm{GL}_F:=\left(\big(\Mod^{\times\Rrightarrow1}_{F}\,,\circ\big)\xrightarrow{\mathrm{dom}}\left(\PN_{F}^\times\,,\circ\right)\,,\rhd\right)\quad.
\end{equation}
\begin{constr}\label{con:Lie cross mod of autopseudos and nice mods}
\eqref{eq:nice vercomp mods} gives the product of $\Xi:\xi\Rrightarrow\Id_F$ and $\Theta:\theta\Rrightarrow\Id_F$ as
\begin{equation}
\Theta\Xi:=\Xi\cdot\Theta\xi:\theta \xi\Rrightarrow\Id_F\qquad,\qquad(\Theta\Xi)_U=\Theta_U\xi_U+\Xi_U\quad.
\end{equation}
$\ID_{\Id_F}:\Id_F\Rrightarrow\Id_F$ is the unit for this product law. The inverse of $\Xi:\xi\Rrightarrow\Id_F$ is given by
\begin{equation}
\Xi^{-1}:=\xi^{-1}\overleftarrow{\Xi}:\xi^{-1}\Rrightarrow\Id_F\qquad,\qquad(\Xi^{-1})_U:=\big(\xi^{-1}\overleftarrow{\Xi}\big)_U=-\xi_U^{-1}\Xi_U\quad,
\end{equation}
as in \eqref{eq:inverse mod}. The coboundary $\mathrm{dom}:\big(\Mod^{\times\Rrightarrow1}_{F}\,,\circ\big)\rightarrow\left(\PN_{F}^\times\,,\circ \right)$ is given simply by 
\begin{equation}
\mathrm{dom}\left(\Xi:\xi\Rrightarrow\Id_F\right)\,=\,\xi\quad.
\end{equation}
The action map $\rhd:\PN_{F}^\times\rightarrow\mathrm{Aut}\big(\Mod^{\times\Rrightarrow1}_{F}\big)$ is given by \textbf{whiskered-conjugation}, i.e.,
\begin{equation}\label{eq:whiskered-conjugation}
\theta\rhd\Xi:=\theta\Xi\theta^{-1}:\theta\xi\theta^{-1}\Rrightarrow\Id_F\qquad,\qquad(\theta\rhd\Xi)_U:=\big(\theta\Xi\theta^{-1}\big)_U=\theta_U\Xi_U\theta_U^{-1}\,.
\end{equation}
\end{constr}

\begin{rem}\label{rem:Lie crossed mod is matrix}
Working with $\Ch_{R,\mathrm{fd}}^{[-1,0]}$-categories means that every hom 2-term complex is finite-dimensional. Taking this fact together with: Definition \ref{def:pseudonatural} of pseudonatural transformations, Definition \ref{defi: modification} of modifications and their $R$-linearity in Construction \ref{con:add mods}, we see that the space of pseudonatural endomorphisms $\PN_F$ is objectwise finite-dimensional and the same holds for the space of endomodifications $\Mod_F$. 

Looking at the characterisation of pseudonatural automorphisms in Remark \ref{rem:pseudonatural endo is auto iff}, it is clear that the group of pseudonatural automorphisms $\PN_F^\times$ is objectwise open in these local linear spaces hence $\PN_F^\times$ is indeed an objectwise finite-dimensional complex matrix Lie group.  

Given an endomodification $(\Theta:\theta\Rrightarrow\theta':F\Rightarrow F)\in\Mod_F$, we define
\begin{equation}\label{eq:crossmod widetilde mod}
\widetilde{\Theta}:=\Theta+\ID_{-\theta'+\Id_F}:\theta-\theta'+\Id_F\Rrightarrow\Id_F\quad.
\end{equation}
The map $\widetilde{(\cdot)}:\Mod_F\rightarrow\Mod_F^{\Rrightarrow1}$ establishes a one-to-one correspondence though it does not preserve the $R$-linear structure, i.e., $\widetilde{\Theta+\Xi}\neq\widetilde{\Theta}+\widetilde{\Xi}$ in general. The map $\mathrm{dom}:\Mod^{\Rrightarrow1}_F\rightarrow\PN_F$ is clearly continuous (actually, smooth) hence the preimage $\mathrm{dom}^{-1}\left(\PN_F^\times\right)=\Mod_F^{\times\Rrightarrow1}$ is open in $\Mod^{\Rrightarrow1}_F$. The map $\xi\mapsto\ID_\xi$ and the adjoint representation (of a Lie group acting on itself) are smooth hence the action map $\rhd$ is also smooth given that it is whiskered-conjugation \eqref{eq:whiskered-conjugation}. 
\end{rem}

\subsubsection*{The underlying differential crossed module}
Remark \ref{rem:Lie crossed mod is matrix} has assured us that the crossed module in Construction \ref{con:Lie cross mod of autopseudos and nice mods} is indeed a \textit{Lie} crossed module and, better yet, the Lie groups appearing are finite-dimensional complex \textit{matrix} Lie groups hence the passage to the corresponding differential crossed module is as follows:
\begin{enumerate}
\item[(i)] The Lie algebra of $\left(\PN_F^\times\,,\circ\right)$ is $\PN_F$ equipped with the commutator bracket of $\circ$.
\item[(ii)] The Lie algebra of $\big(\Mod_F^{\times\Rrightarrow1},\circ\big)$ is $\Mod_F^{\Rrightarrow0}$ equipped with the commutator bracket of $\circ$.
\item[(iii)] Differentiating $\Mod^{\times\Rrightarrow1}_F\xrightarrow{\mathrm{dom}}\PN^\times_F$ gives the coboundary $\Mod^{\Rrightarrow0}_F\xrightarrow{\partial}\PN_F$ from \eqref{eq:coboundary pseudonat}.
\item[(iv)] The derivative of the action by whiskered-conjugation is whiskered-commutation.
\end{enumerate} 
To have full confidence that this is indeed a differential crossed module, we show explicitly in the below construction that the above structures satisfy the axioms of Definition \ref{def:differential crossed module}.
\begin{constr}\label{con:diff cross mod of pseudos}
The \textbf{differential crossed module of pseudonatural endomorphisms $\xi:F\Rightarrow F$ and endomodifications-to-0} $\Xi:\xi\Rrightarrow0$ is defined as
\begin{equation}
\gl_F:=\left(\left(\Mod^{\Rrightarrow0}_{F}\,,\,[\cdot,\cdot]\right)\xrightarrow{\partial}\left(\PN_{F},[\cdot,\cdot]\right)\,,\,\rhd\right)\quad.
\end{equation}
Given pseudonatural endomorphisms $\xi,\theta\in\PN_F\,,$ the bracket given by 
\begin{equation}\label{eqn:psnat Lie bracket}
[\theta,\xi]:\equiv\theta\xi-\xi\theta
\end{equation}
is evidently a Lie bracket because it is the commutator of the associative bilinear product $\circ$. Given endomodifications-to-0, $\Xi:\xi\Rrightarrow0$ and $\Theta:\theta\Rrightarrow0$, the bracket given by 
\begin{equation}\label{eqn:mods Lie bracket}
[\Theta,\Xi]:=\Theta\Xi-\Xi\Theta:\theta\xi-\xi\theta\Rrightarrow0
\end{equation}
is also clearly a Lie bracket and the coboundary map \eqref{eq:coboundary pseudonat} given by
\begin{equation}\label{eqn:p_n differential}
\partial\Xi\equiv\xi
\end{equation}
is a Lie algebra homomorphism,
\begin{equation}
\partial([\Theta,\Xi])=\theta\xi-\xi\theta=[\theta,\xi]=[\partial\Theta,\partial\Xi]\quad.
\end{equation}
The map 
\begin{equation}
\theta\rhd\Xi:=[\theta,\Xi]:=\theta\Xi-\Xi\theta
\end{equation}
is, by the Jacobi identity, a Lie derivation of the bracket on $\Mod^{\Rrightarrow0}_F$,
\begin{equation}
\upsilon\rhd[\Theta,\Xi]=[\upsilon,[\Theta,\Xi]]=[[\upsilon,\Theta],\Xi]+[\Theta,[\upsilon,\Xi]]=[\upsilon\rhd\Theta,\Xi]+[\Theta,\upsilon\rhd\Xi]\quad,
\end{equation}
such that the induced linear map $\rhd:\mathrm{PsNats}_F\rightarrow\T_{\Mod^{\Rrightarrow0}_F}$ is a Lie algebra homomorphism,
\begin{equation}
[\upsilon,\theta]\rhd\Xi=[[\upsilon,\theta],\Xi]=[\upsilon,[\theta,\Xi]]-[\theta,[\upsilon,\Xi]]=\upsilon\rhd(\theta\rhd\Xi)-\theta\rhd(\upsilon\rhd\Xi)\quad.
\end{equation}
The equivariance property \eqref{eq:derivation equivariance} clearly holds, i.e., we have
\begin{equation}\label{eq:equivariance of action for gl_F}
\partial(\theta\rhd\Xi)=\partial\left(\theta\Xi-\Xi\theta\right)=\theta\xi-\xi\theta=[\theta,\xi]=[\theta,\partial\Xi]\quad.
\end{equation}
Finally, the check that the Peiffer identity \eqref{eq:diffcrossmod Peiffer} is upheld requires making use of \textit{both} formulae for the vertical composition of modifications \eqref{eq:nice vercomp mods} in Construction \ref{con:vertcomp mods},
\begin{align}\label{eq:Peiffer identity holds for gl_F}
[\Theta,\Xi]=\Theta\Xi-\Xi\Theta=\theta\Xi-\Xi\theta=\theta\rhd\Xi=(\partial\Theta)\rhd\Xi\quad.
\end{align}
\end{constr}
\begin{ex}\label{ex:CMs diff cross mod is a special case}
The $\Ch_R^{[-1,0]}$-category $\RR$ has: a single-object $\star$, morphisms $\RR(\star,\star):=R$ and composition given by multiplication of algebra elements. Consider also the $\Ch_R^{[-1,0]}$-category $\mathsf{Ho}_2(\Ch_R)$ of cochain complexes over $R$, cochain maps and homotopies modulo 2-homotopies. We denote by $\eta_\V:\RR\to\mathsf{Ho}_2(\Ch_R)$ the $\Ch_R^{[-1,0]}$-functor which merely picks out the cochain complex $\V$. Definition \ref{def:pseudonatural} of pseudonatural transformations tells us that a pseudonatural endomorphism $\xi:\eta_\V\Rightarrow\eta_\V$ is merely given by a cochain endomorphism $\xi_\star:\V\to\V$ whereas Definition \ref{defi: modification} of modifications tells us that an endomodification-to-zero $\Xi:\xi\Rrightarrow0$ is merely given by an endohomotopy-to-zero $\Xi_\star:\xi_\star\Rightarrow0$. In other words, $\gl_{\eta_\V}$ is simply the differential crossed module $\gl(\V)$ of cochain endomorphisms $f:\V\to\V$ and homotopies (modulo 2-homotopies) from them to 0, i.e., $h:f\Rightarrow0:\V\to\V$. See \cite[Subsection 2.2]{Joao2} for an explicit description of $\gl(\V)$.
\end{ex}

\begin{rem}\label{rem:gl_F valued 2-connection is Chern-like}
A $\gl_F$-valued 2-connection $(\A,\B)$ over $M$ consists of a $\PN_F$-valued 1-form $\A$ on $M$ and a $\Mod_F^{\Rrightarrow0}$-valued 2-form $\B$ on $M$. Tautologically, we use \eqref{eqn:p_n differential} to write
\begin{equation}
\B:\partial(\B)\Rrightarrow0
\end{equation}
hence $(\A,\B)$ is fake flat \eqref{eqn:fake flat} if and only if\footnote{In this case, the domain pseudonatural endomorphism of the 2-curvature 3-form $\G_{(\A,\B)}$ becomes
\begin{equation}
\G_{(\A,\B)}:=\dd \B+\A\wedge^\rhd\B:\dd\F_\A+\A\wedge^{[\cdot,\cdot]}\F_\A\Rrightarrow0
\end{equation}
and we can use the definition of the curvature 2-form in \eqref{eqn:fake flat} to write
\begin{equation}
\dd\F_\A+\A\wedge^{[\cdot,\cdot]}\F_\A=(\dd\A)\wedge^{[\cdot,\cdot]}\A+\A\wedge^{[\cdot,\cdot]}(\A\wedge^{[\cdot,\cdot]}\A)
\end{equation}
which is reminiscent of the Chern-Simons 3-form. If 2-flatness also holds then $\G_{(\A,\B)}=\mathbf{0}$ hence
\begin{equation}
(\dd\A)\wedge^{[\cdot,\cdot]}\A+\A\wedge^{[\cdot,\cdot]}(\A\wedge^{[\cdot,\cdot]}\A)=0\quad.
\end{equation}}
\begin{equation}\label{eq:mathcal B is mod from curv of A to 0}
\B:\F_\A\Rrightarrow0\quad.
\end{equation}
\end{rem}
\subsubsection{Cirio-Martins' 2-holonomy for a general linear crossed module} 
We need not go through the effort to translate the Lie crossed module $\mathrm{GL}_F$ into a strict Lie 2-group using Construction \ref{con:translation between 2-group and crossmod} given that it will result in $\mathbf{Aut}_F$. Within this present context, we now interpret the one-to-one correspondence of Theorem \ref{theo:2hol characterisation} as a definition of Item \ref{item:2hol as smooth 2functor}.
\begin{defi}\label{def:Aut_F-valued 2-holonomy over M}
Given a $\Ch_R^{[-1,0]}$-functor $F:\BB\rightarrow\CC$ and a $\gl_F$-valued fake flat 2-connection $(\A,\B)$ on a manifold $M$, a \textbf{2-holonomy} $W$ over $M$ is given by the following two pieces of data:
\begin{enumerate}
\item[(i)] For every 1-path $x\xrightarrow{p}y$ in $M$, a pseudonatural automorphism $W^p:F\Rightarrow F$.
\item[(ii)] For every 2-path $p\xRightarrow{\P}p'$ in $M$, an automodification  $W^\P$.
\end{enumerate}
These two pieces of data are required to satisfy the following three axioms which encapsulate the fact that $W:\p_2(M)\rightarrow\mathbf{Aut}_F$ is well-defined as a smooth 2-functor:
\begin{enumerate}
\item The map on 1-paths must be smoothly functorial and respect thin 2-path equivalence, i.e.:
\begin{itemize}
\item Given $x\xrightarrow{p}y\xrightarrow{q}z$, we must have $W^{qp}=W^qW^p$ and we must also have $W^{1_x}=\Id_F$.
\item If $[p]=[p']$ then $W^p=W^{p'}$.
\item A smoothly parametrised family of 1-paths induces a smooth map into $\PN_F^\times$.
\end{itemize}
\item \textbf{Globularity}, i.e., the automodifications must have the appropriate source/target,
\begin{equation}\label{eq:globularity condition}
W^\P:W^p\Rrightarrow W^{p'}\quad.
\end{equation}
\item The map on 2-paths must be smoothly 2-functorial and respect rank-2 homotopy equivalence:
\begin{itemize}
\item The preservation of vertical composition requires, for $p\xRightarrow{\P}p'\xRightarrow{\P'}p''$ : 
\begin{equation}\label{eq:preserv of vert comp 2-paths}
W^{\P'\cdot\P}=W^{\P'}\cdot W^\P\qquad,\qquad W^{1_p}=\mathbf{0}\quad.
\end{equation}
\item The preservation of horizontal composition requires, for
\begin{equation}
\begin{tikzcd}
	{x} && {y} && {z}
	\arrow[""{name=0, anchor=center, inner sep=0}, "p", curve={height=-32pt}, from=1-1, to=1-3]
	\arrow[""{name=1, anchor=center, inner sep=0}, "p'"', curve={height=32pt}, from=1-1, to=1-3]
	\arrow[""{name=2, anchor=center, inner sep=0}, "q", curve={height=-32pt}, from=1-3, to=1-5]
	\arrow[""{name=3, anchor=center, inner sep=0}, "q'"', curve={height=32pt}, from=1-3, to=1-5]
	\arrow["\P\,"', shorten <=6pt, shorten >=6pt, Rightarrow, from=0, to=1]
	\arrow["{\Q\,}"', shorten <=6pt, shorten >=6pt, Rightarrow, from=2, to=3]
\end{tikzcd}\quad,\quad W^{\Q\P}=W^\Q W^\P=W^\Q W^{p'}\cdot W^qW^\P\,.
\end{equation}
\item If $[\P]=[\P']$ then $W^\P=W^{\P'}$.
\item A smoothly parametrised family of 2-paths induces a smooth map into $\Mod_F^\times$.
\end{itemize}
\end{enumerate}
\end{defi}
Example \ref{ex:CMs diff cross mod is a special case} mentioned that CM's differential crossed module \cite[Subsection 2.2]{Joao2} is a special case of ours (i.e., $\gl(\V)=\gl_{\eta_\V}$) hence we now follow their construction\footnote{Note that throughout CM's publications compositions are written from left to right.} of a more explicit formula for the 2-holonomy (than that of Construction \ref{con:surface-ordered exponential}) in \cite[Theorem 8]{Joao2}. 
\begin{constr}\label{con:Aut_F valued 2-holonomy by CM}
Recall the notion of a cross-sectional 1-path $x\xrightarrow{\P^s}y$ of a 2-path $\P^0\xRightarrow{\P}\P^1$ from Definition \ref{def:2-path}. Using this notation we recall an important fact\footnote{See, for example, \cite[(1.7) of Lemma 1.2]{On2dhol} and \cite[(3.13) of Lemma 3.3]{FARIAMARTINS2011179}.},
\begin{equation}\label{eq:Ambrose-Singer theorem}
\frac{\dd}{\dd s}W^{\P^s}=W^{\P^s}\int_0^1\left(W^{\P^s}_{r0}\right)^{-1}\F_\A\left[\frac{\partial\P}{\partial r},\frac{\partial\P}{\partial s}\right]W^{\P^s}_{r0}\dd r=\int_0^1W^{\P^s}_{1r}\F_\A\left[\frac{\partial\P}{\partial r},\frac{\partial\P}{\partial s}\right]W^{\P^s}_{r0}\dd r\,.
\end{equation}
We emphasise that the differential crossed module being $\gl_F$ and the 2-connection $(\A,\B)$ being fake flat allows us to use \eqref{eq:mathcal B is mod from curv of A to 0} in Remark \ref{rem:gl_F valued 2-connection is Chern-like} to write
\begin{equation}
\int_0^1\int_0^1W^{\P^s}_{1r}\B\left[\frac{\partial\P}{\partial s},\frac{\partial\P}{\partial r}\right]W^{\P^s}_{r0}\dd r\dd s:\int_0^1\int_0^1W^{\P^s}_{1r}\F_\A\left[\frac{\partial\P}{\partial s},\frac{\partial\P}{\partial r}\right]W^{\P^s}_{r0}\dd r\dd s\Rrightarrow0\,.
\end{equation}
Looking at the globularity condition in Definition \ref{def:Aut_F-valued 2-holonomy over M}, we can see that we must take the assignment of an automodification $W^\P$ to a 2-path $\P$ as
\begin{equation}\label{eq:CM def of 2hol for gl_F}
\int_0^1\int_0^1W^{\P^s}_{1r}\B\left[\frac{\partial\P}{\partial s},\frac{\partial\P}{\partial r}\right]W^{\P^s}_{r0}\dd r\dd s+\ID_{W^{\P^1}}:\int_1^0\frac{\dd}{\dd s}W^{\P^s}\dd s+W^{\P^1}=W^{\P^0}\Rrightarrow W^{\P^1}.
\end{equation}
As CM show in \cite[Theorem 8]{Joao2}, this 2-holonomy indeed preserves the vertical and horizontal concatenation of 2-paths. Furthermore, this 2-holonomy is invariant under rank-2 homotopy equivalence of 2-paths and is smooth with respect to the 2-path argument. CM then go on to show in \cite[Theorem 11, Corollary 12]{Joao2} that if $(\A,\B)$ is 2-flat then the 2-holonomy is invariant under 2-homotopies thus the 2-holonomy along a contractible 2-loop is always zero. 
\end{constr}


\subsection{The CMKZ 2-connection on complex configuration space}\label{subsec:KZ 2-connection}
Given some $n\in\bbN$ and an infinitesimally 2-braided symmetric strict monoidal $\Ch_R^{[-1,0]}$-category $(\CC,\otimes,\gamma,t)$, we consider the $\Ch_R^{[-1,0]}$-functor $\otimes^n:\CC^{\,\boxtimes(n+1)}\to\CC$ and take the base space manifold $M$ as the \textbf{configuration space of $n+1$ distinguishable particles on the complex line}, 
\begin{equation}
Y_{n+1}:=\{(z_1,\ldots,z_{n+1})\in\bbC^{n+1}|\mathrm{~if~}i\neq j\mathrm{~then~}z_i\neq z_j\}\quad.
\end{equation}
\begin{defi}\label{def:CMKZ 2-connection pre-formal}
The \textbf{Cirio-Martins-Knizhnik-Zamolodchikov 2-connection} is the $\gl_{\otimes^n}$-valued 2-connection $(\A_\KZ,\B_\CM)$ on $Y_{n+1}$ is given by:
\begin{subequations}\label{subeq:CMKZ 2-connection on Y_n+1}
\begin{align}
\A_\KZ&:\equiv\sum_{1\leq i<j\leq n+1}\omega_{ij}t_{ij}\quad,\label{eq:KZ 1-form on Y_n+1}\\
\B_\CM&:=2\sum_{1\leq i<j<k\leq n+1}\omega_{ij}\wedge\omega_{ki}\L_{ijk}+\omega_{jk}\wedge\omega_{ki}\R_{ijk}\quad,\label{eq:CM 2-form on Y_n+1}
\end{align}
\end{subequations}
where $\omega_{ij}:=\frac{\dd z_i-\dd z_j}{z_i-z_j}$.
\end{defi}
\begin{rem}\label{rem:2-connection is fucking great}
As CM show in \cite[Section 3]{Joao}: 
\begin{enumerate}
\item For a strict $t$, this $\gl_{\otimes^n}$-valued 2-connection is \textit{by construction} fake flat, i.e.,
\begin{equation}\label{eqn:fake flatness}
\F_{\A_\KZ}:\equiv\dd \A_\KZ+\A_\KZ\wedge^{[\cdot,\cdot]}\A_\KZ\equiv\A_\KZ\wedge^{[\cdot,\cdot]}\A_\KZ\equiv\partial\B_\CM\quad.
\end{equation}
In this case there exists a well-defined $\mathbf{Aut}_{\otimes^n}$-valued 2-holonomy over $Y_{n+1}$ whereby 1-paths are assigned pseudonatural automorphisms of $\otimes^n$ (provided by the path-ordered exponential) and 2-paths are assigned modifications between these pseudonatural automorphisms of $\otimes^n$ and such automodifications are explicitly given by \eqref{eq:CM def of 2hol for gl_F}.
\item\label{item:2-curv} The 2-curvature 3-form, 
\begin{equation}\label{eqn:2-curv 3-form}
\G_{(\A_\KZ,\B_\CM)}:=\dd \B_\CM+\A_\KZ\wedge^\rhd\B_\CM=\A_\KZ\wedge^\rhd\B_\CM\quad,
\end{equation}
vanishes if and only if the relations in \cite[Theorem 21 and Theorem 22]{Joao} hold\footnote{See also \cite[Theorem 10]{Joao1}.}. As \cite[Theorem 23]{Joao} says, a coherent strict $t$ provides a 2-flat fake flat 2-connection.
\end{enumerate}
\end{rem}
\subsubsection{Drinfeld-Kohno differential crossed modules}
Let us explicate the fundamental examples $n\in\{0,1,2\}$.
\begin{ex}[\textbf{$0^\mathrm{th}$ Drinfeld-Kohno differential crossed module}]
Consider the differential crossed submodule of $\gl_{\id_\CC}$ generated by the empty set. That is, consider multiples of $1:=\Id_{\id_\CC}$ and $\mathbf{0}:=\ID_{\Id_{\id_\CC}}$. In this instance, we do not consider any 2-connection on $Y_1=\bbC$.
\end{ex}
\begin{ex}[\textbf{$1^\mathrm{st}$ Drinfeld-Kohno differential crossed module}]
We consider the differential crossed submodule of $\gl_\otimes$ generated (in the associative unital sense) by $\{t_{12},t_{21}\}$ subject to the conditions of a symmetric $t$. In short, we consider the algebra generated by $t$ together with the units $1:=\Id_\otimes$ and $\mathbf{0}:=\ID_{\Id_\otimes}$. In this instance, we consider the flat connection
\begin{equation}
\A^{n=1}_\KZ=\omega_{12}t
\end{equation}
on $Y_2=\{(z_1,z_2)\in\bbC^2\,|\,z_1\neq z_2\}$.
\end{ex}
\begin{ex}[\textbf{$2^\mathrm{nd}$ Drinfeld-Kohno differential crossed module}]\label{ex: diff cross mod n=2}
We consider the differential crossed submodule of $\gl_{\otimes^2}$ generated (in the associative unital sense) by
\begin{equation}
\big\{t_{ij}\,,\,t_{(ij)k}\,,\,t_{i(jk)}\,,\,\L_{ijk}:[t_{ij}\,,\,t_{(ij)k}]\Rrightarrow0\,,\,\R_{ijk}:[t_{jk}\,,\,t_{i(jk)}]\Rrightarrow0\,|\,\{i,j,k\}=\{1,2,3\}\,\big\}
\end{equation}
subject to the conditions in Remark \ref{rem:index notation for t} of a coherent totally symmetric strict $t$:
\sk

1. Symmetry of $t$ means: we can eliminate $t_{ji}$ in favour of $t_{ij}$, we can eliminate $t_{k(ij)}$ in favour of $t_{(ij)k}$ and \cite[Lemma 5.5]{Me} gives us 
\begin{equation}
\L=\R_{312}\,,\quad\L_{213}=\R_{321}\,,\quad
\R=\L_{231}\,,\quad\R_{132}=\L_{321}\,,\quad
\L_{132}=\R_{213}\,,\quad\L_{312}=\R_{231}\,.
\end{equation}

2. Total symmetry means we can eliminate $t_{(ji)k}$ in favour of $t_{(ij)k}$ and \cite[Lemma 5.21]{Me} gives us:
\begin{equation}
\L=\R_{312}=\L_{213}=\R_{321}\,,\quad
\R=\L_{231}=\R_{132}=\L_{321}\,,\quad
\L_{132}=\R_{213}=\L_{312}=\R_{231}\,.
\end{equation}

3. Strictness means we have: 
\begin{equation}
t_{k(ij)}\equiv t_{ki}+t_{kj}\qquad,\qquad
t_{(ij)k}\equiv t_{ik}+t_{jk}\,.
\end{equation}

4. Coherency means we have
\begin{equation}
-\R_{213}=\L+\R=-\L_{132}\quad.
\end{equation}
Putting it all together for a coherent totally symmetric strict $t$, the submodule of $\gl_{\otimes^2}$ is generated by only $t_{12}\,,t_{23}\,,t_{13}\in\PN_{\otimes^2}$ and $\L,\R\in\Mod^{\Rrightarrow0}_{\otimes^2}$ where:
\begin{equation}\label{eqn:2ndDK relations}
\partial\L\equiv[t_{12}\,,t_{13}+t_{23}]\qquad,\qquad
\partial\R\equiv[t_{23}\,,t_{12}+t_{13}]\quad.
\end{equation}
The $\gl_{\otimes^2}$-valued 2-connection on $Y_3:=\{(z_1,z_2,z_3)\in\bbC^3\,|\,z_1\neq z_2\neq z_3\neq z_1\}$ is given by:
\begin{subequations}\label{subeq:CMKZ 2-connection on Y_3}
\begin{align}
\A^{n=2}_\KZ&=\left(\frac{\dd z_1-\dd z_2}{z_1-z_2}\right)t_{12}+\left(\frac{\dd z_1-\dd z_3}{z_1-z_3}\right)t_{13}+\left(\frac{\dd z_2-\dd z_3}{z_2-z_3}\right)t_{23}\quad,\label{eq:KZ 1-form on Y_3}\\
\B^{n=2}_\CM&=\frac{2}{(z_3-z_1)}\left(\frac{\R}{(z_2-z_3)}-\frac{\L}{(z_1-z_2)}\right)(\dd z_1\wedge\dd z_2+\dd z_2\wedge\dd z_3+\dd z_3\wedge\dd z_1)\quad.\label{eq:CM 2-form on Y_3}
\end{align}
\end{subequations}
This 2-connection is fake flat if $t$ is strict but the 2-curvature 3-form \eqref{eqn:2-curv 3-form} always vanishes regardless of whether $t$ is coherent or not; a fact more easily seen to hold for the pullback 2-connection \eqref{sub:pullback formal 2-connection}   
\end{ex}


\newpage
\section{A hexagonator series as a 2-holonomy of the CMKZ 2-connection}\label{sec:Constructing Hex}
Throughout this section we assume the infinitesimal 2-braiding $t$ is symmetric and strict. 
\begin{rem}
Cartier integration \cite[Theorem XX.6.1]{Kassel} formalises the notion that the infinitesimal hexagon relations \eqref{eq:strict t as index} are truly first-order relations in some deformation parameter $\hbar$. One defines $\CC[[\hbar]]$ as the category with the same objects as $\CC$ but with morphisms now given by formal power series $\sum_{k\geq0}\hbar^kf_k$ such that the composition and monoidal product are extended in the obvious way. To realise $\gamma^{-1} \sigma:=e^{i\pi\hbar t}$ (a pseudonatural automorphism of $\otimes:\CC[[\hbar]]\boxtimes\CC[[\hbar]]\rightarrow\CC[[\hbar]]$) as a concrete parallel transport, we would need to formalise the CMKZ 2-connection \eqref{subeq:CMKZ 2-connection on Y_n+1} by adding a prefactor of $\hbar$ to $\A_\KZ$ (and, to secure fake flatness \eqref{eqn:fake flatness}, a prefactor of $\hbar^2$ to $\B_\CM$). That being said, it is rather tedious to carry round factors of $\hbar$ hence we absorb them into $t$, $\L$ and $\R$.
\end{rem}   

\subsection{The adapted Bordemann-Rivezzi-Weigel contractible hexagon of 1-paths}\label{subsec:BRW 1-paths}
Following \cite[(2.3)]{BRW}, we denote the \textbf{singly-punctured complex line} by $\bbC^\times:=\bbC\backslash\{0\}$ and the \textbf{doubly-punctured complex line} by $\bbC^{\times\times}:=\bbC\backslash\{0,1\}$. We pullback the formal 2-connection \eqref{subeq:CMKZ 2-connection on Y_3} along the birational biholomorphism
\begin{subequations}\label{subeq:the biholomorphism vartheta}
\begin{alignat}{2}
\vartheta:\bbC^{\times\times}\times\bbC^\times\times\bbC&\cong Y_3\quad&&,\\
(z,v,w)&\mapsto(w,zv+w,v+w)\quad&&,
\end{alignat}
\end{subequations}
which gives us:
\begin{subequations}\label{sub:pullback formal 2-connection}
\begin{alignat}{2}
\nabla:\equiv\vartheta^*\A^{n=2}_\KZ&=\left(\frac{t_{12}}{z}+\frac{t_{23}}{z-1}\right)\dd z+\frac{t_{12}+t_{23}+t_{13}}{v}\dd v\quad&&,\\
\Delta:=\vartheta^*\B^{n=2}_\CM&=2\left(\frac{\L}{zv}+\frac{\R}{(z-1)v}\right)\dd z\wedge\dd v\quad&&.
\end{alignat}
\end{subequations}
\begin{rem}\label{rem:pullback 2-connection is always 2-flat}
Fake flatness \eqref{eqn:fake flatness} of the 2-connection \eqref{sub:pullback formal 2-connection} follows from $t$ being strict \eqref{eqn:2ndDK relations} whereas 2-flatness (i.e., \eqref{eqn:2-curv 3-form} vanishing) \textit{automatically} holds. \cite[Proposition 5.29]{Me} demonstrated that, at second order, the Breen polytope commuting is intrinsically related to coherency of $t$ hence this happening for all orders must be for a reason \textit{more} than simply that the 2-connection $(\nabla,\Delta)$ is 2-flat. In Subsubsection \ref{subsub:S_3 acts coherently}, we will show that coherency of a totally symmetric $t$ means that the 2-connection $(\nabla,\Delta)$ is ``coherently right acted on" by $\mathrm{S}_3$; this subtle property makes the Breen polytope commute at all orders and not merely 2-flatness.
\end{rem}
We denote $\nabla:\equiv\nabla(t_{12},t_{23},t_{13})$ and describe, side-by-side, the transfers along the diffeomorphism \eqref{subeq:the biholomorphism vartheta} of the left action of $\mathrm{S}_3$ on $Y_3$ together with the pullback by such left actions:
\begin{subequations}\label{subeq:S_3 right action on 1-form}
\begin{alignat}{6}
\tau_{12}(z,v,w)&=\left(\frac{z}{z-1},(1-z)v,zv+w\right)&&\qquad\implies\qquad\tau_{12}^*\nabla&&&\equiv\nabla(t_{12}\,,t_{13}\,,t_{23})\,,\label{eq:tau_12 pullback}\\
\tau_{23}(z,v,w)&=\left(\frac{1}{z},zv,w\right)&&\qquad\implies\qquad\tau_{23}^*\nabla&&&\equiv\nabla(t_{13}\,,t_{23}\,,t_{12})\,,\\
\tau_{13}(z,v,w)&=\left(1-z,-v,v+w\right)&&\qquad\implies\qquad\tau_{13}^*\nabla&&&\equiv\nabla(t_{23}\,,t_{12}\,,t_{13})\,,\\
\tau_{(12)3}(z,v,w)&=\left(\frac{1}{1-z},(z-1)v,v+w\right)&&\qquad\implies\qquad\tau_{(12)3}^*\nabla&&&\equiv\nabla(t_{23}\,,t_{13}\,,t_{12})\,,\label{eqn:tau_312 pullback}\\
\tau_{1(23)}(z,v,w)&=\left(\frac{z-1}{z},-zv,zv+w\right)&&\qquad\implies\qquad\tau_{1(23)}^*\nabla&&&\equiv\nabla(t_{13}\,,t_{12}\,,t_{23})\,.\label{eqn:tau_231 pullback}
\end{alignat}
\end{subequations}
In other words, the permutation right action permutes the indices (according to the inverse of the group element of $\mathrm{S}_3$) of the $t$'s and symmetry of $t$ allows us to rewrite $t_{ji}\equiv t_{ij}$ for $i<j$.
\begin{rem}\label{rem:how to visualise  C_xx_times_C_x}
To visualise $\bbC^{\times\times}\times\bbC^\times$, have in mind the following picture
\begin{figure}[H]
\centering
\scalebox{0.4}{\includesvg[width=1000pt]{C_xx_times_C_x}}
\caption{$\bbC^{\times\times}\times\bbC^\times$}
\label{fig: C_xx_times_C_x}
\end{figure}
Basically, we think of the 2 complex dimensions as 4 real dimensions; $(z:=X+iY,v:=Z+ic)$ where $c\in\bbR$ is a colour signature where we take green as $0$, red as negative and blue as positive. The single puncture of $\bbC^\times$ is depicted by the green plane at $Z=0$ which is meant to indicate that $c=0$ is excluded hence a curve could pass through such a plane so long as it is blue or red upon intersection. The two punctures of $\bbC^{\times\times}$ are depicted by the black lines at $(X=0,Y=0)$ and $(X=1,Y=0)$ which indicate that all colours $c\in\bbR$ are excluded for all $Z\in\bbR$.    
\end{rem}
 
\begin{constr}\label{con:BRW1paths}
We make use of \cite[(2.32)]{BRW}:
\begin{figure}[H]
\centering
\scalebox{0.45}{\includesvg[width=1000pt]{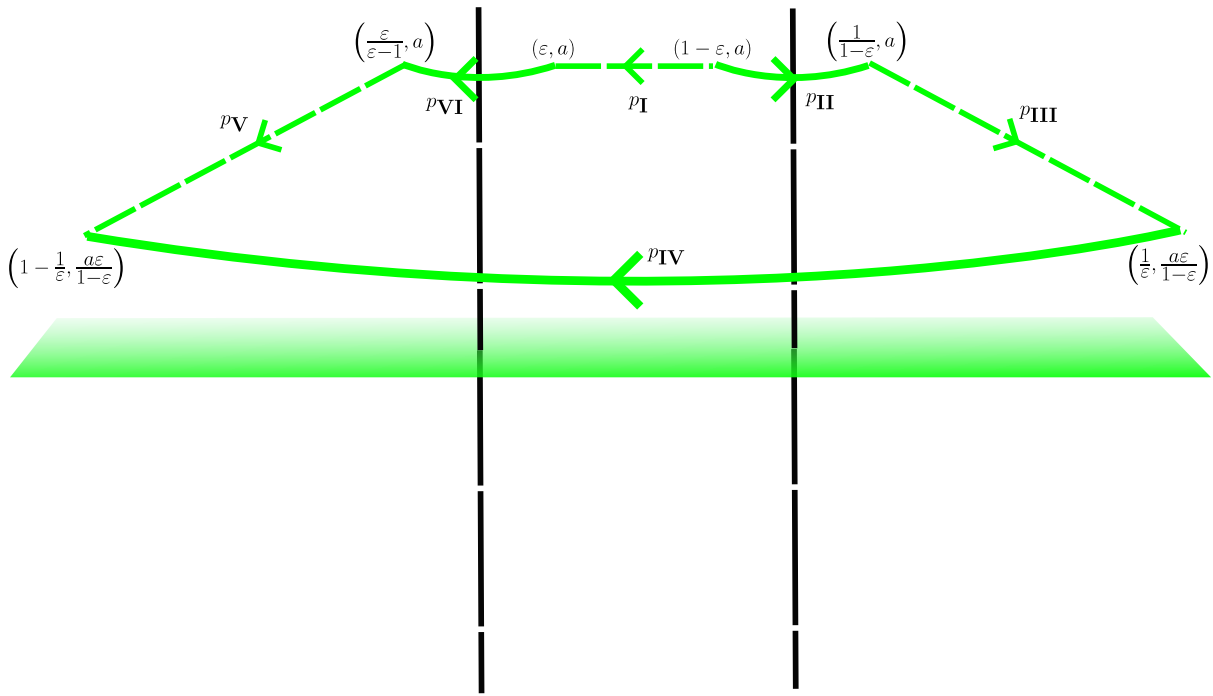}}
\caption{The \textbf{adapted BRW contractible hexagon of 1-paths}}
\label{fig:p-paths}
\end{figure}
\begin{subequations}
\begin{align}
p_\mathbf{I}(r)&:=\big([c_\mathbf{I}\circ\iota](r),a\big)=\big(\varepsilon+(1-r)(1-2\varepsilon),a\big)\\
p_\mathbf{II}(r)&:=\big(c_\mathbf{II}(r),a\big)=\left(\frac{2-\varepsilon-\varepsilon e^{i\pi r}}{2-\varepsilon+\varepsilon e^{i\pi r}},a\right)\\
p_\mathbf{III}(r)&:=\tau_{(12)3}\left(c_\mathbf{I}(r),\frac{a}{\varepsilon-1}\right)=\left(\frac{1}{1-\varepsilon+r(2\varepsilon-1)},\frac{1-\varepsilon+r(2\varepsilon-1)}{1-\varepsilon}a\right)\\
p_\mathbf{IV}(r)&:=\left(c_\mathbf{IV}(r),\frac{a\varepsilon}{1-\varepsilon}\right)=\left(\frac{1}{2}+\left(\frac{1}{\varepsilon}-\frac{1}{2}\right)e^{-i\pi r},\frac{a\varepsilon}{1-\varepsilon}\right)\\
p_\mathbf{V}(r)&:=\tau_{1(23)}\left([c_\mathbf{I}\circ\iota](r),\frac{a}{\varepsilon-1}\right)=\left(\frac{\varepsilon+r(1-2\varepsilon)}{\varepsilon-1+r(1-2\varepsilon)},\frac{1-\varepsilon+r(2\varepsilon-1)}{1-\varepsilon}a\right)\\
p_\mathbf{VI}(r)&:=\big([c_\mathbf{VI}\circ\iota](r),a\big)=\left(\frac{2\varepsilon}{\varepsilon+(2-\varepsilon)e^{i\pi r}},a\right)
\end{align}
\end{subequations}
where: $r\in[0,1]$, $0<\varepsilon\leq\frac{1}{4}$, $\iota$ is the interval inversion \eqref{eq:interval inversion} and $a\in\bbC^\times$ is a free parameter\footnote{It is relevant for building the 2-loop; see Construction \ref{con:the Breen 2-loop} where it will be $\pm1$. For now, just think $a=1$.}.
\end{constr}

\subsection{Vertically-interpolative 2-paths and their explicit 2-holonomies}\label{subsec:interpolative 2-paths}
Given some 2-path $p_\mathbf{V}\,p_\mathbf{VI}\,p_\mathbf{I}\xRightarrow{\P}p_\mathbf{IV}\,p_\mathbf{III}\,p_\mathbf{II}$, the globularity condition \eqref{eq:globularity condition} requires
\begin{equation}\label{eq:W^P as modification}
W^\P:W^{p_\mathbf{V}}W^{p_\mathbf{VI}}W^{p_\mathbf{I}}\Rrightarrow W^{p_\mathbf{IV}}W^{p_\mathbf{III}}W^{p_\mathbf{II}}.
\end{equation}
\begin{rem}\label{rem:explicit parallel transports}
Let us first denote: 
\begin{equation}
\Lambda:\equiv t_{12}+t_{23}+t_{13}\qquad,\qquad\Gamma:\equiv\Gamma(t_{12}\,,t_{23})\equiv\left(\frac{t_{12}}{z}+\frac{t_{23}}{z-1}\right)\dd z\quad.
\end{equation}
We have the following explicit expressions for the parallel transports in \eqref{eq:W^P as modification}:
\begin{subequations}
\begin{equation}
W^{p_\mathbf{I}}\equiv {}^{\nabla}W^{\left(c_\mathbf{I}\circ\iota,a\right)}\equiv {}^{\Gamma}W^{c_\mathbf{I}\circ\iota}\equiv \left({}^{\Gamma}W^{c_\mathbf{I}}\right)^{-1}\equiv \varepsilon^{t_{12}}\Phi^\varepsilon_{321}\varepsilon^{-t_{23}}\quad,
\end{equation}
where the second equality used the fact that $\left(c_\mathbf{I}\circ\iota,a\right)$ is fixed at $v=a$, the third equality used functoriality of parallel transport and the last equality used \eqref{eq:potentially finite Drinfeld} and \eqref{eq:Drinfeld swap is inverse}.
\begin{equation}\label{eq:W^p_II}
W^{p_\mathbf{II}}\equiv {}^{\nabla}W^{(c_\mathbf{II},a)}\equiv {}^{\Gamma}W^{c_\mathbf{II}}\equiv e^{i\pi t_{23}}H_\varepsilon(t_{23}\,,t_{12})\quad,
\end{equation}
where we used \cite[Lemma 21]{BRW} for the last equality.
\begin{equation}\label{eq:3rd partra in rem}
W^{p_\mathbf{III}}\equiv {}^{\tau_{(12)3}^*\nabla}W^{\left(c_\mathbf{I},\frac{a}{\varepsilon-1}\right)}\equiv {}^{\Gamma(t_{23}\,,t_{13})}W^{c_\mathbf{I}}\equiv \varepsilon^{t_{13}}\Phi^\varepsilon_{231}\varepsilon^{-t_{23}}\,,
\end{equation}
where we used \cite[Theorem 9, iv.)]{BRW} for the first equality and \eqref{eqn:tau_312 pullback} for the second equality.
\begin{equation}
W^{p_\mathbf{IV}}\equiv{}^{\Gamma}W^{c_\mathbf{IV}}\equiv e^{i\pi \overline{t_{13}}}H_\varepsilon\big(\overline{t_{13}},t_{23}\big)\quad,
\end{equation}
where $\overline{t_{13}}:\equiv t_{13}-\Lambda$ and we used \cite[(2.34)]{BRW}. 
\begin{equation}\label{eq:p_V parallel transport}
W^{p_\mathbf{V}}\equiv {}^{\nabla}W^{\tau_{1(23)}\left(c_\mathbf{I}\circ\iota,\frac{a}{\varepsilon-1}\right)}\equiv {}^{\Gamma(t_{13}\,,t_{12})}W^{c_\mathbf{I}\circ\iota}\equiv \varepsilon^{t_{13}}\Phi^\varepsilon_{213}\varepsilon^{-t_{12}}\quad,
\end{equation}
where we used \eqref{eqn:tau_231 pullback} for the second equality.
\begin{equation}\label{eq:8 of first partras}
W^{p_\mathbf{VI}}\equiv {}^{\nabla}W^{(c_\mathbf{VI}\circ\iota,a)}\equiv {}^{\Gamma}W^{c_\mathbf{VI}\circ\iota}\equiv \left({}^{\Gamma}W^{c_\mathbf{VI}}\right)^{-1}\equiv \left(H_\varepsilon\left(t_{12}\,,\overline{t_{13}}\right)\right)^{-1}e^{i\pi \overline{t_{(12)3}}}\quad,
\end{equation}
\end{subequations}
where we used \cite[(2.34)]{BRW} again for the fourth equality.
\end{rem}
Now we can explicitly express the LHS and RHS of \eqref{eq:W^P as modification}. For space reasons we write, for example, $H_\varepsilon(t_{23},t_{12})=H_\varepsilon^\mathbf{2}$ in the following,
\begin{subequations}
\begin{align}
\varepsilon^{- t_{13}}W^\P\varepsilon^{ t_{23}}:\Phi^\varepsilon_{213}\left[\widetilde{H}_\varepsilon^\mathbf{6}\right]^{-1}e^{i\pi \overline{t_{(12)3}}}\Phi^\varepsilon_{321}\Rrightarrow \varepsilon^{-\ad_{t_{13}}}\left[e^{i\pi\overline{t_{13}}}\right]\widetilde{H}_\varepsilon^\mathbf{4}\Phi^\varepsilon_{231}e^{i\pi t_{23}}\widetilde{H}_\varepsilon^\mathbf{2}
\end{align}
where the tilde notation refers to that in \cite[Theorem 22]{BRW}. Taking the limit $\varepsilon\rightarrow0$, we get 
\begin{equation}\label{RHS prehex}
\varepsilon^{- t_{13}}W^\P\varepsilon^{ t_{23}}:\,\Phi_{213}e^{i\pi\overline{t_{(12)3}}}\Phi_{321}\Rrightarrow \varepsilon^{-\ad_{t_{13}}}\left[e^{i\pi\overline{t_{13}}}\right]\Phi_{231}e^{i\pi t_{23}}.
\end{equation}
\end{subequations}
Prima facie, \eqref{RHS prehex} looks very similar to \eqref{eq:sym t right pre-hex}. Not until Subsection \ref{sub:finale of hex} will we be able to provide an explicit formula for the right hexagonator series but in this subsection we can already re-derive the infinitesimal right hexagonator by balancing the second order terms of the 2-holonomy against the second order surplus terms of \eqref{RHS prehex}. The concrete 2-path $\P$ we construct is built from three ``vertically-interpolative" 2-paths. 
\subsubsection{The first vertically-interpolative 2-path}\label{subsub:1st 2path}
What we have in mind is given by $p_\mathbf{V}\xRightarrow{\P_\mathbf{V}}p^1_\downarrow(c_\mathbf{V}\circ\iota,a)$ where
\begin{equation}\label{first 2-path}
\P_\mathbf{V}(s,r):=\begin{cases}
      p_\leftarrow(r):=\big([c_\mathbf{V}\circ\iota](2r),a\big)\,, & 0\leq r\leq\frac{s}{2}\\
      p^s_\downarrow(r):=\left([c_\mathbf{V}\circ\iota](s),\frac{1-c_\mathbf{I}(2r-s)}{1-\varepsilon}a\right)\,, & \frac{s}{2}\leq r\leq s\\
      p_\mathbf{V}(r)\,, & s\leq r\leq1
    \end{cases}\qquad.
\end{equation}
Figure \ref{fig:1stInterpolative2path} below provides the reason why we call this a \textbf{vertically-interpolative 2-path}.
\begin{figure}[H]
\centering
\scalebox{0.2}{\includesvg[width=1000pt]{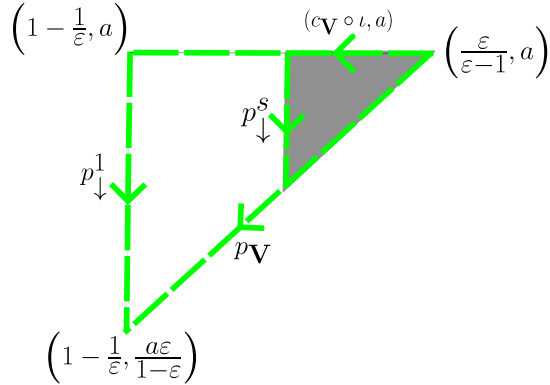}}
\caption{The first vertically-interpolative 2-path}
\label{fig:1stInterpolative2path}
\end{figure}
\begin{rem}
The idea is that, for increasing value of the homotopical parameter $s$, the grey triangle grows and the cross-sectional 1-path $\P_\mathbf{V}^s$ increasingly interpolates from the source 1-path $p_\mathbf{V}$ to the target 1-path $p_\downarrow^1\,(c_\mathbf{V}\circ\iota,a)$ and the ``midsection" of the cross-sectional 1-path is always purely vertical, in this case from $p_\downarrow^s\big(\tfrac{s}{2}\big)$ to $p_\downarrow^s(s)$. This interpolative property is \emph{essential} when it comes to explicitly computing the 2-holonomy (and hence an explicit formula for the right hexagonator series itself) because the integrand of, say, \eqref{1st 2-holonomy} requires knowledge of the parallel transport along the cross-sectional 1-path. In other words, if the cross-sectional 1-path was not built from concatenating 1-paths which we \textit{already} know how to parallel transport along then no explicit computation of the 2-holonomy would be possible.
\end{rem}
The formula for the 2-holonomy \eqref{eq:CM def of 2hol for gl_F} along this specific 2-path $\P_\mathbf{V}$ becomes
\begin{subequations}
\begin{equation}\label{1st 2-holonomy}
W^{\P_\mathbf{V}}=\int_0^1\int_{\frac{s}{2}}^sW_{1r}^{\P_\mathbf{V}^s}\Delta\left[\frac{\partial\P_\mathbf{V}}{\partial s},\frac{\partial\P_\mathbf{V}}{\partial r}\right]W_{r0}^{\P_\mathbf{V}^s}\,\dd r\dd s
\end{equation}
where the first integral is over $\frac{s}{2}\leq r\leq s$ because outside of that range $\frac{\partial\P_\mathbf{V}}{\partial s}=(0,0)$. Conversely, within that range we say $\P_\mathbf{V}^s(r)=p_\downarrow^s(r)=\left(z(s):=[c_\mathbf{V}\circ\iota](s)\,,\,v(r,s):=\frac{1-c_\mathbf{I}(2r-s)}{1-\varepsilon}a\right)$ hence 
\begin{equation}
\Delta\left[\frac{\partial\P_\mathbf{V}}{\partial s},\frac{\partial\P_\mathbf{V}}{\partial r}\right]=\Delta\big[\left(\partial_sz,\partial_sv\right),\left(0,\partial_rv\right)\big]=\frac{1}{v}\left(\frac{\L}{z}+\frac{\R}{z-1}\right)\frac{\partial z}{\partial s}\frac{\partial v}{\partial r}
\end{equation}
and we also have $W_{r0}^{\P_\mathbf{V}^s}\equiv W_{r\frac{s}{2}}^{p_\downarrow^s}W_{\frac{s}{2}0}^{p_\leftarrow}$ and $W_{1r}^{\P_\mathbf{V}^s}\equiv W_{1s}^{p_\mathbf{V}}W_{s\frac{s}{2}}^{p_\downarrow^s}\left(W_{r\frac{s}{2}}^{p_\downarrow^s}\right)^{-1}$ where
\begin{equation}
W_{r\frac{s}{2}}^{p_\downarrow^s}\equiv {}^{\frac{\Lambda}{v}\dd v}W_{(2r-s)0}^{\frac{1-c_\mathbf{I}}{1-\varepsilon}a}\equiv \exp\left(\left[\ln v\right]^{v(r,s)}_{a}\Lambda\right)\equiv e^{\ln\left(\frac{v}{a}\right)\Lambda}\quad.
\end{equation}
These two facts allow us to rewrite the above double integral as
\begin{equation}
W^{\P_\mathbf{V}}=\int_0^1\frac{\partial z}{\partial s}W_{1s}^{p_\mathbf{V}}W_{s\frac{s}{2}}^{p_\downarrow^s}\int_{\frac{s}{2}}^s\frac{1}{v}\frac{\partial v}{\partial r}e^{-\ln\left(\frac{v}{a}\right)\Lambda}\left(\frac{\L}{z}+\frac{\R}{z-1}\right)e^{\ln\left(\frac{v}{a}\right)\Lambda}\dd r\,W_{\frac{s}{2}0}^{p_\leftarrow}\dd s\quad.
\end{equation}
\end{subequations}
We can explicitly solve the inner integral by using the following lemma.
\begin{lem}\label{lem:solve 1st integral}
Given some function $f:\left[0\,,1\right]\to\bbC$ and modification $\M$, we have 
\begin{equation}
\frac{\partial}{\partial r}\left(e^{-f(r)\Lambda}\sum_{k=1}^\infty\frac{\big(f(r)\big)^k}{k!}\sum_{j=0}^{k-1}\Lambda^j\M\Lambda^{k-1-j}\right)=f'(r)e^{-f(r)\Lambda}\M e^{f(r)\Lambda}\quad.
\end{equation}
\end{lem}
\begin{proof}
Simply use the product rule on the LHS.
\end{proof}
We have: $\frac{1}{v}\frac{\partial v}{\partial r}=\frac{\partial\ln\left(\frac{v}{a}\right)}{\partial r}$, $v(\frac{s}{2},s)=a$ and $\frac{v(s,s)}{a}=\frac{1-c_\mathbf{I}(s)}{1-\varepsilon}$ hence applying Lemma \ref{lem:solve 1st integral} gives us
\begin{equation}\label{eqn:pre-explicit 2-hol}
W^{\P_\mathbf{V}}= W^{p_\mathbf{V}}\int_0^1\frac{\partial z}{\partial s}\left(W_{s0}^{p_\mathbf{V}}\right)^{-1}\sum_{k=1}^\infty\frac{\left[\ln\frac{1-c_\mathbf{I}(s)}{1-\varepsilon}\right]^k}{k!}\sum_{j=0}^{k-1}\Lambda^j\left(\frac{\L}{z}+\frac{\R}{z-1}\right)\Lambda^{k-1-j}W_{\frac{s}{2}0}^{p_\leftarrow}\dd s\,.
\end{equation}
To integrate (by parts, iteratively) out the parallel transports, we must find a common horizontal 1-path:
\begin{subequations}\label{subeq:1st common hor 1-path}
\begin{equation}\label{eqn:partra along p_leftarrow}
W_{\frac{s}{2}0}^{p_\leftarrow}\equiv\left({}^{\Gamma}W_{s0}^{c_\mathbf{V}}\right)^{-1}\equiv{}^{\Gamma\left(t_{12}\,,\overline{t_{13}}\right)}W_{s0}^{c_\mathbf{I}}\quad,
\end{equation}
where the first equality comes from the definition of $p_\leftarrow$ in \eqref{first 2-path} and the second equality comes from \cite[third line of (2.33)]{BRW} together with \eqref{eq:Drinfeld swap is inverse}. We also have
\begin{equation}\label{eqn:partra along p_V}
\left(W_{s0}^{p_\mathbf{V}}\right)^{-1}\equiv \left({}^{\Gamma(t_{12}\,,t_{13})}W_{s0}^{c_\mathbf{I}}\right)^{-1}
\end{equation}
\end{subequations}
which simply comes from \eqref{eq:p_V parallel transport}.

\begin{propo}\label{propo:solved 1st 2hol}
Recall the notation for iterated integrals from Remark \ref{rem:iterated integral defined inductively} and the particular 1-forms $\Omega_0:=\frac{ds}{s}$ and $\Omega_1:=\frac{ds}{s-1}$. The explicit solution for \eqref{eqn:pre-explicit 2-hol} is given by
\begin{equation}\label{eq:W^(P_V)=}
W^{\P_\mathbf{V}}=f_\mathbf{V}(1)W^{c_\mathbf{V}\circ\iota}
\end{equation} 
where the function $f_\mathbf{V}(s)$ is defined as
\begin{equation}
\sum_{\begin{smallmatrix}r=0\\k=1\end{smallmatrix}}^\infty\,\sum_{i_1,\ldots,i_r=0}^1\,\sum_{j=0}^{k-1}\int_\varepsilon^{c_\mathbf{I}(s)}\prod_{m=1}^r\Omega_{i_m}\ad_{t_{12}}^{1-i_m}(\ad_{\overline{t_{13}}}+\Lambda)^{i_m}\Lambda^j\left[\Omega_0\Omega_1^k\L-\Omega_1^{k+1}(\L+\R)\right]\Lambda^{k-1-j}.
\end{equation}
\end{propo}
\begin{proof}
The origin of $f_\mathbf{V}$ is as the solution (with $f_\mathbf{V}(0)=0$) of the following $1^\mathrm{st}$-order ODE, 
\begin{subequations}
\begin{equation}
\frac{\dd }{\dd s}\left[W_{s0}^{p_\mathbf{V}\circ\iota}f_\mathbf{V}\,W_{\frac{s}{2}0}^{p_\leftarrow}\right]=\frac{\partial z}{\partial s}W_{s0}^{p_\mathbf{V}\circ\iota}\sum_{k=1}^\infty\frac{\left[\ln\frac{1-c_\mathbf{I}}{1-\varepsilon}\right]^k}{k!}\sum_{j=0}^{k-1}\Lambda^j\left[\frac{\L}{z}+\frac{\R}{z-1}\right]\Lambda^{k-1-j}W_{\frac{s}{2}0}^{p_\leftarrow}.
\end{equation}
Indeed, using \eqref{subeq:1st common hor 1-path} and the product rule on the LHS gives us\footnote{For parallel transport ${}^{A}W_{s0}^{p}$, we have $\frac{\dd }{\dd s}{}^{\A}W_{s0}^{p}\equiv\A_p\,{}^{\A}W_{s0}^{p}$ hence $\frac{\dd }{\dd s}\left({}^{\A}W_{s0}^{p}\right)^{-1}\equiv -\left({}^{\A}W_{s0}^{p}\right)^{-1}\A_p$.}
\begin{equation}
W_{s0}^{p_\mathbf{V}\circ\iota}\left(f'_\Leftarrow+f_\mathbf{V}\left(\frac{t_{12}}{c_\mathbf{I}}+\frac{\overline{t_{13}}}{c_\mathbf{I}-1}\right)\frac{\dd c_\mathbf{I}}{\dd s}-\left(\frac{t_{12}}{c_\mathbf{I}}+\frac{t_{13}}{c_\mathbf{I}-1}\right)\frac{\dd c_\mathbf{I}}{\dd s}f_\mathbf{V}\right)W_{\frac{s}{2}0}^{p_\leftarrow}\quad.
\end{equation}
Cancelling $W_{s0}^{p_\mathbf{V}\circ\iota}$ and $W_{\frac{s}{2}0}^{p_\leftarrow}$ on both sides then integrating from 0 to $s$ and using $f_\mathbf{V}(0)=0$ yields
\begin{footnotesize}
\begin{equation}
f_\mathbf{V}-\int_0^s\left[\frac{\ad_{t_{12}}}{c_\mathbf{I}}+\frac{\ad_{\overline{t_{13}}}+\Lambda}{c_\mathbf{I}-1}\right]f_\mathbf{V}\frac{\dd c_\mathbf{I}}{\dd s'}\dd s'=\int_0^s\frac{\partial z}{\partial s'}\sum_{k=1}^\infty\frac{\left[\ln\frac{1-c_\mathbf{I}}{1-\varepsilon}\right]^k}{k!}\sum_{j=0}^{k-1}\Lambda^j\left[\frac{\L}{z}+\frac{\R}{z-1}\right]\Lambda^{k-1-j}\dd s'\,.
\end{equation}
\end{footnotesize}
For the RHS, let us simply recall that $z:=c_\mathbf{V}(1-s)=1-\frac{1}{1-c_\mathbf{I}(s)}$ which gives:
\begin{equation}\label{eq:z=c_V(1-s) implies 1/z dz/ds =.. and 1/z-1 dz/ds=..}
\frac{1}{z}\frac{\partial z}{\partial s}=\left(\frac{1}{c_\mathbf{I}(s)}+\frac{1}{1-c_\mathbf{I}(s)}\right)\frac{\partial c_\mathbf{I}}{\partial s}\qquad,\qquad
\frac{1}{z-1}\frac{\partial z}{\partial s}=\frac{1}{1-c_\mathbf{I}(s)}\frac{\partial c_\mathbf{I}}{\partial s}\quad,
\end{equation}
and now \eqref{eq:int_a^b Omega_1^k} gives us:
\begin{alignat}{2}
\int_0^s\frac{\left[\ln\left(\frac{1-c_\mathbf{I}(s')}{1-\varepsilon}\right)\right]^k}{k!(z-1)}\frac{\partial z}{\partial s'}\dd s'&=-\int_\varepsilon^{c_\mathbf{I}(s)}\Omega_1^{k+1}\quad&&,\\
\int_0^s\frac{\left[\ln\left(\frac{1-c_\mathbf{I}(s')}{1-\varepsilon}\right)\right]^k}{k!z}\frac{\partial z}{\partial s'}\dd s'&=\int_\varepsilon^{c_\mathbf{I}(s)}\Omega_0\Omega_1^k-\int_\varepsilon^{c_\mathbf{I}(s)}\Omega_1^{k+1}\quad&&.
\end{alignat}
The RHS now reads as
\begin{equation}
\sum_{k=1}^\infty\sum_{j=0}^{k-1}\Lambda^j\left[\int_\varepsilon^{c_\mathbf{I}(s)}\Omega_0\Omega_1^k\L-\int_\varepsilon^{c_\mathbf{I}(s)}\Omega_1^{k+1}(\L+\R)\right]\Lambda^{k-1-j}\quad.
\end{equation}
As for the LHS, we can rewrite this as
\begin{equation}
f_\mathbf{V}(s)-\int_0^s\Gamma_{c_\mathbf{I}}\left(\ad_{t_{12}},\ad_{\overline{t_{13}}}+\Lambda\right)f_\mathbf{V}(s')=\left(\Id-I_{s0}\circ L_{\Gamma_{c_\mathbf{I}}\left(\ad_{t_{12}},\ad_{\overline{t_{13}}}+\Lambda\right)}\right)f_\mathbf{V}\quad.
\end{equation}
\end{subequations}
As in \eqref{eq:super compact expression for W^(c_I)}, the (implicit) factor of $\hbar$ allows us to invert this operator to get an iterated integral series.
\end{proof}

\begin{cor}\label{cor:h^2 1st 2hol}
The lowest order term of both $\lim_{\varepsilon\to0} f_\mathbf{V}(1)$ and $\lim_{\varepsilon\to0}W^{\P_\mathbf{V}}$ is given by
\begin{equation}\label{eqn:lowest term 1st 2hol}
\left(W^{\P_\mathbf{V}}\right)_2\rightarrow\,-\frac{\pi^2}{6}\L-\frac{(\ln\varepsilon)^2}{2}(\L+\R)\quad.
\end{equation}
\end{cor}
\begin{proof}
Simply use \eqref{eq:int_a^b Omega_1^k} and \eqref{eq:int_0^1 Omega_0 Omega_1=-pi^2/6}
\end{proof}
\subsubsection{Deriving the 2-holonomy through the globularity condition}\label{subsub:derivation via glob cond of 1st 2hol}
We can indirectly determine the limit $\lim_{\varepsilon\to0}\left[f_\mathbf{V}(1)W^{c_\mathbf{V}\circ\iota}\right]$ by considering the globularity condition \eqref{eq:globularity condition} and comparing the limit of the parallel transports in $W^{\P_\mathbf{V}}:W^{p_\mathbf{V}}\Rrightarrow W^{p^1_\downarrow}\left(W^{c_\mathbf{V}}\right)^{-1}$, i.e., 
\begin{equation}\label{eq:glob of 1st 2hol}
\lim_{\varepsilon\to0}(W^{\P_\mathbf{V}}):\varepsilon^{t_{13}}\Phi_{213}\varepsilon^{-t_{12}}\Rrightarrow\varepsilon^\Lambda\varepsilon^{\overline{t_{13}}}\Phi(t_{12},\overline{t_{13}})\varepsilon^{-t_{12}}\quad.
\end{equation}
There are two modifications that contribute to \eqref{eq:glob of 1st 2hol}: one from $\varepsilon^{t_{13}}$ to $\varepsilon^\Lambda\varepsilon^{\overline{t_{13}}}$ and another from $\Phi_{213}$ to $\Phi(t_{12},\overline{t_{13}})$. In order to explicitly express each of these modifications, we will need the following lemma.
\begin{lem}
Given $n\in\bbN$ and elements $A$ and $B$ of a unital associative algebra, we have:
\begin{subequations}
\begin{alignat}{6}
(A+B)^n=&\,A^n+\sum_{m=0}^{n-1}(A+B)^{n-1-m}BA^m\quad&&,\label{eq:(A+B)^n easy}\\
=&\sum_{p=0}^n\binom{n}{p}A^{n-p}B^p+\sum_{j=1}^{n-1}\sum_{k=0}^{n-1-j}\binom{n-j}{k}(A+B)^{j-1}[B,A^{n-j-k}]B^k\quad&&.\label{eq:(A+B)^n hard}
\end{alignat}
\end{subequations}
\end{lem}
\begin{proof}
Clearly \eqref{eq:(A+B)^n easy} is true for $n=0$ hence let us suppose it is true for $n\in\bbN^*$ then prove the inductive step by computing $(A+B)^{n+1}=(A+B)(A+B)^n$ as
\begin{subequations}
\begin{equation}
(A+B)\left[A^n+\sum_{m=0}^{n-1}(A+B)^{n-1-m}BA^m\right]=A^{n+1}+BA^n+\sum_{m=0}^{n-1}(A+B)^{n-m}BA^m.
\end{equation}
The second equation \eqref{eq:(A+B)^n hard} is more difficult to prove, we start by expanding the LHS as 
\begin{equation}\label{eq:noncommutative binomial expansion}
(A+B)^n=\sum_{i_1,\ldots,i_n=0}^1A^{1-i_1}B^{i_1}\cdots A^{1-i_n}B^{i_n}\quad.
\end{equation}
The summand can be rewritten as 
\begin{equation}
A^{n-i_1-\ldots-i_n}B^{i_1+\ldots+i_n}+\sum_{j=1}^{n-1}A^{1-i_1}B^{i_1}\cdots A^{1-i_j}[B^{i_j},A^{n-j-i_{j+1}-\ldots-i_n}]B^{i_{j+1}+\ldots+i_n}\quad.
\end{equation}
Imposing the summation constraints on the first term gives 
\begin{equation}
\sum_{i_1,\ldots,i_n=0}^1A^{n-i_1-\ldots-i_n}B^{i_1+\ldots+i_n}=\sum_{p=0}^n\binom{n}{p}A^{n-p}B^p
\end{equation}
whereas, for the second term to be nontrivial, we need $i_j\neq0$ thus we have
\begin{align}
&\sum_{j=1}^{n-1}\sum_{i_1,\ldots,i_{j-1}=0}^1A^{1-i_1}B^{i_1}\cdots A^{1-i_{j-1}}B^{i_{j-1}}\sum_{i_{j+1},\ldots,i_n=0}^1[B,A^{n-j-i_{j+1}-\ldots-i_n}]B^{i_{j+1}+\ldots+i_n}\nn\\&=\sum_{j=1}^{n-1}(A+B)^{j-1}\sum_{k=0}^{n-j}\binom{n-j}{k}[B,A^{n-j-k}]B^k\quad.
\end{align}
\end{subequations}
$k=n-j\implies[B,A^{n-j-k}]=0$ so $\sum_{k=0}^{n-j}\binom{n-j}{k}[B,A^{n-j-k}]B^k=\sum_{k=0}^{n-j-1}\binom{n-j}{k}[B,A^{n-j-k}]B^k$.
\end{proof}
\begin{cor}\label{cor:BCH-type mod}
For a strict infinitesimal 2-braiding $t$, we have 
\begin{equation}\label{eq:corollary of binom expansion}
\sum_{n=2}^\infty\frac{(\ln\varepsilon)^n}{n!}\sum_{j=1}^{n-1}\sum_{k=0}^{n-j-1}\sum_{l=1}^{n-j-k}\binom{n-j}{k}t_{13}^{j-1}\Lambda^{l-1}(\L+\R)\Lambda^{n-j-k-l}\overline{t_{13}}^k:\varepsilon^\Lambda\varepsilon^{\overline{t_{13}}}\Rrightarrow\varepsilon^{t_{13}}.
\end{equation}
\end{cor}
\begin{proof}
We use \eqref{eq:(A+B)^n hard} as follows,
\begin{footnotesize}
\begin{align}
\varepsilon^{t_{13}}=\sum_{n=0}^\infty\frac{(\ln\varepsilon)^n}{n!}[\Lambda+\overline{t_{13}}]^n=\sum_{n=0}^\infty\frac{(\ln\varepsilon)^n}{n!}\left[\sum_{p=0}^n\binom{n}{p}\Lambda^{n-p}\overline{t_{13}}^p+\sum_{j=1}^{n-1}\sum_{k=0}^{n-1-j}\binom{n-j}{k}t_{13}^{j-1}[\overline{t_{13}},\Lambda^{n-j-k}]\overline{t_{13}}^k\right]\,.
\end{align}
\end{footnotesize}
The first term gives $\sum_{n=0}^\infty\frac{(\ln\varepsilon)^n}{n!}\sum_{p=0}^n\binom{n}{p}\Lambda^{n-p}\overline{t_{13}}^p=\sum_{q=0}^\infty\frac{(\ln\varepsilon)^q}{q!}\Lambda^q\sum_{p=0}^\infty\frac{(\ln\varepsilon)^p}{p!}\overline{t_{13}}^p=\varepsilon^\Lambda\varepsilon^{\overline{t_{13}}}$. For the second term, to have $1\leq n-1$ we need $2\leq n$ and we expand the commutator bracket as 
\begin{equation}
[\overline{t_{13}},\Lambda^{n-j-k}]=\sum_{l=1}^{n-j-k}\Lambda^{l-1}[\overline{t_{13}},\Lambda]\Lambda^{n-j-k-l}=\sum_{l=1}^{n-j-k}\Lambda^{l-1}[t_{13},t_{12}+t_{23}]\Lambda^{n-j-k-l}\,.
\end{equation}
Finally, \cite[Lemma 5.6]{Me} implies that, for a strict $t$, we have $\L+\R:[t_{12}+t_{23},t_{13}]\Rrightarrow0$.
\end{proof}
\begin{rem}\label{rem:LM Drinfeld series}
Recall from \cite[Proposition XIX.6.4]{Kassel} and \cite[Theorem A.9]{LeMu} that we can explicitly express Drinfeld's KZ associator series $\Phi_\KZ(A,B)$ in terms of the MZVs \eqref{eq:iterated integral expression for multiple zeta} as follows. Given a non-empty tuple $\p=(p_1,\ldots,p_{\tilde{\p}})$ of natural numbers, $\p>0$ means that every entry is strictly positive, i.e., $p_1,\ldots,p_{\tilde{\p}}>0$. Given another tuple of natural numbers $\j$, by $0\leq\j\leq\p$ we mean that $\j=(j_1,\ldots,j_{\tilde{\p}})$ and $0\leq j_i\leq p_i$ for all $1\leq i\leq\tilde{\p}$. We define $|\p|:=\sum_{l=1}^{\tilde{\p}}p_l$ and, for $\q>0$ such that $\tilde{\q}=\tilde{\p}$,  
\begin{equation}\label{eq:compact zeta notation}
\zeta_\j^{\p,\q}:=(-1)^{|\j|+|\p|}\zeta(p_1+1,\{1\}^{q_1-1},\ldots,p_{\tilde{\p}}+1,\{1\}^{q_{\tilde{\p}}-1})\prod_{l=1}^{\tilde{\p}}\binom{p_l}{j_l}\quad.
\end{equation} 
We now have the compact expression for Drinfeld's KZ associator series,
\begin{equation}\label{eq:LM expression for Phi}
\Phi=1+\sum_{\{\p,\q>0\,|\,\tilde{\p}=\tilde{\q}\}}\sum_{\begin{smallmatrix}0\leq\j\leq\p\\0\leq\k\leq\q\end{smallmatrix}}\zeta_\j^{\p,\q}\left[\prod_{l=1}^{\tilde{\p}}(-1)^{k_l}\binom{q_l}{k_l}\right]B^{|\q|-|\k|}A^{j_1}B^{k_1}\cdots A^{j_{\tilde{\p}}}B^{k_{\tilde{\p}}}A^{|\p|-|\j|}\,.
\end{equation}
\end{rem}
We can carry out the $\k$-sum appearing in \eqref{eq:LM expression for Phi} by using the following lemma.
\begin{lem}\label{lem:ad^q}
For $q\in\bbN$,
\begin{equation}\label{eq:ad^q_B(A)}
\ad^q_B(A)=\sum_{k=0}^q\binom{q}{k}(-1)^kB^{q-k}AB^k\quad.
\end{equation}
\end{lem}
\begin{proof}
It is clearly true for $q=0$ hence assume it is true for $q\in\bbN^*$ and compute the inductive step $\ad^{q+1}_B(A)=B\ad^q_B(A)-\ad^q_B(A)B$ as
\begin{align}
&\sum_{k=0}^q\binom{q}{k}(-1)^kB^{q+1-k}AB^k-\sum_{k=0}^q\binom{q}{k}(-1)^kB^{q-k}AB^{k+1}\nn\\&=B^{q+1}A+\sum_{k=1}^q\left[\binom{q}{k}+\binom{q}{k-1}\right](-1)^kB^{q+1-k}AB^k+(-1)^{q+1}AB^{q+1}\label{eq:ad^q inductive step}
\end{align}
where we simply isolated $k=0$ of the first sum and $k=q$ of the second sum, then reindexed the remainder of the second sum. However, \eqref{eq:ad^q inductive step} is equal to $\sum_{k=0}^{q+1}\binom{q+1}{k}(-1)^kB^{q+1-k}AB^k$ because, for $1\leq k\leq q$, we have $\binom{q+1}{k}=\binom{q}{k-1}+\binom{q}{k}$.
\end{proof}
\begin{cor}
Drinfeld's KZ associator series \eqref{eq:LM expression for Phi} can be rewritten as
\begin{equation}\label{eq:Phi with ad form}
\Phi=1+\sum_{\{\p,\q>0\,|\,\tilde{\p}=\tilde{\q}\}}\sum_{0\leq\j\leq\p}\zeta_\j^{\p,\q}\ad^{q_{\tilde{\p}}}_B\Big(\ad_B^{q_{{\tilde{\p}}-1}}\Big[\cdots\left(\ad^{q_2}_B\big[\ad^{q_1}_B(A^{j_1})A^{j_2}\big]A^{j_3}\right)\cdots\Big]A^{j_{\tilde{\p}}}\Big)A^{|\p|-|\j|}.
\end{equation}
\end{cor}
\begin{constr}\label{con:mod from Phi_213}
For $\Phi_{213}=\Phi(t_{12},t_{13})$, we rewrite (part of) the summand of \eqref{eq:Phi with ad form} as
\begin{subequations}
\begin{equation}\label{eq:ad R...ad R(1)}
\ad^{q_{\tilde{\p}}}_{t_{13}}\left(\ad_{t_{13}}^{q_{{\tilde{\p}}-1}}\left[\cdots\big(\ad^{q_2}_{t_{13}}\big[\ad^{q_1}_{t_{13}}(t_{12}^{j_1})t_{12}^{j_2}\big]t_{12}^{j_3}\big)\cdots \right]t_{12}^{j_{\tilde{\p}}}\right)=\ad^{q_{\tilde{\p}}}_{t_{13}}R_{t_{12}}^{j_{\tilde{\p}}}\cdots\ad^{q_1}_{t_{13}}R_{t_{12}}^{j_1}(1)
\end{equation}
where $R_A$ simply means right multiplication by $A$. We use \eqref{eq:(A+B)^n easy} twice: first to write 
\begin{equation}
\ad_{t_{13}}^q=\ad_{\overline{t_{13}}}^q+\sum_{m=0}^{q-1}\ad_{t_{13}}^{q-1-m}\ad_\Lambda\ad_{\overline{t_{13}}}^m
\end{equation}
and secondly to rewrite the string of operators in the RHS of \eqref{eq:ad R...ad R(1)} as
\begin{equation}\label{eq:expanded summand for phi_213}
\prod_{l=0}^{\tilde{\p}-1}\ad^{q_{\tilde{\p}-l}}_{\overline{t_{13}}}R_{t_{12}}^{j_{\tilde{\p}-l}}+\sum_{l=1}^{\tilde{\p}}\prod_{g=0}^{\tilde{\p}-l-1}\ad^{q_{\tilde{\p}-g}}_{t_{13}}R_{t_{12}}^{j_{\tilde{\p}-g}}\sum_{m=0}^{q_l-1}\ad_{t_{13}}^{q_l-m-1}\ad_\Lambda\ad_{\overline{t_{13}}}^mR_{t_{12}}^{j_l}\prod_{g=\tilde{\p}-l+1}^{\tilde{\p}-1}\ad^{q_{\tilde{\p}-g}}_{\overline{t_{13}}}R_{t_{12}}^{j_{\tilde{\p}-g}}.
\end{equation}
Clearly the first term of \eqref{eq:expanded summand for phi_213} gives us $\Phi(t_{12},\overline{t_{13}})$ whereas the second term produces the modification that we are looking for. Specifically, we want to consider: for $0\leq k_l\leq m$, $j_0:=0$, and $k_0:=m-k_l+\sum_{i=1}^{l-1}q_i-k_i$, the commutator bracket $\left[\Lambda,\prod_{g=0}^l t_{12}^{j_g}\overline{t_{13}}^{k_g}\right]$ as 
\begin{align}
\sum_{i=0}^l\prod_{g=0}^{i-1}t_{12}^{j_g}\overline{t_{13}}^{k_g}\left[\sum_{n=1}^{j_i}t_{12}^{n-1}[\Lambda,t_{12}]t_{12}^{j_i-n}\overline{t_{13}}^{k_i}+t_{12}^{j_i}\sum_{n=1}^{k_i}\overline{t_{13}}^{n-1}[\Lambda,t_{13}]\overline{t_{13}}^{k_i-n}\right]\prod_{g=i+1}^lt_{12}^{j_g}\overline{t_{13}}^{k_g}.
\end{align}
We thus have the following modification from $\Phi_{213}$ to $\Phi(t_{12},\overline{t_{13}})$, 
\begin{align}
\nn&\sum_{\{\p,\q>0\,|\,\tilde{\p}=\tilde{\q}\}}\sum_{\begin{smallmatrix}0\leq\j\leq\p\\1\leq l\leq\tilde{\p} 
\end{smallmatrix}}\zeta_\j^{\p,\q}\prod_{g=0}^{\tilde{\p}-l-1}\ad^{q_{\tilde{\p}-g}}_{t_{13}}R_{t_{12}}^{j_{\tilde{\p}-g}}\sum_{m=0}^{q_l-1}\ad_{t_{13}}^{q_l-m-1}\Big(\prod_{r=1}^{l-1}\sum_{k_r=0}^{q_r}(-1)^{k_r}\binom{q_r}{k_r}\sum_{k_l=0}^m(-1)^{k_l}\binom{m}{k_l}\sum_{i=0}^l\\&\prod_{g=0}^{i-1}t_{12}^{j_g}\overline{t_{13}}^{k_g}\left[t_{12}^{j_i}\sum_{n=1}^{k_i}\overline{t_{13}}^{n-1}(\L+\R)\overline{t_{13}}^{k_i-n}-\sum_{n=1}^{j_i}t_{12}^{n-1}\L t_{12}^{j_i-n}\overline{t_{13}}^{k_i}\right]\prod_{g=i+1}^lt_{12}^{j_g}\overline{t_{13}}^{k_g}\Big)t_{12}^{|\p|-|\j|}\label{eq:mod from phi_213}
\end{align}
\end{subequations}
\end{constr}
We denote the modifications \eqref{eq:corollary of binom expansion} and \eqref{eq:mod from phi_213} as $\int\{\varepsilon^\Lambda\varepsilon^{\overline{t_{13}}}-\varepsilon^{t_{13}}\}$ and $\int\{\Phi_{213}-\Phi(t_{12},\overline{t_{13}})\}$, respectively. We thus have
\begin{equation}
\lim_{\varepsilon\to0}(W^{\P_\mathbf{V}})=\varepsilon^{t_{13}}\int\{\Phi_{213}-\Phi(t_{12},\overline{t_{13}})\}\varepsilon^{-t_{12}}-\int\{\varepsilon^\Lambda\varepsilon^{\overline{t_{13}}}-\varepsilon^{t_{13}}\}\Phi(t_{12},\overline{t_{13}})\varepsilon^{-t_{12}}
\end{equation}

\subsubsection*{The second vertically-interpolative 2-path}
\begin{equation}\label{eqn:second 2-path}
\P_\mathbf{III}(s,r):=\begin{cases}
      p_\rightarrow(r):=\left(c_\mathbf{III}(2r),a\right), & 0\leq r\leq\frac{s}{2}\\
      p^\downarrow_s(r):=\left(c_\mathbf{III}(s),\frac{1-c_\mathbf{I}(2r-s)}{1-\varepsilon}a\right), & \frac{s}{2}\leq r\leq s\\
      p_\mathbf{III}(r), & s\leq r\leq1
    \end{cases}\,:p_\mathbf{III}\Rightarrow p^\downarrow_1\left(c_\mathbf{III},a\right)\quad.
\end{equation}
\begin{figure}[H]
\centering
\scalebox{0.2}{\includesvg[width=1000pt]{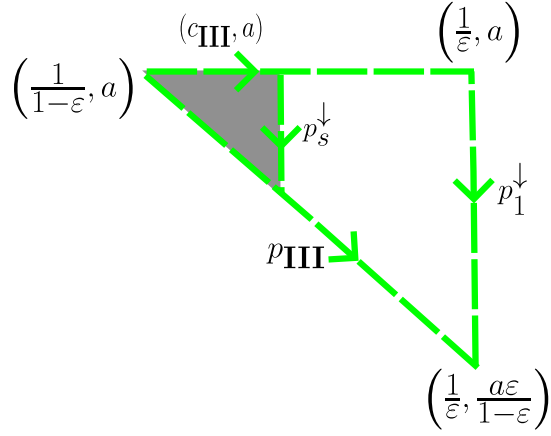}}
\caption{The second vertically-interpolative 2-path}
\label{fig:2ndInterpolative}
\end{figure}
We have $W_{s0}^{c_\mathbf{III}}\equiv{}^{\Gamma(t_{23},\overline{t_{13}})}W_{s0}^{c_\mathbf{I}}$ and $(W_{s0}^{p_\mathbf{III}})^{-1}\equiv\big({}^{\Gamma(t_{23},t_{13})}W_{s0}^{c_\mathbf{I}}\big)^{-1}$ from \cite[second line of (2.33)]{BRW} and \eqref{eq:3rd partra in rem}, respectively. Also $z=c_\mathbf{III}(s)=\frac{1}{1-c_\mathbf{I}(s)}$ gives:
\begin{equation}
\frac{1}{z}\frac{\partial z}{\partial s}=\frac{1}{1-c_\mathbf{I}(s)}\frac{\partial c_\mathbf{I}}{\partial s}\qquad,\qquad\frac{1}{z-1}\frac{\partial z}{\partial s}=\left(\frac{1}{c_\mathbf{I}(s)}+\frac{1}{1-c_\mathbf{I}(s)}\right)\frac{\partial c_\mathbf{I}}{\partial s}\quad.
\end{equation}
Using the same arguments as developed in Subsubsection \ref{subsub:1st 2path}, the 2-holonomy is thus given by $W^{\P_\mathbf{III}}= f_\mathbf{III}(1)W^{c_\mathbf{III}}$ where the function $f_\mathbf{III}(s)$ is defined as
\begin{equation}\label{eq:f_III}
\sum_{\begin{smallmatrix}r=0\\k=1\end{smallmatrix}}^\infty\,\sum_{i_1,\ldots,i_r=0}^1\,\sum_{j=0}^{k-1}\int_\varepsilon^{c_\mathbf{I}(s)}\prod_{m=1}^r\Omega_{i_m}\ad_{t_{23}}^{1-i_m}(\ad_{\overline{t_{13}}}+\Lambda)^{i_m}\Lambda^j\left[\Omega_0\Omega_1^k\R-\Omega_1^{k+1}(\L+\R)\right]\Lambda^{k-1-j}.
\end{equation}

\begin{cor}\label{cor:h^2 2nd 2hol}
The lowest order term of $\lim_{\varepsilon\to0}W^{\P_\mathbf{III}}$ is given by
\begin{equation}\label{eqn:lowest term 2nd 2hol}
\left(W^{\P_\mathbf{III}}\right)_2\rightarrow\,-\frac{\pi^2}{6}\R-\frac{(\ln\varepsilon)^2}{2}(\L+\R)\quad.
\end{equation}
\end{cor}
\begin{rem}
As in Subsubsection \ref{subsub:derivation via glob cond of 1st 2hol}, we can indirectly determine the limit of the 2-holonomy via the globularity condition;
\begin{equation}\label{eq:glob of 2nd 2hol}
\lim_{\varepsilon\to0}(W^{\P_\mathbf{III}}):\varepsilon^{t_{13}}\Phi_{231}\varepsilon^{-t_{23}}\Rrightarrow\varepsilon^\Lambda\varepsilon^{\overline{t_{13}}}\Phi(t_{23},\overline{t_{13}})\varepsilon^{-t_{23}}\quad.
\end{equation}
In order to compute the modification from $\Phi_{231}$ to $\Phi(t_{23},\overline{t_{13}})$, we need only apply the permutation $(123)\to(321)$ to Construction \ref{con:mod from Phi_213},
\begin{footnotesize}
\begin{align}
\nn&\sum_{\{\p,\q>0\,|\,\tilde{\p}=\tilde{\q}\}}\sum_{\begin{smallmatrix}0\leq\j\leq\p\\1\leq l\leq\tilde{\p} 
\end{smallmatrix}}\zeta_\j^{\p,\q}\prod_{g=0}^{\tilde{\p}-l-1}\ad^{q_{\tilde{\p}-g}}_{t_{13}}R_{t_{23}}^{j_{\tilde{\p}-g}}\sum_{m=0}^{q_l-1}\ad_{t_{13}}^{q_l-m-1}\Big(\prod_{r=1}^{l-1}\sum_{k_r=0}^{q_r}(-1)^{k_r}\binom{q_r}{k_r}\sum_{k_l=0}^m(-1)^{k_l}\binom{m}{k_l}\sum_{i=0}^l\\&\prod_{g=0}^{i-1}t_{23}^{j_g}\overline{t_{13}}^{k_g}\left[t_{23}^{j_i}\sum_{n=1}^{k_i}\overline{t_{13}}^{n-1}(\L+\R)\overline{t_{13}}^{k_i-n}-\sum_{n=1}^{j_i}t_{23}^{n-1}\R t_{23}^{j_i-n}\overline{t_{13}}^{k_i}\right]\prod_{g=i+1}^lt_{23}^{j_g}\overline{t_{13}}^{k_g}\Big)t_{23}^{|\p|-|\j|}\label{eq:mod from phi_231}
\end{align}
\end{footnotesize}
\end{rem}

\subsubsection*{The third vertically-interpolative 2-path}
\begin{equation}\label{eq:1st horizontal interpolative}
\P_\mathbf{IV}(s,r):=\begin{cases}
      \big(c_\mathbf{IV}(2r),a\big)\,, & 0\leq r\leq\frac{s}{2}\\
      \left(c_\mathbf{IV}(s),\frac{1-c_\mathbf{I}(2r-s)}{1-\varepsilon}a\right)\,, & \frac{s}{2}\leq r\leq\frac{1+s}{2}\\
      p_\mathbf{IV}(2r-1)\,, & \frac{1+s}{2}\leq r\leq1
    \end{cases}\,:p_\mathbf{IV}p^\downarrow_1\Rightarrow p_\downarrow^1(c_\mathbf{IV},a)\quad.
\end{equation}
\begin{figure}[H]
\centering
\scalebox{0.4}{\includesvg[width=1000pt]{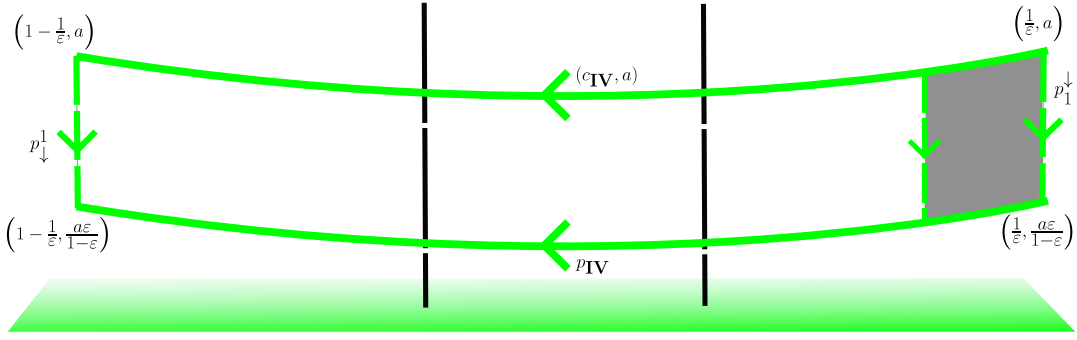}}
\caption{The third vertically-interpolative 2-path}
\label{fig:hor_interpol}
\end{figure}
This time round, we \textit{first} indirectly compute the limit of the 2-holonomy through the globularity condition then show that we can recover the exact series formula by directly calculating the limit of the corresponding iterated integral series formula. 
\begin{rem}\label{rem:1st horinterpol glob condition}
Under $\lim_{\varepsilon\to0}$, the globularity condition \eqref{eq:globularity condition} requires 
\begin{subequations}
\begin{equation}\label{eq:glob cond lim of horinterpol}
\lim_{\varepsilon\to0}(W^{\P_\mathbf{IV}}):e^{i\pi\overline{t_{13}}}\varepsilon^\Lambda\Rrightarrow\varepsilon^\Lambda e^{i\pi\overline{t_{13}}}
\end{equation}
hence we compute
\begin{align}
\left[e^{i\pi\overline{t_{13}}},\varepsilon^\Lambda\right]&=\sum_{m=1}^\infty\sum_{k=1}^\infty\frac{(i\pi)^m(\ln\varepsilon)^k}{m!k!}\big[\overline{t_{13}}^m,\Lambda^k\big]\label{eq:1st determination}\\
&=\sum_{m=1}^\infty\sum_{k=1}^\infty\frac{(i\pi)^m(\ln\varepsilon)^k}{m!k!}\sum_{l=0}^{m-1}\sum_{j=0}^{k-1}\overline{t_{13}}^l\Lambda^j[t_{13},t_{12}+t_{23}]\Lambda^{k-1-j}\overline{t_{13}}^{m-l-1}\nn
\end{align}
where the second equality used the biderivation property of the commutator bracket, thus 
\begin{equation}\label{eq:a priori 2hol of hor interpol}
-\sum_{m=1}^\infty\sum_{k=1}^\infty\frac{(i\pi)^m(\ln\varepsilon)^k}{m!k!}\sum_{l=0}^{m-1}\sum_{j=0}^{k-1}\overline{t_{13}}^l\Lambda^j(\L+\R)\Lambda^{k-1-j}\overline{t_{13}}^{m-l-1}:\left[e^{i\pi\overline{t_{13}}},\varepsilon^\Lambda\right]\Rrightarrow0\quad.
\end{equation}
\end{subequations}
\end{rem}
Again, the arguments developed in Subsubsection \ref{subsub:1st 2path} allow us to immediately see that the 2-holonomy is given by 
\begin{equation}\label{eq:W^(P_IV)=}
W^{\P_\mathbf{IV}}=f_\mathbf{IV}(1)W^{p_\mathbf{IV}}
\end{equation}
where $f_\mathbf{IV}(s)$ is defined as
\begin{equation}\label{eq:f_IV(s)}
\sum_{\begin{smallmatrix}r=0\\k=1\end{smallmatrix}}^\infty\frac{1}{k!}\left(\ln\frac{\varepsilon}{1-\varepsilon}\right)^k\sum_{i_1,\ldots,i_r=0}^1\,\sum_{j=0}^{k-1}\int_{c_\mathbf{IV}(0)}^{c_\mathbf{IV}(s)}\prod_{m=1}^r\Omega_{i_m}\ad_{t_{12}}^{1-i_m}\ad_{t_{23}}^{i_m}\Lambda^j\left[\Omega_0\L+\Omega_1\R\right]\Lambda^{k-1-j}.
\end{equation}
\begin{constr}\label{con:lim 2hol of 3rd verinterpol}
We take $\lim_{\varepsilon\to0}$ of \eqref{eq:W^(P_IV)=}. Recall that $z(s)=c_\mathbf{IV}(s):=\frac{1}{2}+\left(\frac{1}{\varepsilon}-\frac{1}{2}\right)e^{-si\pi}$ hence $\frac{1}{z}\frac{\partial z}{\partial s}\to-i\pi$ and $\frac{1}{z-1}\frac{\partial z}{\partial s}\to-i\pi$ thus $\int_{c_\mathbf{IV}(0)}^{c_\mathbf{IV}(1)}\Omega_{i_1}\cdots\Omega_{i_r}\to\frac{(-i\pi)^r}{r!}$ so
\begin{equation}
f_\mathbf{IV}(1)\to\sum_{\begin{smallmatrix}r=0\\k=1\end{smallmatrix}}^\infty\frac{(-i\pi)^{r+1}(\ln\varepsilon)^k}{(r+1)!k!}\sum_{i_1,\ldots,i_r=0}^1\,\sum_{j=0}^{k-1}\prod_{m=1}^r\ad_{t_{12}}^{1-i_m}\ad_{t_{23}}^{i_m}\Lambda^j(\L+\R)\Lambda^{k-1-j}.
\end{equation}
We now use that $\sum_{i_1,\ldots,i_r=0}^1\ad_{t_{12}}^{1-i_1}\ad_{t_{23}}^{i_1}\cdots\ad_{t_{12}}^{1-i_r}\ad_{t_{23}}^{i_r}=(\ad_{t_{12}}+\ad_{t_{23}})^r=(-1)^r\ad^r_{\overline{t_{13}}}$ together with Lemma \ref{lem:ad^q}. As $\varepsilon\to0$, $W^{\P_\mathbf{IV}}$ becomes
\begin{subequations}\label{subeq:lim 2hol of 3rd verinterpol}
\begin{equation}
-\sum_{\begin{smallmatrix}r=0\\k=1\end{smallmatrix}}^\infty\frac{(i\pi)^{r+1}(\ln\varepsilon)^k}{(r+1)!k!}\sum_{q=0}^r\sum_{j=0}^{k-1}\binom{r}{q}(-1)^{r-q}\overline{t_{13}}^q\Lambda^j(\L+\R)\Lambda^{k-1-j}\overline{t_{13}}^{r-q}\sum_{l=0}^\infty\frac{(i\pi)^l}{l!}\overline{t_{13}}^l.
\end{equation}
We simply reindex $m=r+l+1$ hence $W^{\P_\mathbf{IV}}$ tends to
\begin{footnotesize}
\begin{equation}
\sum_{\begin{smallmatrix}m=1\\k=1\end{smallmatrix}}^\infty\frac{(i\pi)^m(\ln\varepsilon)^k}{m!k!}\sum_{l=0}^{m-1}\sum_{q=0}^{m-l-1}\sum_{j=0}^{k-1}\binom{m}{l}\binom{m-l-1}{q}(-1)^{m-l-q}\overline{t_{13}}^q\Lambda^j(\L+\R)\Lambda^{k-1-j}\overline{t_{13}}^{m-q-1}.
\end{equation}
\end{footnotesize}
\end{subequations}
Let us rewrite the troublesome double sum as $\sum_{l=0}^{m-1}\sum_{q=0}^{m-l-1}=\sum_{q=0}^{m-1}\sum_{l=0}^{m-q-1}$. We will be done if we can show that 
\begin{subequations}\label{subeq:1st alt double binom sum}
\begin{equation}\label{eq:alternating double binom sum}
\sum_{l=0}^{m-q-1}\binom{m}{l}\binom{m-l-1}{q}(-1)^l=(-1)^{m-q-1}
\end{equation}
so let us consider
\begin{equation}\label{eq:(1-x)^m-q-1=(1-x)^m(1-x)^-q-1}
(1-x)^{m-q-1}=(1-x)^m(1-x)^{-q-1}=\sum_{a=0}^m\binom{m}{a}(-x)^a\sum_{b=0}^\infty\binom{-q-1}{b}(-x)^b\quad.
\end{equation}
\end{subequations}
The generalised binomial coefficient becomes $\binom{-q-1}{b}=\frac{(-q-1)(-q-2)\cdots(-q-b)}{b!}=(-1)^b\binom{q+b}{q}$ and we recover \eqref{eq:alternating double binom sum} by considering the coefficient of $x^{m-q-1}$ in \eqref{eq:(1-x)^m-q-1=(1-x)^m(1-x)^-q-1}.
\end{constr}

\begin{cor}\label{cor:h^2 3rd 2hol}
The lowest order term of $\lim_{\varepsilon\to0}W^{\P_\mathbf{IV}}$ is given by
\begin{equation}\label{eqn:lowest term 3rd 2hol}
\left(\lim_{\varepsilon\to0}W^{\P_\mathbf{IV}}\right)_2=i\pi(\ln\varepsilon)\big(\L+\R\big)\quad.
\end{equation}
\end{cor}

\subsubsection{Putting the first three vertically-interpolative 2-paths together}
The \textit{purely horizontal} 2-path $\big([c_\mathbf{I}\,c_\mathbf{VI}\,c_\mathbf{V}]\circ\iota,a\big)\xRightarrow{\P_{v=a}}\big(c_\mathbf{IV}\,c_\mathbf{III}\,c_\mathbf{II},a\big)$ in
\begin{figure}[H]
    \centering
\scalebox{0.4}{\includesvg[width=1000pt]{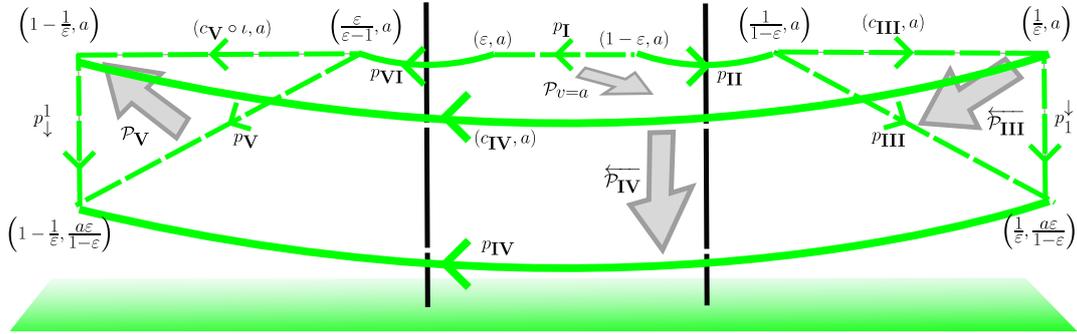}}
    \caption{The concrete 2-path $\P$ as a composite}
    \label{fig:P_2paths}
\end{figure}
is constant in the $v$-dimension thus $\Delta[\partial_s\P_{v=a},\partial_r\P_{v=a}]=0$ and hence $W^{\P_{v=a}}=\mathbf{0}$. We define the total 2-path $\P$ as the following vertical composite,
\begin{subequations}
\begin{equation}\label{eq:total 2-path P as vertical composite diagram}
\begin{tikzcd}
	{p_\mathbf{V}\,p_\mathbf{VI}\,p_\mathbf{I}} &&&& {p_\mathbf{IV}\,p_\mathbf{III}\,p_\mathbf{II}} \\
	{p^1_\downarrow\,(c_\mathbf{V}\circ\iota,a)\,p_\mathbf{VI}\,p_\mathbf{I}} && {p^1_\downarrow\,p^0_\mathbf{IV}\,(c_\mathbf{III},a)\,p_\mathbf{II}} && {p_\mathbf{IV}\,p^\downarrow_1\,(c_\mathbf{III},a)\,p_\mathbf{II}}
	\arrow["\P", Rightarrow, dashed, from=1-1, to=1-5]
	\arrow["{\P_\mathbf{V}1_{p_\mathbf{VI}\,p_\mathbf{I}}}"', Rightarrow, from=1-1, to=2-1]
	\arrow["{1_{p^1_\downarrow}\mathcal{P}_{v=a}}", Rightarrow, from=2-1, to=2-3]
	\arrow["{\overleftarrow{\P_\mathbf{IV}}1_{(c_\mathbf{III},a)p_\mathbf{II}}}", Rightarrow, from=2-3, to=2-5]
	\arrow["1_{p_\mathbf{IV}}\overleftarrow{\P_\mathbf{III}}1_{p_\mathbf{II}}"', Rightarrow, from=2-5, to=1-5]
\end{tikzcd}
\end{equation}
whose 2-holonomy is then given by
\begin{align}
W^\P=W^{\P_\mathbf{V}}W^{p_\mathbf{VI}}W^{p_\mathbf{I}}-W^{\P_\mathbf{IV}}W^{c_\mathbf{III}}W^{p_\mathbf{II}}-W^{p_\mathbf{IV}}W^{\P_\mathbf{III}}W^{p_\mathbf{II}}\quad.
\end{align}
Using Corollaries: \ref{cor:h^2 1st 2hol}, \ref{cor:h^2 2nd 2hol} and \ref{cor:h^2 3rd 2hol}, we have
\begin{equation}\label{eqn:hbar^2 of 1st stage 2-hol}
\left(\varepsilon^{-t_{13}}W^\P \varepsilon^{t_{23}}\right)_2\to\,\frac{\pi^2}{6}\left(\R-\L\right)+i\pi\ln(\varepsilon)\big(\L+\R\big)\quad.
\end{equation}
We recover the infinitesimal right hexagonator by siphoning off the second-order excess terms of 
\begin{equation}
\Phi_{213}e^{i\pi\overline{t_{(12)3}}}\Phi_{321}-\varepsilon^{-\ad_{t_{13}}}\left[e^{i\pi\overline{t_{13}}}\right]\Phi_{231}e^{i\pi t_{23}}\quad,
\end{equation}
i.e.,
\begin{align}
&\frac{\pi^2}{2}\left(\Lambda t_{(12)3}+t_{(12)3}\Lambda-\Lambda^2\right)+i\pi\ln(\varepsilon)[\Lambda,t_{13}]+\frac{\pi^2}{2}\left(\Lambda^2-\Lambda t_{13}-t_{13}\Lambda\right)-\pi^2\Lambda t_{23}\nn\\&\equiv\frac{\pi^2}{2}[t_{23},\Lambda]+i\pi\ln(\varepsilon)[\Lambda,t_{13}]\,.
\end{align}
\end{subequations}


\subsection{The last three vertically-interpolative 2-paths for constructing the hexagonator}
\begin{figure}[H]
    \centering
    \scalebox{0.4}{\includesvg[width=1000pt]{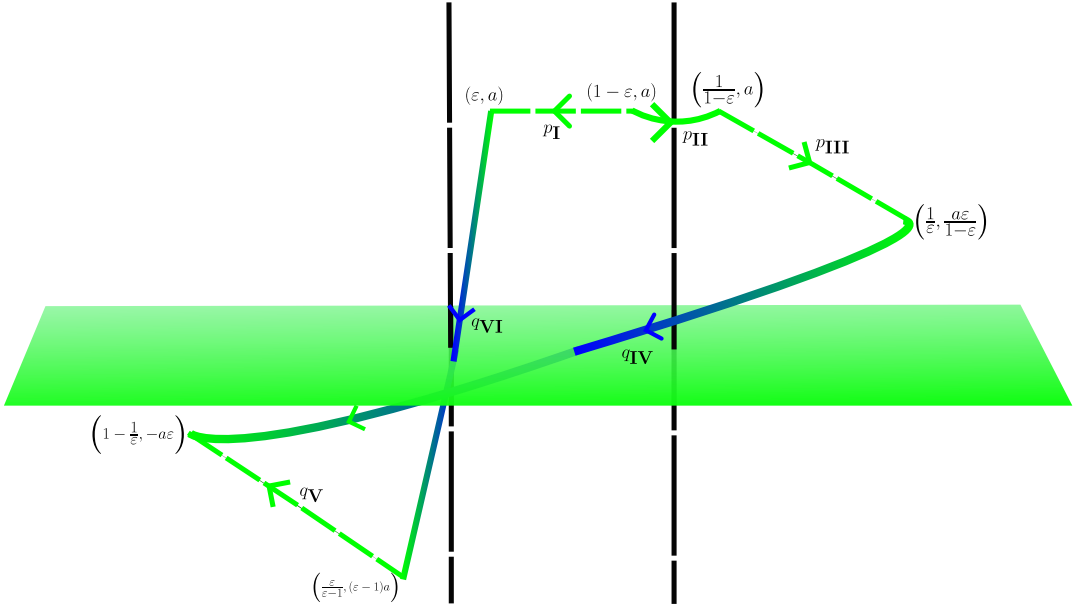}}
    \caption{The 1-paths for the right pre-hexagonator series}
    \label{fig:q1paths}
\end{figure}
\begin{subequations}
\begin{alignat}{3}
q_\mathbf{IV}(r)&:=\tau_{12}(c_\mathbf{II}(r+1),-a)=\left(\frac{1}{2}+\left(\frac{1}{\varepsilon}-\frac{1}{2}\right)e^{-i\pi r},\frac{a}{\left(\frac{1}{\varepsilon}-\frac{1}{2}\right)e^{-i\pi r}-\frac{1}{2}}\right)\quad&&,\\
q_\mathbf{V}(r)&:=\tau_{12}(c_\mathbf{I}(r),-a)=\left(\frac{\varepsilon+r(1-2\varepsilon)}{\varepsilon-1+r(1-2\varepsilon)},(\varepsilon-1+r(1-2\varepsilon))a\right)\quad&&,\\
q_\mathbf{VI}(r)&:=\tau_{(12)3}([c_\mathbf{IV}\circ\iota](r),-a\varepsilon)=\left(\frac{\varepsilon}{\frac{\varepsilon}{2}+\left(1-\frac{\varepsilon}{2}\right)e^{i\pi r}},\left[\frac{\varepsilon}{2}+\left(1-\frac{\varepsilon}{2}\right)e^{i\pi r}\right]a\right)\quad&&.
\end{alignat} 
\end{subequations}
\begin{rem}\label{rem:2nd explicit parallel transport}
Looking at the proof of \cite[Lemma 21]{BRW}, we can see that the parallel transport along $c_\mathbf{II}(r+1)$ is similar to that along $c_\mathbf{II}(r)$, i.e., 
\begin{equation}
W_{21}^{c_\mathbf{II}}\equiv e^{i\pi t_{23}}H'_\varepsilon(t_{23}\,,t_{12})
\end{equation}
\begin{subequations}
hence we have
\begin{equation}
W^{q_\mathbf{IV}}\equiv\,{}^{\nabla}W_{21}^{\tau_{12}(c_\mathbf{II},-a)}\equiv {}^{\Gamma(t_{12}\,,t_{13})}W_{21}^{c_\mathbf{II}}\equiv e^{i\pi t_{13}}H'_\varepsilon(t_{13}\,,t_{12})
\end{equation}
where we used \eqref{eq:tau_12 pullback} for the second equality. Using \eqref{eqn:tau_312 pullback} and Remark \ref{rem:explicit parallel transports}, we also have:
\begin{alignat}{2}
W^{q_\mathbf{V}}\equiv &\,{}^{\nabla}W^{\tau_{12}(c_\mathbf{I},-a)}\equiv {}^{\Gamma(t_{12}\,,t_{13})}W^{c_\mathbf{I}}\equiv e^{  \ln(\varepsilon)t_{13}}\Phi^\varepsilon_{213}e^{-  \ln(\varepsilon)t_{12}}\,&&,\label{eq:W^q_V}\\
W^{q_\mathbf{VI}}\equiv &\,{}^{\nabla}W^{\tau_{(12)3}(c_\mathbf{IV}\circ\iota,-a\varepsilon)}\equiv \big({}^{\Gamma(t_{23}\,,t_{13})}W^{c_\mathbf{IV}}\big)^{-1}\equiv H_\varepsilon\left(\overline{t_{12}},t_{13}\right)^{-1}e^{i\pi(t_{13}+t_{23})}\,&&.\label{eq:W^q_VI}
\end{alignat}
\end{subequations}
\end{rem}
Given some 2-path $q_\mathbf{V}\,q_\mathbf{VI}\,p_\mathbf{I}\xRightarrow{\Q}q_\mathbf{IV}\,p_\mathbf{III}\,p_\mathbf{II}$, the globularity condition \eqref{eq:globularity condition} imposes
\begin{align}
\varepsilon^{- t_{13}}W^\Q\varepsilon^{ t_{23}}:\Phi^\varepsilon_{213}[\widetilde{H}_\varepsilon^\mathbf{6'}]^{-1}\varepsilon^{-\ad_{t_{12}}}\big[e^{i\pi t_{(12)3}}\big]\Phi^\varepsilon_{321}\Rrightarrow e^{i\pi t_{13}}\widetilde{H}_\varepsilon^\mathbf{4'}\Phi^\varepsilon_{231}e^{i\pi t_{23}}\widetilde{H}_\varepsilon^\mathbf{2}.
\end{align}
Here we can easily siphon off the surplus terms (to all orders) by noting 
\begin{equation}\label{eq:siphoning off surplus terms}
[\varepsilon^{-t_{12}},e^{i\pi t_{(12)3}}]=\sum_{\begin{smallmatrix}m=1\\n=1
\end{smallmatrix}}^\infty\frac{(-\ln\varepsilon)^n(i\pi)^m}{n!m!}\sum_{\begin{smallmatrix}0\leq j\leq m-1\\0\leq k\leq n-1
\end{smallmatrix}}t_{(12)3}^jt_{12}^k[t_{12},t_{(12)3}]t_{12}^{n-1-k}t_{(12)3}^{m-1-j}
\end{equation}
where we used the biderivation property of the commutator bracket as in \eqref{eq:1st determination}. We define
\begin{equation}\label{eqn:def Ad}
t_\varepsilon^{t_{12}}:=\sum_{\begin{smallmatrix}m=1\\n=1
\end{smallmatrix}}^\infty\frac{(-\ln\varepsilon)^n(i\pi)^m}{n!m!}\sum_{\begin{smallmatrix}0\leq j\leq m-1\\0\leq k\leq n-1
\end{smallmatrix}}t_{(12)3}^jt_{12}^k\L t_{12}^{n-1-k}t_{(12)3}^{m-1-j}\varepsilon^{t_{12}}
\end{equation}
so that
\begin{equation}\label{eqn:R^epsilon=2hol-Leibniz}
\RR^\varepsilon:=\varepsilon^{- t_{13}}W^\Q\varepsilon^{ t_{23}}-\Phi^\varepsilon_{213}[\widetilde{H}_\varepsilon^\mathbf{6'}]^{-1}t_\varepsilon^{t_{12}}\Phi^\varepsilon_{321}
\end{equation}
and the right pre-hexagonator series is given by $\RR:=\lim_{\varepsilon\rightarrow0}\RR^\varepsilon$.
\subsubsection*{The fourth vertically-interpolative 2-path}
\begin{equation}
\Q_\mathbf{VI}(s,r):=\begin{cases}
      p_\mathbf{VI}(2r), & 0\leq r\leq\frac{s}{2}\\
      q_\downarrow^s(2r-s):=\left([c_\mathbf{VI}\circ\iota](s)\,,\,\frac{a\varepsilon}{[c_\mathbf{VI}\circ\iota](2r-s)}\right), & \frac{s}{2}\leq r\leq s\\
      q_\mathbf{VI}(r), & s\leq r\leq1
    \end{cases}\,:q_\mathbf{VI}\Rightarrow q^1_\downarrow\,p_\mathbf{VI}\,.
\end{equation}
\begin{figure}[H]
    \centering
    \scalebox{0.15}{\includesvg[width=1000pt]{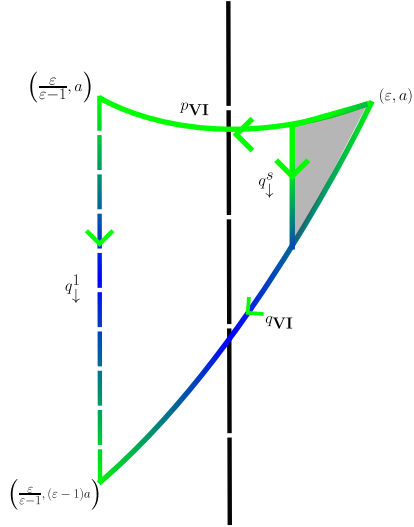}}
    \caption{The fourth vertically-interpolative 2-path}
    \label{fig:3rdvertinterpol}
\end{figure}
This time we have $\frac{v(s,s)}{a}=\frac{\varepsilon}{[c_\mathbf{VI}\circ\iota](s)}=\frac{\varepsilon}{z}$ and we use \eqref{eq:int_a^b Omega_0^k} in writing:
\begin{equation}
\int_0^s\frac{\left[\ln\frac{z}{\varepsilon}\right]^k}{k!z}\frac{\partial z}{\partial s'}\dd s'=\int_\varepsilon^{[c_\mathbf{VI}\circ\iota](s)}\Omega_0^{k+1}\qquad,\qquad
\int_\varepsilon^s\frac{\left[\ln\frac{z}{\varepsilon}\right]^k}{k!(z-1)}\frac{\partial z}{\partial s'}\dd s'=\int_\varepsilon^{[c_\mathbf{VI}\circ\iota](s)}\Omega_1\Omega_0^k\,.
\end{equation}
We have the following two expressions for the parallel transports:
\begin{equation}
W_{s0}^{p_\mathbf{VI}}\equiv{}^{\Gamma(t_{12},t_{23})}W_{s0}^{c_\mathbf{VI}\circ\iota}\qquad,\qquad\left(W_{s0}^{q_\mathbf{VI}}\right)^{-1}\equiv\big({}^{\Gamma(\overline{t_{12}},t_{23})}W_{s0}^{c_\mathbf{VI}\circ\iota}\big)^{-1}\,,
\end{equation}
where the first equality comes from \eqref{eq:8 of first partras} and the second equality comes from comparing \eqref{eq:W^q_VI} with \eqref{eq:8 of first partras}. The 2-holonomy is thus $W^{\Q_\mathbf{VI}}=g_\mathbf{VI}(1)W^{c_\mathbf{VI}\circ\iota}$ where $g_\mathbf{VI}(s)$ is given by
\begin{align}
\sum_{\begin{smallmatrix}r=0\\k=1
\end{smallmatrix}}^\infty(-1)^k\sum_{i_1,\ldots,i_r=0}^1\,\sum_{j=0}^{k-1}\int_{[c_\mathbf{VI}\circ\iota](0)}^{[c_\mathbf{VI}\circ\iota](s)}\prod_{m=1}^r\Omega_{i_m}[\ad_{t_{12}}-\Lambda]^{1-i_m}\ad_{t_{23}}^{i_m}\Lambda^j\left[\Omega_0\L+\Omega_1\R\right]\Omega_0^k\Lambda^{k-1-j}.
\end{align}

\begin{rem}
The globularity condition \eqref{eq:globularity condition} requires 
\begin{subequations}\label{subeq:W^(Q_VI) a priori from glob}
\begin{equation}\label{eq:g_uparrow a priori}
\lim_{\varepsilon\to0}(W^{\Q_\mathbf{VI}}):e^{i\pi t_{(12)3}}\Rrightarrow e^{i\pi\Lambda}e^{-i\pi t_{12}}\quad.
\end{equation}
An analogous argument to Corollary \ref{cor:BCH-type mod} shows that 
\begin{align}\label{eq:W^(Q_VI) from glob cond}
\lim_{\varepsilon\to0}(W^{\Q_\mathbf{VI}})=-\sum_{m=2}^\infty\frac{(i\pi)^m}{m!}\sum_{l=1}^{m-1}\sum_{q=0}^{m-l-1}\sum_{j=0}^{m-l-q-1}\binom{m-l}{q}(-1)^qt_{(12)3}^{l-1}\Lambda^j\L\Lambda^{m-l-q-1-j}t_{12}^q
\end{align}
\end{subequations}
which has the lowest order term 
\begin{equation}\label{eq:W^Q_Uparrow order 2 term}
(W^{\Q_\mathbf{VI}})_2\to\frac{\pi^2}{2}\L\quad.
\end{equation}
\end{rem}
As in Construction \ref{con:lim 2hol of 3rd verinterpol}, we directly calculate the limit of the 2-holonomy $W^{\Q_\mathbf{VI}}=g_\mathbf{VI}(1)W^{c_\mathbf{VI}\circ\iota}$ and show that we recover exactly the series formula \eqref{eq:W^(Q_VI) from glob cond}.
\begin{constr}\label{con:limit of 4th 2hol}
Because $z=[c_\mathbf{VI}\circ\iota](s)=\frac{2\varepsilon}{\varepsilon+(2-\varepsilon)e^{i\pi s}}$, we have that $\frac{1}{z}\frac{\partial z}{\partial s}\to-i\pi$ and $\frac{1}{z-1}\frac{\partial z}{\partial s}\to0$ as $\varepsilon\to0$. In other words, iterated integrals with appearances of $\Omega_1$ will vanish under the limit. We also have that $(\Lambda-\ad_{t_{12}})^r=(t_{(12)3}+R_{t_{12}})^r=\sum_{m=0}^r\binom{r}{m}t_{(12)3}^mR_{t_{12}}^{r-m}$ and \eqref{eq:int_a^b Omega_0^k} tells us that $\int_{[c_\mathbf{VI}\circ\iota](0)}^{[c_\mathbf{VI}\circ\iota](1)}\Omega_0^{r+k+1}\to\frac{(-i\pi)^{r+k+1}}{(r+k+1)!}$ hence 
\begin{equation}\label{eq:limit of hol of Q uparrow}
W^{\Q_\mathbf{VI}}\to-\sum_{\begin{smallmatrix}r=0\\k=1
\end{smallmatrix}}^\infty\frac{(i\pi)^{r+k+1}}{(r+k+1)!}\sum_{j=0}^{k-1}\sum_{m=0}^r\binom{r}{m}t_{(12)3}^m\Lambda^j\L\Lambda^{k-1-j}t_{12}^{r-m}e^{-i\pi t_{12}}\quad.
\end{equation}
As in \eqref{subeq:lim 2hol of 3rd verinterpol} of Construction \ref{con:lim 2hol of 3rd verinterpol}, we reindex all the sums appearing in \eqref{eq:limit of hol of Q uparrow} to eventually get 
\begin{equation}
\sum_{m=2}^\infty\frac{(i\pi)^m}{m!}\sum_{l=1}^{m-1}\sum_{q=0}^{m-l-1}\sum_{j=0}^{m-l-q-1}\sum_{n=1}^{q+1}\binom{m}{l-1}\binom{q}{n-1}(-1)^lt_{(12)3}^{n-1}\Lambda^j\L\Lambda^{m-l-q-1-j}t_{12}^{l+q-n}.
\end{equation}
We rewrite the troublesome triple sum as $\sum_{l=1}^{m-1}\sum_{q=0}^{m-l-1}\sum_{n=1}^{q+1}=\sum_{n=1}^{m-1}\sum_{l=1}^{m-n}\sum_{q=n-1}^{m-l-1}$. Now let us reindex twice; first by $a=q-n$ and then by $b=a+l$,
\begin{footnotesize}
\begin{equation}
\sum_{l=1}^{m-n}\sum_{a=-1}^{m-l-n-1}\binom{m}{l-1}\binom{a+n}{n-1}(-1)^lt_{(12)3}^{n-1}Xt_{12}^{l+a}=\sum_{b=0}^{m-n-1}\sum_{l'=0}^b\binom{m}{l'}\binom{b+n-l'-1}{n-1}(-1)^{l'-1}t_{(12)3}^{n-1}Xt_{12}^b.
\end{equation}
\end{footnotesize}
As in \eqref{subeq:1st alt double binom sum}, we equate the coefficient of $x^{l'}$ in $(1-x)^{m-n}=(1-x)^m(1-x)^{-n}$ to get the binomial convolution formula $\sum_{l'=0}^b\binom{m}{l'}\binom{b+n-l'-1}{n-1}(-1)^{l'}=\binom{m-n}{b}(-1)^b$.
\end{constr}


\subsubsection*{The fifth vertically-interpolative 2-path} 
\begin{equation}
\Q_\mathbf{V}:=\begin{cases}
      p_\mathbf{V}(2r), & 0\leq r\leq\frac{s}{2}\\
      q_\leftarrow^s(2r-s):=\left([c_\mathbf{V}\circ\iota](s),\frac{a\varepsilon[1-c_\mathbf{I}(s)]}{(1-\varepsilon)[c_\mathbf{VI}\circ\iota](2r-s)}\right), & \frac{s}{2}\leq r\leq\frac{1+s}{2}\\
      q_\mathbf{V}(2r-1), & \frac{1+s}{2}\leq r\leq1
    \end{cases}:q_\mathbf{V}q_\leftarrow^0\Rightarrow q_\leftarrow^1p_\mathbf{V}\,.
\end{equation}
\begin{figure}[H]
    \centering
    \scalebox{0.35}{\includesvg[width=1000pt]{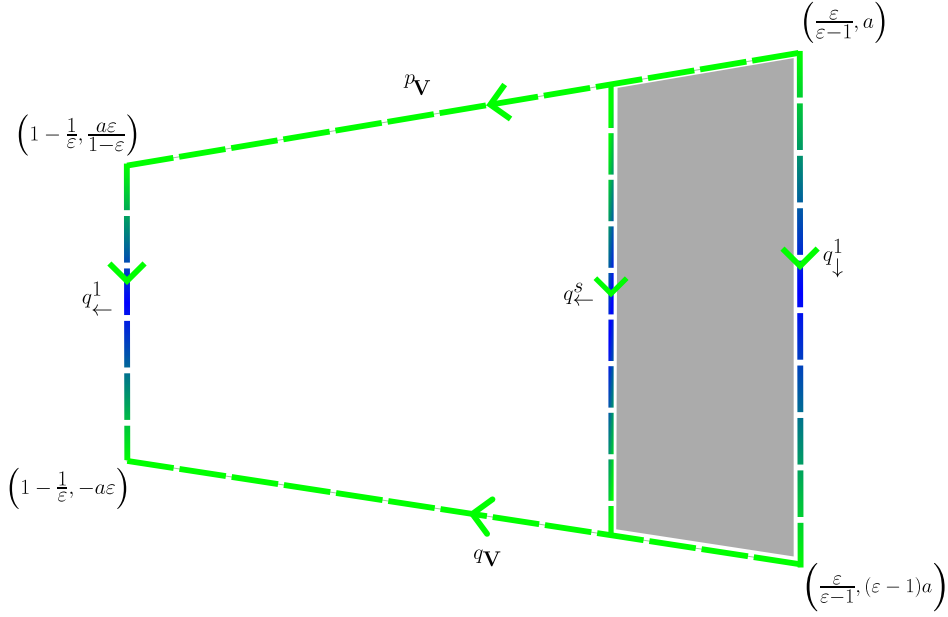}}
    \caption{The fifth vertically-interpolative 2-path, note that $q^1_\downarrow=q^0_\leftarrow$}
    \label{fig:5th interpol}
\end{figure}
We have $\frac{v\left(\frac{1+s}{2},s\right)}{a}\frac{1-\varepsilon}{1-c_\mathbf{I}(s)}=\varepsilon-1$ and $W_{s0}^{p_\mathbf{V}}\equiv W_{s0}^{q_\mathbf{V}}\equiv {}^{\Gamma(t_{12},\,t_{13})}W_{s0}^{c_\mathbf{I}}$ from \eqref{eqn:partra along p_V} and \eqref{eq:W^q_V}, respectively. As in \eqref{eq:z=c_V(1-s) implies 1/z dz/ds =.. and 1/z-1 dz/ds=..}, we have:
\begin{equation}
\int_0^s\frac{1}{z}\frac{\partial z}{\partial s'}\dd s'=\int_\varepsilon^{c_\mathbf{I}(s)}\Omega_0-\int_\varepsilon^{c_\mathbf{I}(s)}\Omega_1\qquad,\qquad\int_0^s\frac{1}{z-1}\frac{\partial z}{\partial s'}\dd s'=-\int_\varepsilon^{c_\mathbf{I}(s)}\Omega_1\quad.
\end{equation}
The 2-holonomy is thus $W^{\Q_\mathbf{V}}= g_\mathbf{V}(1)W^{p_\mathbf{V}}$ where $g_\mathbf{V}(s)$ is given by
\begin{align}
\sum_{\begin{smallmatrix}r=0\\k=1
\end{smallmatrix}}^\infty\frac{\left[\ln(\varepsilon-1)\right]^k}{k!}\sum_{i_1,\ldots,i_r=0}^1\,\sum_{j=0}^{k-1}\int_\varepsilon^{c_\mathbf{I}(s)}\prod_{m=1}^r\Omega_{i_m}\ad_{t_{12}}^{1-i_m}\ad_{t_{13}}^{i_m}\Lambda^j\left[\Omega_0\L-\Omega_1(\L+\R)\right]\Lambda^{k-1-j}\,.
\end{align}
We thus have 
\begin{equation}\label{eq:W^Q_Leftarrow order 2 term}
(W^{\Q_\mathbf{V}})_2\to -i\pi\ln(\varepsilon)\big(2\L+\R\big)
\end{equation}

\begin{rem}
The globularity condition \eqref{eq:globularity condition} imposes
\begin{equation}
\lim_{\varepsilon\to0}(W^{\Q_\mathbf{V}}):\varepsilon^{t_{13}}\Phi_{213}\varepsilon^{-t_{12}}e^{i\pi\Lambda}\Rrightarrow e^{i\pi\Lambda}\varepsilon^{t_{13}}\Phi_{213}\varepsilon^{-t_{12}}\quad,
\end{equation}
so we expand the commutator bracket 
\begin{equation}\label{eq:comm bracket of fifth 2hol}
\left[\varepsilon^{t_{13}}\Phi_{213}\varepsilon^{-t_{12}},e^{i\pi\Lambda}\right]=\varepsilon^{t_{13}}\Phi_{213}\left[\varepsilon^{-t_{12}},e^{i\pi\Lambda}\right]+\left[\varepsilon^{t_{13}},e^{i\pi\Lambda}\right]\Phi_{213}\varepsilon^{-t_{12}}+\varepsilon^{t_{13}}\left[\Phi_{213},e^{i\pi\Lambda}\right]\varepsilon^{-t_{12}}.
\end{equation} 
We already know from \eqref{eq:1st determination} and \eqref{eq:siphoning off surplus terms} how to extract modifications from terms like $\left[\varepsilon^{-t_{12}},e^{i\pi\Lambda}\right]$ and $\left[\varepsilon^{t_{13}},e^{i\pi\Lambda}\right]$ so let us focus on $\left[\Phi_{213},e^{i\pi\Lambda}\right]$. We compactify (even more) the notation surrounding Drinfeld's KZ associator series in Remark \ref{rem:LM Drinfeld series}, i.e., let us define:
\begin{subequations}
\begin{alignat}{6}
\zeta_{\j,\k}^{\p,\q}&:=(-1)^{|\p|+|\j|+|\k|}\zeta(p_1+1,\{1\}^{q_1-1},\ldots,p_{\tilde{\p}}+1,\{1\}^{q_{\tilde{\p}}-1})\prod_{l=1}^{\tilde{\p}}\binom{p_l}{j_l}\binom{q_l}{k_l}\quad&&,\\j_0&:=0\qquad,\qquad k_0:=|\q|-|\k|\qquad,\qquad j_{\tilde{\p}+1}:=|\p|-|\j|\qquad,\qquad k_{\tilde{\p}+1}:=0\quad&&,
\end{alignat}
\end{subequations}
so that 
\begin{equation}
\Phi_{213}=1+\sum_{\{\p,\q>0\,:\,\tilde{\p}=\tilde{\q}\}}\sum_{\begin{smallmatrix}0\leq\j\leq\p\\0\leq\k\leq\q\end{smallmatrix}}\zeta_{\j,\k}^{\p,\q}\prod_{l=0}^{{\tilde{\p}}+1}t_{12}^{j_l}t_{13}^{k_l}\quad.
\end{equation}
We expand $\left[\prod_{l=0}^{\tilde{\p}+1}t_{12}^{j_l}t_{13}^{k_l},\Lambda\right]$ as
\begin{align}
\sum_{l=0}^{{\tilde{\p}}+1}\prod_{n=0}^{l-1}t_{12}^{j_n}t_{13}^{k_n}\left[\sum_{m=1}^{j_l}t_{12}^{m-1}[t_{12},\Lambda]t_{12}^{j_l-m}t_{13}^{k_l}+t_{12}^{j_l}\sum_{m=1}^{k_l}t_{13}^{m-1}[t_{13},\Lambda]t_{13}^{k_l-m}\right]\prod_{n=l+1}^{{\tilde{\p}}+1}t_{12}^{j_n}t_{13}^{k_n}.
\end{align}
so that we have the following series formula for the modification $\int\{[\Phi_{213},e^{i\pi\Lambda}]\}$, 
\begin{footnotesize}
\begin{align}\label{eq:int[phi_213,e^Lambda]}
\sum_{\begin{smallmatrix}
\{\p,\q>0\,|\,\tilde{\p}=\tilde{\q}\}\\1\leq s\leq\infty
\end{smallmatrix}}\sum_{\begin{smallmatrix}0\leq\j\leq\p\\0\leq\k\leq\q\\0\leq l\leq{\tilde{\p}}+1\\1\leq a\leq s\end{smallmatrix}}\tfrac{(i\pi)^s\zeta_{\j,\k}^{\p,\q}}{s!}\Lambda^{a-1}\prod_{n=0}^{l-1}t_{12}^{j_n}t_{13}^{k_n}\left[\sum_{m=1}^{j_l}t_{12}^{m-1}\L t_{12}^{j_l-m}t_{13}^{k_l}-t_{12}^{j_l}\sum_{m=1}^{k_l}t_{13}^{m-1}(\L+\R)t_{13}^{k_l-m}\right]\prod_{n=l+1}^{{\tilde{\p}}+1}t_{12}^{j_n}t_{13}^{k_n}\Lambda^{s-a}
\end{align}
\end{footnotesize}
\end{rem}

\subsubsection*{The sixth and last vertically-interpolative 2-path}
\begin{equation}
\Q_\mathbf{IV}:=\begin{cases}
      p_\mathbf{IV}(2r), & 0\leq r\leq\frac{s}{2}\\
      q^\downarrow_s(2r-s):=\left(c_\mathbf{IV}(s)\,,\,\frac{a}{c_\mathbf{IV}(2r-s)-1}\right), & \frac{s}{2}\leq r\leq s\\
      q_\mathbf{IV}(r), & s\leq r\leq1
    \end{cases}\,:q_\mathbf{IV}\Rightarrow q^\downarrow_1p_\mathbf{IV}\,.
\end{equation}
\begin{figure}[H]
    \centering
    \scalebox{0.4}{\includesvg[width=1000pt]{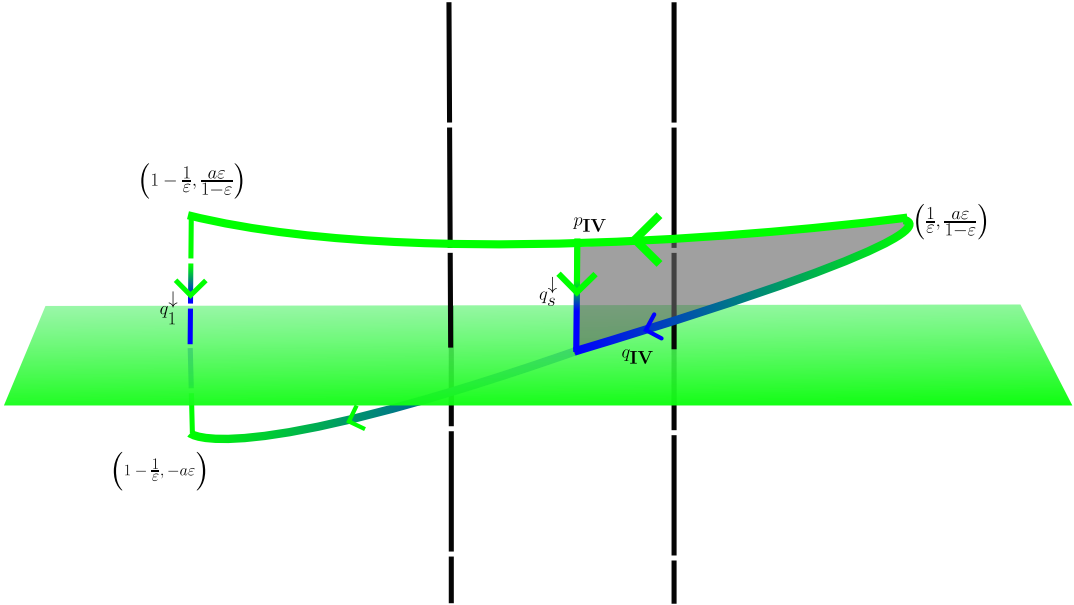}}
    \caption{The final vertically-interpolative 2-path}
    \label{fig:last vertical interpolative}
\end{figure}
Given that $\lim_{\varepsilon\to0}(W^{\Q_\mathbf{IV}}):e^{i\pi t_{13}}\Rrightarrow e^{i\pi\Lambda}e^{i\pi\overline{t_{13}}}$, Corollary \ref{cor:BCH-type mod} immediately gives us 
\begin{equation}\label{eq:W^(Q_IV)}
W^{\Q_\mathbf{IV}}\to-\sum_{m=2}^\infty\frac{(i\pi)^m}{m!}\sum_{l=1}^{m-1}\sum_{q=0}^{m-l-1}\sum_{j=0}^{m-l-q-1}\binom{m-l}{q}t_{13}^{l-1}\Lambda^j(\L+\R)\Lambda^{m-l-q-1-j}\overline{t_{13}}^q
\end{equation}
which has the lowest order term 
\begin{equation}\label{eq:W^Q_Downarrow order 2 term}
(W^{\Q_\mathbf{IV}})_2\to\frac{\pi^2}{2}(\L+\R)\quad.
\end{equation}

\subsection{The total 2-path and its 2-holonomy}\label{sub:finale of hex}
We omit the \textit{purely vertical} 2-path $q^1_\leftarrow\xRightarrow{\Q_{z=1-\frac{1}{\varepsilon}}}q^\downarrow_1$ below because it provides trivial 2-holonomy.
\begin{figure}[H]
    \centering
    \scalebox{0.4}{\includesvg[width=1000pt]{Q2paths}}
    \caption{The total 2-path $\Q$ as a composite}
    \label{fig:Q2paths}
\end{figure}
We construct the total 2-path $\Q$ as the following vertical composite:
\begin{subequations}
\begin{equation}
\begin{tikzcd}
	{q_\mathbf{V}\,q_\mathbf{VI}\,p_\mathbf{I}} &&&& {q_\mathbf{IV}\,p_\mathbf{III}\,p_\mathbf{II}} \\
	{q_\mathbf{V}\,q^1_\downarrow\,p_\mathbf{VI}\,p_\mathbf{I}} && {q^1_\leftarrow\,p_\mathbf{V}\,p_\mathbf{VI}\,p_\mathbf{I}} && {q^\downarrow_1\,p_\mathbf{IV}\,p_\mathbf{III}\,p_\mathbf{II}}
	\arrow["\Q", Rightarrow, dashed, from=1-1, to=1-5]
	\arrow["{1_{q_\mathbf{V}}\,\Q_\mathbf{VI}\,1_{p_\mathbf{I}}}"', Rightarrow, from=1-1, to=2-1]
	\arrow["{\Q_\mathbf{V}\,1_{p_\mathbf{VI}p_\mathbf{I}}}", Rightarrow, from=2-1, to=2-3]
	\arrow["{\mathcal{Q}_{z=1-\frac{1}{\varepsilon}}\P}", Rightarrow, from=2-3, to=2-5]
	\arrow["{\overleftarrow{\Q_\mathbf{IV}}\,1_{p_\mathbf{III}\,p_\mathbf{II}}}"', Rightarrow, from=2-5, to=1-5]
\end{tikzcd}
\end{equation}
whose 2-holonomy is thus given by 
\begin{align}
W^\Q=W^{q_\mathbf{V}}W^{\Q_\mathbf{VI}}W^{p_\mathbf{I}}+W^{\Q_\mathbf{V}}W^{p_\mathbf{VI}}W^{p_\mathbf{I}}+W^{q^\downarrow_1}W^\P-W^{\Q_\mathbf{IV}}W^{p_\mathbf{III}}W^{p_\mathbf{II}}\quad.
\end{align}
Recall from \eqref{eqn:R^epsilon=2hol-Leibniz} that this 2-holonomy features in the construction of the right pre-hexagonator series $\RR:=\lim_{\varepsilon\to0}\RR^\varepsilon$ as
\begin{equation}\label{eqn:R under limit}
\RR=\varepsilon^{-t_{13}}W^\Q\varepsilon^{t_{23}}-\Phi_{213}\sum_{\begin{smallmatrix}m=1\\n=1
\end{smallmatrix}}^\infty\frac{(-\ln\varepsilon)^n(i\pi)^m}{n!m!}\sum_{\begin{smallmatrix}1\leq j\leq m\\1\leq k\leq n
\end{smallmatrix}}t_{(12)3}^{j-1}t_{12}^{k-1}\L t_{12}^{n-k}t_{(12)3}^{m-j}\varepsilon^{t_{12}}\Phi_{321}\,.
\end{equation}
We can derive the infinitesimal right hexagonator again by using the 2nd order limits: \eqref{eqn:hbar^2 of 1st stage 2-hol}, \eqref{eq:W^Q_Uparrow order 2 term}, \eqref{eq:W^Q_Leftarrow order 2 term} and \eqref{eq:W^Q_Downarrow order 2 term}, 
\begin{align}
(\RR)_2&=\frac{\pi^2}{6}\big(\R-\L\big)+i\pi\ln(\varepsilon)\big(\L+\R\big)+\frac{\pi^2}{2}\L-i\pi\ln(\varepsilon)\big(2\L+\R\big)-\frac{\pi^2}{2}(\L+\R)+i\pi(\ln\varepsilon)\L\nn\\&=-\frac{\pi^2}{6}(\L+2\R)\quad.\label{eq:2nd derivation of inf right hex}
\end{align}
\end{subequations}
\subsubsection{The hexagonator series formula from BRW's functional equation}
Recall the purely horizontal 2-path $\P_{v=a}$ from Figure \ref{fig:P_2paths} and its trivial 2-holonomy, 
\begin{equation}\label{eq:W^(P_v=a)}
W^{\P_{v=a}}=\mathbf{0}:W^{c_\mathbf{V}\circ\iota}W^{p_\mathbf{VI}}W^{p_\mathbf{I}}\equiv W^{p_\mathbf{IV}}W^{c_\mathbf{III}}W^{p_\mathbf{II}}\quad.
\end{equation}
If we substitute the explicit expressions for the parallel transports in \eqref{eq:W^(P_v=a)} and take the limit $\varepsilon\to0$ then we recover the \textbf{BRW functional equation} \cite[Last equation in the proof of Theorem 22]{BRW}, i.e.,
\begin{equation}\label{eq:BRW functional relation}
\varepsilon^{-\overline{t_{13}}}W^{\P_{v=a}}\varepsilon^{t_{23}}=\mathbf{0}:\Phi(t_{12},\overline{t_{13}})e^{-i\pi t_{12}}\Phi_{321}\equiv e^{i\pi\overline{t_{13}}}\Phi(t_{23},\overline{t_{13}})e^{i\pi t_{23}}\quad.
\end{equation}
We can construct the right pre-hexagonator series $\RR$ directly by using \eqref{eq:BRW functional relation} together with the modifications: \eqref{eq:W^(Q_VI) from glob cond}, \eqref{eq:int[phi_213,e^Lambda]}, \eqref{eq:mod from phi_213}, \eqref{eq:mod from phi_231} and \eqref{eq:W^(Q_IV)}, i.e.,
\begin{footnotesize}
\begin{equation}\label{eq:direct construction of hexagonator}
\begin{tikzcd}
	{\Phi_{213}e^{i\pi t_{(12)3}}\Phi_{321}} & {e^{i\pi t_{13}}\Phi_{231}e^{i\pi t_{23}}} \\
	{\Phi_{213}e^{i\pi\Lambda}e^{-i\pi t_{12}}\Phi_{321}} & {e^{i\pi\Lambda}e^{i\pi\overline{t_{13}}}\Phi_{231}e^{i\pi t_{23}}} \\
	{e^{i\pi\Lambda}\Phi_{213}e^{-i\pi t_{12}}\Phi_{321}} & {e^{i\pi\Lambda}\Phi(t_{12},\overline{t_{13}})e^{-i\pi t_{12}}\Phi_{321}\equiv e^{i\pi\Lambda}e^{i\pi\overline{t_{13}}}\Phi(t_{23},\overline{t_{13}})e^{i\pi t_{23}}}
	\arrow["{\mathbf{R}}", Rightarrow, scaling nfold=3, dashed, from=1-1, to=1-2]
	\arrow["{\Phi_{213}W^{\mathcal{Q}_\mathbf{VI}}\Phi_{321}}"', Rightarrow, scaling nfold=3, from=1-1, to=2-1]
	\arrow["{\int\{[\Phi_{213},e^{i\pi\Lambda}]\}e^{-i\pi t_{12}}\Phi_{321}}"', Rightarrow, scaling nfold=3, from=2-1, to=3-1]
	\arrow["{-W^{\mathcal{Q}_\mathbf{IV}}\Phi_{231}e^{i\pi t_{23}}}"', Rightarrow, scaling nfold=3, from=2-2, to=1-2]
	\arrow["{e^{i\pi\Lambda}\int\{\Phi_{213}-\Phi(t_{12},\overline{t_{13}})\}e^{-i\pi t_{12}}\Phi_{321}}"', curve={height=12pt}, Rightarrow, scaling nfold=3, from=3-1, to=3-2]
	\arrow["{-e^{i\pi\Lambda}e^{i\pi\overline{t_{13}}}\int\{\Phi_{231}-\Phi(t_{23},\overline{t_{13}})\}e^{i\pi t_{23}}}"', Rightarrow, scaling nfold=3, from=3-2, to=2-2]
\end{tikzcd}
\end{equation}
\end{footnotesize}
\begin{rem}
It can be shown through a rather lengthy (but straightforward) computation that the `algebraic' construction \eqref{eq:direct construction of hexagonator} is recovered from the 2-holonomy construction \eqref{eqn:R under limit} by substituting the globularity derivations of the 2-holonomies of the six vertically-interpolative 2-paths. In other words, the 2-holonomy construction \eqref{eqn:R under limit} is indeed convergent and can be seen as a geometric origin of the algebraic construction.
\begin{enumerate}
\item There might be confusion as to why we have provided this geometric origin at all if we could have merely produced the latter algebraic construction but one need only recall that, back in the 1-categorical context, proving that the hexagon axiom holds for Drinfeld's KZ associator series is far easier via the methods of parallel transport than by analysing the explicit series formula for $\Phi_\KZ(A,B)$ in Remark \ref{rem:LM Drinfeld series} or \eqref{subeq:BRW's explicit Phi formula}. To this point, we have already seen that modifications like $\int\{\Phi_{213}-\Phi(t_{12},\overline{t_{13}})\}$ have an unfeasibly large series formula \eqref{eq:mod from phi_213} thus problematising a direct proof that the Breen polytope axiom \eqref{eqn:Breen axiom} holds. Instead, our proof in Construction \ref{con:the Breen 2-loop} mirrors the idea from the 1-categorical context of building a loop whose sides yield the requisite terms. 
\item Given that we can derive the 2-holonomies through the globularity condition, it might seem that the methods developed for vertically-interpolative 2-paths in Subsubsection \ref{subsub:1st 2path} are completely redundant. However, we will show in the next paper of this series that such methods are imperative when it comes to constructing the pentagonator series because the 2-holonomies which constitute such a pentagonator can \textit{not} be derived through their globularity nor is there a corresponding functional equation from which the pentagon equation is derived in \cite[Theorem 24]{BRW}. 
\end{enumerate}
\end{rem}

\newpage
\section{The Breen polytope axiom as a 2-loop}\label{sec:Breen}
Throughout this section we assume a coherent totally symmetric strict infinitesimal 2-braiding.
\subsection{The Breen equation}\label{subsec:Breen equation}
Recall the Breen polytope axiom \eqref{eqn:Breen axiom},
\begin{equation}\label{eq:Breen axiom index version}
\sigma_{\sigma_{12}}+\alpha_{321}\H^R_{213}\alpha_{213}\sigma_{12}-\alpha_{321}\sigma_{23}\alpha_{231}^{-1}\H^L+\alpha_{321}\sigma_{\sigma_{23}}\alpha+\H^L_{132}\alpha_{132}^{-1}\sigma_{23}\alpha=\sigma_{12}\alpha_{312}\H^R\alpha\,.
\end{equation}
\begin{constr}\label{con:derivation of Breen equation}
We wish to express \eqref{eq:Breen axiom index version} purely in terms of the infinitesimal 2-braiding $t$ and the right pre-hexagonator $\RR$. First let us make use of symmetry of $t$ to note that, for example,
\begin{subequations}
\begin{equation}
\alpha_{231}:\equiv \Phi(t_{23},t_{13})\qquad\implies\qquad\alpha_{231}^{-1}\equiv \Phi(t_{13},t_{23})\equiv:\Phi_{132}\quad.
\end{equation}
Second, let us make use of the total symmetry of $t$ to note that, for example,
\begin{equation}
\sigma_{\sigma_{12}}=\left(\gamma e^{i\pi t}\right)_{\gamma_{12}e^{i\pi t_{12}}}=\gamma_{(21)3}\left(e^{i\pi t}\right)_{\gamma_{12}e^{i\pi t_{12}}}=\gamma_{(21)3}\gamma_{12}\left(e^{i\pi t}\right)_{e^{i\pi t_{12}}}\,.
\end{equation}
where the first equality used the definition of the braiding, the second equality used the definition of the homotopy components of the vertical composition of pseudonatural transformations \eqref{eq:vercomp pseudos} together with the fact that the symmetric braiding $\gamma$ is a $\Ch_\bbC^{[-1,0]}$-\emph{natural} transformation and the third equality used the fact that the homotopy components of a pseudonatural transformation splits products \eqref{eqn:dubindex splits prods} together with the fact\footnote{Which itself comes from: the fact that identity pseudonatural transformations have trivial homotopy components, the total symmetry \eqref{left total symmetry} and, again, the homotopy components of the vertical composition of pseudonatural transformations \eqref{eq:vercomp pseudos}.} that $\left(e^{i\pi t}\right)_{\gamma_{12}}=0$. Third, we note that 
\begin{equation}
\H^R:=\gamma_{(12)3}\RR\qquad\implies\qquad\H_{213}^R=\gamma_{(21)3}\RR_{213}\quad,    
\end{equation}
and, similarly,
\begin{equation}
\H^L:=\gamma_{1(23)}\LL\qquad\implies\qquad\H_{132}^L=\gamma_{1(32)}\LL_{132}\quad.
\end{equation}
Fourth, we recall that $\gamma$ satisfies the hexagon axiom thus
\begin{equation}
\gamma_{(21)3}\gamma_{12}=\gamma_{12}\gamma_{(12)3}=\gamma_{12}\gamma_{13}\gamma_{23}=\gamma_{1(32)}\gamma_{23}=\gamma_{23}\gamma_{1(23)}\quad.
\end{equation}
Fifth, we strip away the appearances of the symmetric braiding and denote the \textbf{congruences} as $\underline{\L}:=\left(e^{i\pi t}\right)_{e^{i\pi t_{12}}}$ and $\underline{\R}:=\left(e^{i\pi t}\right)_{e^{i\pi t_{23}}}$,
\begin{align}
\underline{\L}+\Phi_{321}\RR_{213}\Phi_{213}e^{i\pi t_{12}}-\Phi_{321}e^{i\pi t_{23}}\Phi_{132}\LL+\Phi_{321}\underline{\R}\Phi +\LL_{132}\Phi_{231}e^{i\pi t_{23}}\Phi =e^{i\pi t_{12}}\Phi_{312}\RR\Phi\,.
\end{align}
Sixth, we side-by-side compare the pre-hexagonator series
\begin{alignat}{6}
\RR&:\Phi_{213}e^{i\pi t_{(12)3}}\Phi_{321}\Rrightarrow e^{i\pi t_{13}}\Phi_{231}e^{i\pi t_{23}}\quad&&,\\
\LL&:\Phi_{231}e^{i\pi t_{1(23)}}\Phi \Rrightarrow e^{i\pi t_{13}}\Phi_{213}e^{i\pi t_{12}}\quad&&,
\end{alignat}
hence, appealing to total symmetry of $t$, we are justified in choosing $\LL=\RR_{321}$ thus $\LL_{132}=\RR_{231}$ (cf. the relations in Example \ref{ex: diff cross mod n=2} for a totally symmetric $t$) whence 
\begin{align}\label{Breen polytope in t only}
\underline{\L}+\Phi_{321}\RR_{213}\Phi_{213}e^{i\pi t_{12}}-\Phi_{321}e^{i\pi t_{23}}\Phi_{132}\RR_{321}+\Phi_{321}\underline{\R}\Phi +\RR_{231}\Phi_{231}e^{i\pi t_{23}}\Phi =e^{i\pi t_{12}}\Phi_{312}\RR\Phi\,.
\end{align}
\end{subequations}
\end{constr}
We now refer to \eqref{Breen polytope in t only} as the \textbf{Breen equation} throughout.


\subsubsection{Coherent infinitesimal 2-braidings and coherent CMKZ 2-connections}\label{subsub:S_3 acts coherently}
The challenge at this point is to find a six-faced 2-loop in $\bbC^{\times\times}\times\bbC^{\times}$ such that we realise the LHS minus the RHS of the Breen equation \eqref{Breen polytope in t only} as the limit (as $\varepsilon$ tends to 0) of the 2-holonomy. The 2-connection $(\nabla,\Delta)$ being 2-flat will guarantee that this 2-holonomy is zero so long as we do not wrap around the singularities $z=0$, $z=1$ or $v=0$. As mentioned in Remark \ref{rem:pullback 2-connection is always 2-flat}, this 2-flatness property is not by itself sufficient and we must figure out what we are missing by making use of the coherency of $t$. Let us pullback the entire 2-connection $(\nabla,\Delta)$ along the transferred left actions of $\mathrm{S}_3$ (extending \eqref{subeq:S_3 right action on 1-form}):
\begin{subequations}
\begin{align}
\tau_{12}^*(\nabla,\Delta)&=\left(\nabla(t_{12}\,,t_{13}\,,t_{23})\,,\,2\left(\frac{\L}{zv}+\frac{-(\L+\R)}{(z-1)v}\right)\dd z\wedge\dd v\right)\quad&,\\
\tau_{23}^*(\nabla,\Delta)&=\left(\nabla(t_{13}\,,t_{23}\,,t_{12})\,,\,2\left(\frac{-(\L+\R)}{zv}+\frac{\R}{(z-1)v}\right)\dd z\wedge\dd v\right)\quad&,\\
\tau_{13}^*(\nabla,\Delta)&=\left(\nabla(t_{23}\,,t_{12}\,,t_{13})\,,\,2\left(\frac{\R}{zv}+\frac{\L}{(z-1)v}\right)\dd z\wedge\dd v\right)\quad&,\\
\tau_{(12)3}^*(\nabla,\Delta)&=\left(\nabla(t_{23}\,,t_{13}\,,t_{12})\,,\,2\left(\frac{\R}{zv}+\frac{-(\L+\R)}{(z-1)v}\right)\dd z\wedge\dd v\right)\quad&,\\
\tau_{1(23)}^*(\nabla,\Delta)&=\left(\nabla(t_{13}\,,t_{12}\,,t_{23})\,,\,2\left(\frac{-(\L+\R)}{zv}+\frac{\L}{(z-1)v}\right)\dd z\wedge\dd v\right)\quad&.
\end{align}
\end{subequations}
As mentioned below \eqref{subeq:S_3 right action on 1-form}, the right action permutes the indices of the $t$'s according to the inverse of the group element of $\mathrm{S}_3$ and symmetry of $t$ allows us to rewrite $t_{ji}=t_{ij}$ for $i<j$. The Breen equation \eqref{Breen polytope in t only} informs us to only take interest in the pullbacks according to $\tau_{12}$, $\tau_{13}$ and $\tau_{(12)3}$ so we use the relations of a totally symmetric $t$ in Example \ref{ex: diff cross mod n=2} to write:
\begin{subequations}\label{subeq:relevant 2-form pullbacks}
\begin{align}
\tau_{12}^*\Delta&=\,2\left(\frac{\L_{213}}{zv}+\frac{-(\L+\R)}{(z-1)v}\right)\dd z\wedge\dd v\quad&,\\
\tau_{13}^*\Delta&=\,2\left(\frac{\L_{321}}{zv}+\frac{\R_{321}}{(z-1)v}\right)\dd z\wedge\dd v\quad&,\\
\tau_{(12)3}^*\Delta&=\,2\left(\frac{\L_{231}}{zv}+\frac{-(\L+\R)}{(z-1)v}\right)\dd z\wedge\dd v\quad&.
\end{align}
\end{subequations}
Now recall, again from Example \ref{ex: diff cross mod n=2}, that coherency of $t$ implies $-(\L+\R)=\R_{213}$ and a totally symmetric coherent $t$ gives us $-(\L+\R)=\R_{231}$. Substituting these relations into the above pullbacks \eqref{subeq:relevant 2-form pullbacks} means that the right action of $\mathrm{S}_3$ on the 2-connection permutes the indices of \textit{both} the $t$'s and the 4-term relationators $\L/\R$ according to the inverse of the group element of $\mathrm{S}_3$. It is this \textit{conjunctive} right action that we refer to as a \textbf{coherent} right action on the 2-connection $(\nabla,\Delta)$. This property means that we can construct $\RR^\varepsilon_{213}$, $\RR^\varepsilon_{321}$ and $\RR^\varepsilon_{231}$ by considering the 2-holonomy of the 2-paths $\Q_{213}:=\tau_{12}\Q$, $\Q_{321}:=\tau_{13}\Q$ and $\Q_{312}:=\tau_{(12)3}\Q$, respectively. Actually, we know from \eqref{eqn:R^epsilon=2hol-Leibniz} that such pre-hexagonator series come with conjugations and an extra term. Subsection \ref{subsec:Breen equation as 2hol} is devoted to showing that all these conjugations and extra terms balance out.


\subsection{The Breen equation as a 2-holonomy}\label{subsec:Breen equation as 2hol}
As always, we are actually interested in the limit $\varepsilon\to0$ so we work under that limit in this subsection and suppress the notation $\lim_{\varepsilon\to0}$. \eqref{eqn:R under limit} reads as
\begin{equation}
\RR=\varepsilon^{-t_{13}}W^\Q \varepsilon^{t_{23}}-\Phi_{213}t_\varepsilon^{t_{12}}\Phi_{321}\quad.
\end{equation}
The Breen equation \eqref{Breen polytope in t only} features this term as
\begin{align}
e^{i\pi t_{12}}\Phi_{312}\RR\Phi =&e^{i\pi t_{12}}\Phi_{312}\varepsilon^{-t_{13}}W^\Q \varepsilon^{t_{23}}\Phi -e^{i\pi t_{12}}t_\varepsilon^{t_{12}}\quad.
\end{align}
We invoke both the total symmetry and coherency of $t$ in writing 
\begin{equation}
\RR_{213}=\varepsilon^{-t_{23}}W^{\Q_{213}}\varepsilon^{t_{13}}-\Phi t_\varepsilon^{t_{12}}\Phi_{312}\quad,
\end{equation}
and this term features in the Breen equation \eqref{Breen polytope in t only} as
\begin{equation}
\Phi_{321}\RR_{213}\Phi_{213}e^{i\pi t_{12}}=\Phi_{321}\varepsilon^{-t_{23}}W^{\Q_{213}}\varepsilon^{t_{13}}\Phi_{213}e^{i\pi t_{12}}-t_\varepsilon^{t_{12}}e^{i\pi t_{12}}\quad.
\end{equation}
The other two pre-hexagonator series terms feature as:
\begin{align}
\Phi_{321}e^{i\pi t_{23}}\Phi_{132}\RR_{321}&=\Phi_{321}e^{i\pi t_{23}}\left(\Phi_{132}\varepsilon^{-t_{13}}W^{\Q_{321}}\varepsilon^{t_{12}}-t_\varepsilon^{t_{23}}\Phi \right)\\
\RR_{231}\Phi_{231}e^{i\pi t_{23}}\Phi &=\left(\varepsilon^{-t_{12}}W^{\Q_{312}}\varepsilon^{t_{13}}\Phi_{231}-\Phi_{321}t_\varepsilon^{t_{23}}\right)e^{i\pi t_{23}}\Phi 
\end{align}
\begin{rem}
We use Remarks \ref{rem:explicit parallel transports} and \ref{rem:2nd explicit parallel transport} to write\footnote{For now, read terms like $W^{c_\mathbf{VI}[p_\mathbf{V}\circ\iota]\Q c_\mathbf{I}}$ as $W^{c_\mathbf{VI}}W^{[p_\mathbf{V}\circ\iota]}W^{\Q}W^{c_\mathbf{I}}$.}:
\begin{subequations}
\begin{align}
\varepsilon^{t_{12}}e^{i\pi t_{12}}\Phi_{312}\RR\Phi \varepsilon^{-t_{12}}&=W^{c_\mathbf{VI}[p_\mathbf{V}\circ\iota]\Q c_\mathbf{I}}-\varepsilon^{\ad_{t_{12}}}\left(e^{i\pi t_{12}}t_\varepsilon^{t_{12}}\right)\,,\\
\varepsilon^{t_{12}}\Phi_{321}\RR_{213}\Phi_{213}e^{i\pi t_{12}}\varepsilon^{-t_{12}}&=W^{p_\mathbf{I}\Q_{213}p_\mathbf{V}c_\mathbf{VI}}-\varepsilon^{\ad_{t_{12}}}\left(t_\varepsilon^{t_{12}}e^{i\pi t_{12}}\right)\,,\\
\varepsilon^{t_{12}}\Phi_{321}e^{i\pi t_{23}}\Phi_{132}\RR_{321}\varepsilon^{-t_{12}}&=W^{p_\mathbf{I}c_\mathbf{II}}\left(W^{[p_\mathbf{III}\circ\iota]\Q_{321}}-\varepsilon^{\ad_{t_{23}}}\left(t_\varepsilon^{t_{23}}\right)W^{c_\mathbf{I}}\right)\,,\\
\varepsilon^{t_{12}}\RR_{231}\Phi_{231}e^{i\pi t_{23}}\Phi \varepsilon^{-t_{12}}&=\left(W^{\Q_{312}p_\mathbf{III}}-W^{p_\mathbf{I}}\varepsilon^{\ad_{t_{23}}}\left(t_\varepsilon^{t_{23}}\right)\right)W^{c_\mathbf{II}c_\mathbf{I}}\,,\\
\varepsilon^{t_{12}}\underline{\L}\varepsilon^{-t_{12}}&=\varepsilon^{\ad_{t_{12}}}\left(\underline{\L}\right)\quad,\\
\varepsilon^{t_{12}}\Phi_{321}\underline{\R}\Phi \varepsilon^{-t_{12}}&=W^{p_\mathbf{I}}\varepsilon^{\ad_{t_{23}}}\left(\underline{\R}\right)W^{c_\mathbf{I}}\quad.
\end{align}
\end{subequations}
The Breen equation \eqref{Breen polytope in t only} conjugated by $\varepsilon^{t_{12}}$ now gives us
\begin{align}
&\varepsilon^{\ad_{t_{12}}}\left(\underline{\L}+\big[e^{i\pi t_{12}},t_\varepsilon^{t_{12}}\big]\right)+W^{p_\mathbf{I}\Q_{213}p_\mathbf{V}c_\mathbf{VI}}-W^{p_\mathbf{I}c_\mathbf{II}[p_\mathbf{III}\circ\iota]\Q_{321}}\label{eqn:suggestive Breen}\\&\qquad\qquad+W^{p_\mathbf{I}}\varepsilon^{\ad_{t_{23}}}\left(\underline{\R}+\big[e^{i\pi t_{23}},t_\varepsilon^{t_{23}}\big]\right)W^{c_\mathbf{I}}+W^{\Q_{312}p_\mathbf{III}c_\mathbf{II}c_\mathbf{I}}=W^{c_\mathbf{VI}[p_\mathbf{V}\circ\iota]\Q c_\mathbf{I}}\quad.\nn
\end{align}
\end{rem}
We now wish to show that we have:
\begin{equation}\label{eqn:rel between L and Adt_12}
\varepsilon^{\ad_{t_{12}}}\left(\underline{\L}+\big[e^{i\pi t_{12}},t_\varepsilon^{t_{12}}\big]\right)=\underline{\L}\qquad,\qquad \varepsilon^{\ad_{t_{23}}}\left(\underline{\R}+\big[e^{i\pi t_{23}},t_\varepsilon^{t_{23}}\big]\right)=\underline{\R}\quad,
\end{equation}
because Subsection \ref{sub:2-loop} will encode the congruences $\underline{\mathcal{L/R}}$ as 2-holonomies of particular 2-paths and then the above suggestive form \eqref{eqn:suggestive Breen} makes it clear where to begin with the construction of the desired 2-loop in $\bbC^{\times\times}\times\bbC^\times$. We already have an explicit series formula for the term $t_\varepsilon^{t_{12}}$ as in \eqref{eqn:def Ad} but, to show \eqref{eqn:rel between L and Adt_12} holds, we also need an explicit series formula for the congruence $\underline{\L}$.


\subsubsection{The congruences as series formulae}
\begin{align}
\left[e^{i\pi t}\right]_{e^{i\pi t_{12}}}&=\sum_{m=1}^\infty\frac{(i\pi)^m}{m!}\left[t^m\right]_{e^{i\pi t_{12}}}\nn\\&=\sum_{m=1}^\infty\frac{(i\pi)^m}{m!}\sum_{j=0}^{m-1}t_{(12)3}^jt_{e^{i\pi t_{12}}}t_{(12)3}^{m-1-j}\nn\\&=\sum_{m=1}^\infty\sum_{n=1}^\infty\frac{(i\pi)^{m+n}}{m!n!}\sum_{j=0}^{m-1}t_{(12)3}^jt_{t_{12}^n}t_{(12)3}^{m-1-j}\nn\\&=\sum_{\begin{smallmatrix}m=1\\n=1
\end{smallmatrix}}^\infty\frac{(i\pi)^{m+n}}{m!n!}\sum_{\begin{smallmatrix}0\leq k\leq n-1\\0\leq j\leq m-1
\end{smallmatrix}}t_{(12)3}^jt_{12}^k\L t_{12}^{n-1-k}t_{(12)3}^{m-1-j}\quad.\label{eqn:e^L series}
\end{align}
where the first equality used the fact that the identity pseudonatural transformation has trivial homotopy components (see Construction \ref{con:vert comp pseudo}) in conjunction with the linearity of pseudonatural transformations \eqref{eq:add pseudos and mods}, the second equality used the formula for the homotopy components of the vertical composition of pseudonatural transformations \eqref{eq:vercomp pseudos}, the third equality used the linearity of the homotopy components (see Definition \ref{def:pseudonatural}) in conjunction with the fact that they annihilate units \eqref{eqn: homotopy kills unit} and the last equality used the fact that the homotopy components split products \eqref{eqn:dubindex splits prods} together with the definition of the left 4-term relationator $\L:=t_{t_{12}}$.
\begin{lem}\label{lem:rel between Ad and e^L}
The relation \eqref{eqn:rel between L and Adt_12} holds.
\end{lem}
\begin{proof}
The derivation \eqref{eqn:e^L series} makes it obvious that 
\begin{equation}
t^{t_{12}}_\varepsilon=\left(e^{i\pi t}\right)_{\varepsilon^{-t_{12}}}\varepsilon^{t_{12}}
\end{equation}
hence \eqref{eqn:rel between L and Adt_12} is equivalent to
\begin{equation}
[e^{i\pi t}]_{e^{i\pi t_{12}}}\varepsilon^{-t_{12}}+e^{i\pi t_{12}}[e^{i\pi t}]_{\varepsilon^{-t_{12}}}=\varepsilon^{-t_{12}}[e^{i\pi t}]_{e^{i\pi t_{12}}}+[e^{i\pi t}]_{\varepsilon^{-t_{12}}}e^{i\pi t_{12}}
\end{equation}
but, again, we can use the fact that the homotopy components split products \eqref{eqn:dubindex splits prods} to see that each side is equal to $(e^{i\pi t})_{e^{(i\pi-\ln\varepsilon) t_{12}}}$.
\end{proof}


\subsection{The 2-loop}\label{sub:2-loop}
Lemma \ref{lem:rel between Ad and e^L} (and an obvious analogue for the right congruence $\underline{\R}$) permits us to substitute \eqref{eqn:rel between L and Adt_12} into \eqref{eqn:suggestive Breen} giving us
\begin{align}\label{eqn:Breen as 2hol}
\underline{\L}+W^{p_\mathbf{I}\Q_{213}p_\mathbf{V}c_\mathbf{VI}}-W^{p_\mathbf{I}c_\mathbf{II}[p_\mathbf{III}\circ\iota]\Q_{321}}+W^{p_\mathbf{I}}\underline{\R}W^{c_\mathbf{I}}+W^{\Q_{312}p_\mathbf{III}c_\mathbf{II}c_\mathbf{I}}=W^{c_\mathbf{VI}[p_\mathbf{V}\circ\iota]\Q c_\mathbf{I}}
\end{align}

\subsubsection{The 2-paths for the congruences}
Let us show how to encode the congruences, $\underline{\L}$ and $\underline{\R}$, as 2-holonomies of specific 2-paths. We treat the left congruence $\underline{\L}$ which, by comparing its series expression \eqref{eqn:e^L series} with that of $W^{\P_\mathbf{IV}}$ in Remark \ref{rem:1st horinterpol glob condition}, is easily seen as the following specific instance of the exchanger\footnote{Cf. \cite[Construction 2.11, Proposition 5.3 and Remark 5.4]{Me}.}, 
\begin{equation}\label{left congruence}
\left(e^{i\pi t}\right)_{e^{i\pi t_{12}}}=\ast^2_{\left(\Id_{\otimes\,},\,e^{i\pi t}\,\boxtimes\,\Id_{\id_\CC}\right),\left(e^{i\pi t},\,\Id_{\otimes\,\boxtimes\,\id_\CC}\right)}:e^{i\pi t_{12}}e^{i\pi t_{(12)3}}\Rrightarrow e^{i\pi t_{(12)3}}e^{i\pi t_{12}}\quad.
\end{equation}
Now we have a hint of how to encode such a congruence as a 2-holonomy. Recall that \eqref{eq:8 of first partras} and \eqref{eq:W^q_VI}, respectively, gives us:
\begin{equation}
W^{c_\mathbf{VI}}\equiv e^{i\pi t_{12}}H_\varepsilon\left(t_{12}\,,\overline{t_{13}}\right)\qquad,\qquad W^{q_\mathbf{VI}}\equiv H_\varepsilon\left(\overline{t_{12}},t_{13}\right)^{-1}e^{i\pi(t_{13}+t_{23})}\quad.
\end{equation}
Remark \ref{rem:2nd explicit parallel transport} mentions the parametric-translation invariance of parallel transport along $c_\mathbf{II}$; this property obviously also holds for $c_\mathbf{VI}$. Define
\begin{subequations}
\begin{equation}
q_\searrow(r):=\left(\frac{2\varepsilon}{\varepsilon+(\varepsilon-2)e^{i\pi r}}\,,\,\left[\frac{\varepsilon}{2}+\left(1-\frac{\varepsilon}{2}\right)e^{i\pi r}\right]a\right)
\end{equation}
along which the connection 1-form $\nabla$ is given by 
\begin{align}
\nabla_{q_\searrow}\equiv i\pi(t_{23}+t_{13})\dd r+\varepsilon i\pi\left[\frac{t_{12}+t_{23}}{\varepsilon+(\varepsilon-2)e^{i\pi r}}-\frac{t_{12}+2t_{23}+t_{13}}{\varepsilon+(2-\varepsilon)e^{i\pi r}}\right]\dd r\quad,
\end{align}
and a similar argument to that of \cite[Lemma 21]{BRW} tells us that
\begin{equation}
W^{q_\searrow}\equiv e^{i\pi(t_{13}+t_{23})}H_\varepsilon(t_{12}+t_{23},t_{23}+t_{13})\quad.    
\end{equation}
\end{subequations}
\begin{propo}\label{propo:left cong 2-hol}
The 2-path $\big(c_\mathbf{VI},(\varepsilon-1)a\big)\,q_\mathbf{VI}\xRightarrow{\P_\L}q_\searrow\big(c_\mathbf{VI}(r+1),a\big)$ given by 
\begin{equation}\label{eqn:left cong 2-path}
\P_\L(s,r):=\begin{cases}
      q_\mathbf{VI}(2r)\,, & 0\leq r\leq\frac{1-s}{2}\\
      \left(c_\mathbf{VI}(2r+2s-1),c_\mathbf{IV}(s-1)\varepsilon a\right)\,, & \frac{1-s}{2}\leq r\leq1-\frac{s}{2}\\
      q_\searrow(2r-1)\,, & 1-\frac{s}{2}\leq r\leq1
    \end{cases}\quad,
\end{equation}
\begin{figure}[H]
\centering
\scalebox{0.2}{\includesvg[width=1000pt]{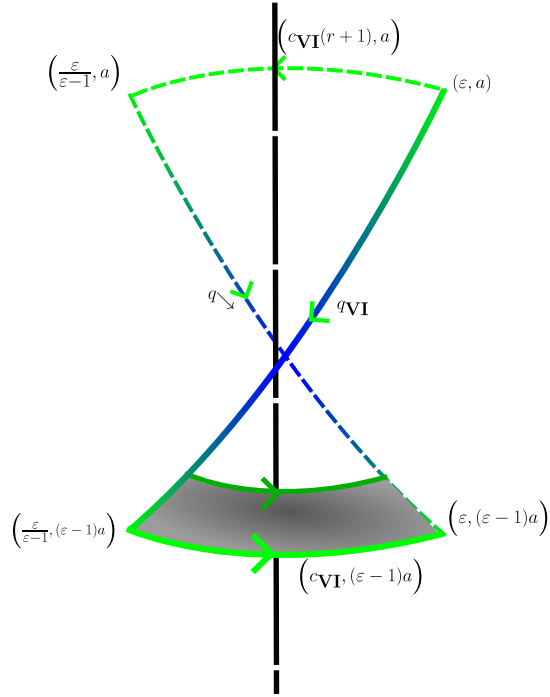}}
\caption{The horizontally-interpolative 2-path for the left congruence}
\label{fig:congruence2path}
\end{figure}
is such that
\begin{equation}
    \lim_{\varepsilon\to0}W^{\P_\L}=\underline{\L}\quad.
\end{equation}
\end{propo}
\begin{proof}
The 2-holonomy is given, as always, by
\begin{align}
W^{\P_\L}=\int_0^1\int_\frac{1-s}{2}^{1-\frac{s}{2}}W_{1r}^{\P_\L^s}\Delta\left[\frac{\partial\P_\L}{\partial s},\frac{\partial\P_\L}{\partial r}\right]W_{r0}^{\P_\L^s}\dd r\dd s\quad.
\end{align}
We have, for $\frac{1-s}{2}\leq r\leq1-\frac{s}{2}$:
\begin{subequations}
\begin{alignat}{6}
z&=c_\mathbf{VI}(2r+2s-1)=\frac{2\varepsilon}{\varepsilon+(2-\varepsilon)e^{-2i\pi(s+r)}}&&\implies\frac{\partial z}{\partial r}=\frac{4i\pi\varepsilon(2-\varepsilon)e^{-2i\pi(s+r)}}{[\varepsilon+(2-\varepsilon)e^{-2i\pi(s+r)}]^2}\,\,&&&,\\
v&=c_\mathbf{IV}(s-1)\varepsilon a=\left[\frac{\varepsilon}{2}+\left(\frac{\varepsilon}{2}-1\right)e^{-i\pi s}\right]a&&\implies\frac{\partial v}{\partial s}=i\pi\left(1-\frac{\varepsilon}{2}\right)e^{-i\pi s}a\,\,&&&.
\end{alignat}
\end{subequations}
Together this gives us
\begin{align}
\Delta\left[\frac{\partial\P_\L}{\partial s},\frac{\partial\P_\L}{\partial r}\right]=\frac{[i\pi(2-\varepsilon)]^2}{[\frac{\varepsilon}{2}(e^{is\pi}+1)-1][\varepsilon(1-e^{-2i\pi(s+r)})-2]}\left[\L+\frac{2\varepsilon\R}{\varepsilon+(\varepsilon-2)e^{-2i\pi(s+r)}}\right]\,,
\end{align}
which has the very nice and simple limit 
\begin{equation}
\lim_{\varepsilon\to0}\left(\Delta\left[\frac{\partial\P_\L}{\partial s},\frac{\partial\P_\L}{\partial r}\right]\right)=2(i\pi)^2\L\quad.
\end{equation}
For $\frac{1-s}{2}\leq r\leq1-\frac{s}{2}$, we also have:
\begin{subequations}
\begin{alignat}{6}
W_{r0}^{\P_\L^s}&\equiv W_{(2r+2s-1)s}^{c_\mathbf{VI}}W_{(1-s)0}^{q_\mathbf{VI}}\quad&&\implies\quad\lim_{\varepsilon\to0}\left(W_{r0}^{\P_\L^s}\right)\equiv e^{i\pi(2r+s-1)t_{12}}e^{i\pi(1-s)t_{(12)3}}\,&&&,\\
W_{1r}^{\P_\L^s}&\equiv W_{1(1-s)}^{q_\searrow}W_{(1+s)s}^{c_\mathbf{VI}}W_{(1-s)0}^{q_\mathbf{VI}}\quad&&\implies\quad\lim_{\varepsilon\to0}\left(W_{1r}^{\P_\L^s}\right)\equiv e^{i\pi st_{(12)3}}e^{i\pi(2-2r-s)t_{12}}\,&&&.
\end{alignat}
\end{subequations}
Now we can say 
\begin{align}
\lim_{\varepsilon\to0}\left(W^{\P_\L}\right)=&\int_0^1i\pi e^{i\pi st_{(12)3}}e^{i\pi t_{12}}\left(\int_\frac{1-s}{2}^{1-\frac{s}{2}}2i\pi e^{-i\pi(2r+s-1)t_{12}}\L e^{i\pi(2r+s-1)t_{12}}\dd r\right)e^{i\pi(1-s)t_{(12)3}}\dd s\nn\\
=&\int_0^1i\pi e^{i\pi st_{(12)3}}\sum_{n=1}^\infty\frac{(i\pi)^n}{n!}\sum_{k=0}^{n-1}t_{12}^k\L t_{12}^{n-1-k}e^{i\pi(1-s)t_{(12)3}}\dd s\nn\\
=&e^{i\pi t_{(12)3}}\int_0^1i\pi e^{-i\pi st_{(12)3}}\sum_{n=1}^\infty\frac{(i\pi)^n}{n!}\sum_{k=0}^{n-1}t_{12}^k\L t_{12}^{n-1-k}e^{i\pi st_{(12)3}}\dd s\nn\\
=&\sum_{m,n\in\bbN^*}\frac{(i\pi)^{m+n}}{m!n!}\sum_{j=0}^{m-1}\sum_{k=0}^{n-1}t_{(12)3}^jt_{12}^k\L t_{12}^{n-1-k}t_{(12)3}^{m-1-j}\quad,
\end{align}
where the $2^\mathrm{nd}$ and $4^\mathrm{th}$ equalities used arguments completely analogous to that of Lemma \ref{lem:solve 1st integral} whereas the $3^\mathrm{rd}$ equality used the substitution $s'=1-s$.
\end{proof}

\subsubsection*{The Breen equation as an equality of two sequences of modifications}
Let us interpret each side of the Breen equation \eqref{Breen polytope in t only} as a sequence of lateral composites of modifications,
\begin{equation}
\begin{tikzcd}
	{e^{i\pi t_{(12)3}}e^{i\pi t_{12}}} \\
	{\Phi_{321}e^{i\pi t_{23}}\Phi_{132}e^{i\pi t_{13}}\Phi_{213}e^{i\pi t_{12}}} & {e^{i\pi t_{12}}e^{i\pi t_{(12)3}}} \\
	{\Phi_{321}e^{i\pi t_{23}}e^{i\pi t_{1(23)}}\Phi } & {e^{i\pi t_{12}}\Phi_{312}e^{i\pi t_{13}}\Phi_{231}e^{i\pi t_{23}}\Phi } \\
	{\Phi_{321}e^{i\pi t_{1(23)}}e^{i\pi t_{23}}\Phi }
	\arrow["{\Phi_{321}\mathbf{R}_{213}\Phi_{213}e^{i\pi t_{12}}}"', Rightarrow, scaling nfold=3, from=1-1, to=2-1]
	\arrow["{\Phi_{321}e^{i\pi t_{23}}\Phi_{132}\overleftarrow{\mathbf{R}_{321}}}"', Rightarrow, scaling nfold=3, from=2-1, to=3-1]
	\arrow["{\underline{\mathcal{L}}}"', Rightarrow, scaling nfold=3, from=2-2, to=1-1]
	\arrow["{e^{i\pi t_{12}}\Phi_{312}\mathbf{R}\Phi }", Rightarrow, scaling nfold=3, from=2-2, to=3-2]
	\arrow["{\Phi_{321}\underline{\mathcal{R}}\Phi }"', Rightarrow, scaling nfold=3, from=3-1, to=4-1]
	\arrow["{\mathbf{R}_{231}\Phi_{231}e^{i\pi t_{23}}\Phi }"', Rightarrow, scaling nfold=3, from=4-1, to=3-2]
\end{tikzcd}
\end{equation}
We already know the 2-paths which encode the congruences and pre-hexagonator but for the permutations such as $\RR_{213}$ we must calculate the corresponding 2-path. We start with $\tau_{12}\Q$:
\begin{subequations}
\begin{alignat}{6}
\tau_{12}p_\mathbf{I}(r)&=\left(\frac{\varepsilon-1+r(1-2\varepsilon)}{\varepsilon+r(1-2\varepsilon)},[\varepsilon+r(1-2\varepsilon)]a\right)\quad&&,\\
\tau_{12}p_\mathbf{II}(r)&=\left(\frac{1}{2}+\left(\frac{1}{2}-\frac{1}{\varepsilon}\right)e^{-i\pi r},\frac{a}{\left(\frac{1}{\varepsilon}-\frac{1}{2}\right)e^{-i\pi r}+\frac{1}{2}}\right)\quad&&,\\
\tau_{12}p_\mathbf{III}(r)&=\left(\frac{1}{\varepsilon+r(1-2\varepsilon)},\frac{[\varepsilon+r(1-2\varepsilon)]a}{\varepsilon-1}\right)\quad&&,\\
\tau_{12}q_\mathbf{IV}(r)&=\left(1-\frac{\varepsilon}{\frac{\varepsilon}{2}-\left(1-\frac{\varepsilon}{2}\right)e^{-i\pi r}}\,,\,-a\right)\quad&&,\\
\tau_{12}q_\mathbf{V}(r)&=\big(\varepsilon+r(1-2\varepsilon),-a\big)\quad&&,\\
\tau_{12}q_\mathbf{VI}(r)&=\left(\frac{\varepsilon}{\frac{\varepsilon}{2}+\left(\frac{\varepsilon}{2}-1\right)e^{i\pi r}},\left[\left(1-\frac{\varepsilon}{2}\right)e^{i\pi r}-\frac{\varepsilon}{2}\right]a\right)\quad&&.
\end{alignat}
\end{subequations}
\begin{figure}[H]
\centering
\scalebox{0.35}{\includesvg[width=1000pt]{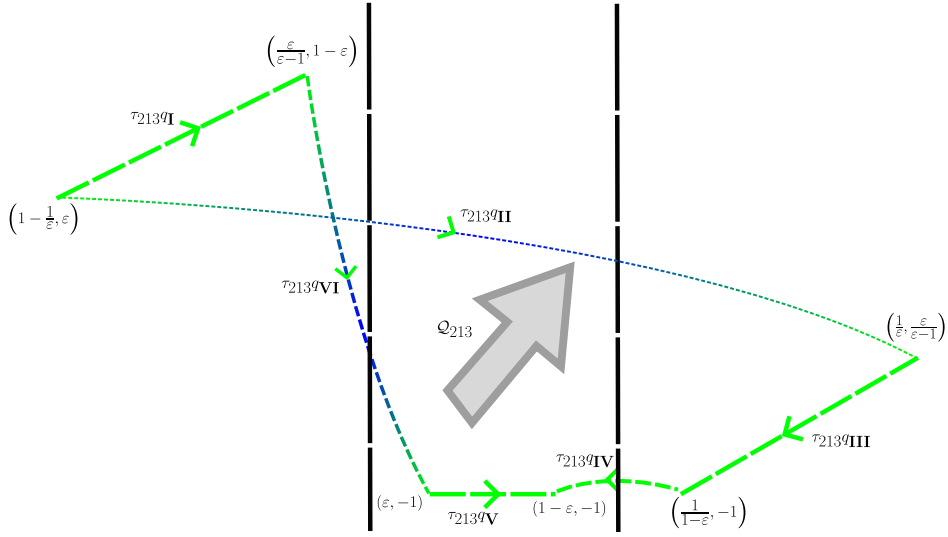}}
\caption{$\Q_{213}:=\tau_{12}\Q$ with $a=1$}
\label{fig:213}
\end{figure}
\begin{subequations}
\begin{alignat}{6}
\tau_{13}p_\mathbf{I}(r)&=\big(\varepsilon+r(1-2\varepsilon),-a\big)\quad&&,\\
\tau_{13}p_\mathbf{II}(r)&=\left(\frac{\varepsilon}{\left(\frac{1}{\varepsilon}-\frac{1}{2}\right)e^{-i\pi r}+\frac{1}{2}}\,,\,-a\right)\quad&&,\\
\tau_{13}p_\mathbf{III}(r)&=\left(\frac{-\varepsilon+r(2\varepsilon-1)}{(1-r)(1-\varepsilon)+r\varepsilon},\frac{r\varepsilon a}{\varepsilon-1}+(r-1)a\right)\quad&&,\\
\tau_{13}q_\mathbf{IV}(r)&=\left(\frac{1}{2}+\left(\frac{1}{2}-\frac{1}{\varepsilon}\right)e^{-i\pi r},\frac{-a\varepsilon}{(1-\frac{\varepsilon}{2})e^{-i\pi r}-\frac{\varepsilon}{2}}\right)\quad&&,\\
\tau_{13}q_\mathbf{V}(r)&=\left(\frac{1}{1-\varepsilon+r(2\varepsilon-1)},[1-\varepsilon+r(2\varepsilon-1)]a\right)\quad&&,\\
\tau_{13}q_\mathbf{VI}(r)&=\left(1+\frac{\varepsilon}{\left(\frac{\varepsilon}{2}-1\right)e^{i\pi r}-\frac{\varepsilon}{2}},\left[\left(\frac{\varepsilon}{2}-1\right)e^{i\pi r}-\frac{\varepsilon}{2}\right]a\right)\quad&&.
\end{alignat}
\end{subequations}
\begin{figure}[H]
\centering
\scalebox{0.4}{\includesvg[width=1000pt]{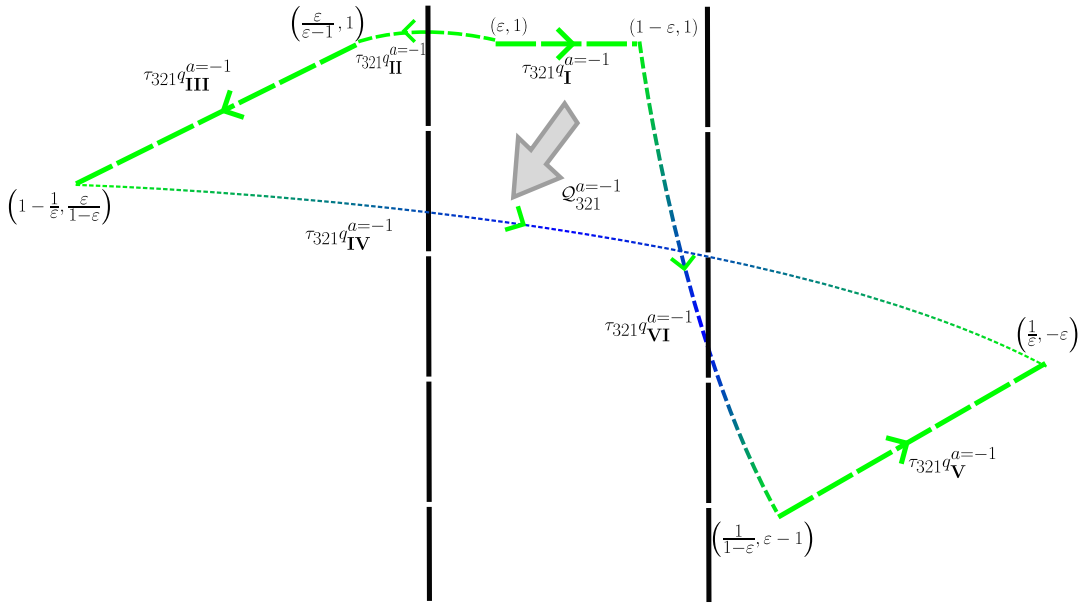}}
\caption{$\Q_{321}:=\tau_{13}\Q$ with $a=-1$}
\label{fig:321}
\end{figure}
\begin{subequations}
\begin{alignat}{6}
\tau_{(12)3}p_\mathbf{I}(r)&=\left(\frac{1}{\varepsilon+r(1-2\varepsilon)},[(2\varepsilon-1)r-\varepsilon]a\right)\quad&&,\\
\tau_{(12)3}p_\mathbf{II}(r)&=\left(\frac{1}{2}+\left(\frac{1}{\varepsilon}-\frac{1}{2}\right)e^{-i\pi r},\frac{a}{\left(\frac{1}{2}-\frac{1}{\varepsilon}\right)e^{-i\pi r}-\frac{1}{2}}\right)\quad&&,\\
\tau_{(12)3}p_\mathbf{III}(r)&=\left(\frac{(1-r)(1-\varepsilon)+r\varepsilon}{-\varepsilon+r(2\varepsilon-1)},\frac{[\varepsilon+r(1-2\varepsilon)]a}{1-\varepsilon}\right)\quad&&,\\
\tau_{(12)3}q_\mathbf{IV}(r)&=\left(\frac{\varepsilon}{\frac{\varepsilon}{2}+\left(\frac{\varepsilon}{2}-1\right)e^{-i\pi r}}\,,\,a\right)\quad&&,\\
\tau_{(12)3}q_\mathbf{V}(r)&=\left(1-\varepsilon+r(2\varepsilon-1),a\right)\quad&&,\\
\tau_{(12)3}q_\mathbf{VI}(r)&=\left([c_\mathbf{II}\circ\iota](r),\left[\left(\frac{\varepsilon}{2}-1\right)e^{i\pi r}+\frac{\varepsilon}{2}\right]a\right)\quad&&.
\end{alignat}
\end{subequations}
\begin{figure}[H]
    \centering
    \scalebox{0.33}{\includesvg[width=1000pt]{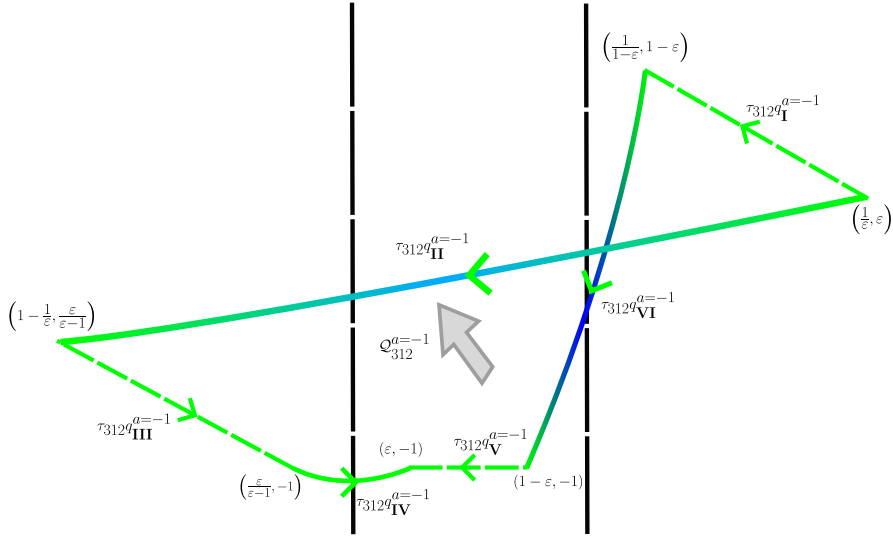}}
    \caption{$\Q_{312}:=\tau_{(12)3}\Q$ with $a=-1$}
    \label{fig:312}
\end{figure}
\begin{constr}\label{con:the Breen 2-loop}
\begin{enumerate}
\item We begin with the 1-path $(c_\mathbf{VI},\varepsilon-1)\,q_\mathbf{VI}$ and apply $\P_\L$ to get to the 1-path $q_\searrow\big(c_\mathbf{VI}(r+1),1\big)$ picking up $\underline{\L}$ in the limit $\varepsilon\to0$.
\begin{figure}[H]
    \centering
    \scalebox{0.13}{\includesvg[width=1000pt]{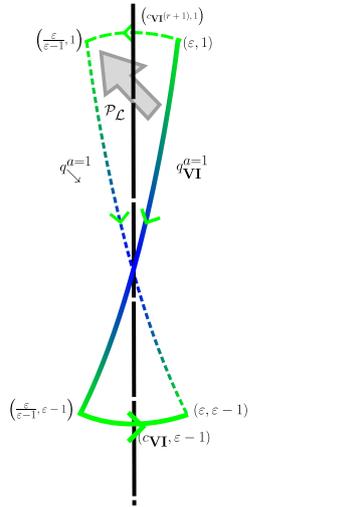}}
    \caption{Left congruence with $a=1$}
    \label{fig:leftor}
\end{figure}
\item We modify the 1-path $q_\searrow$ to the 1-path $(\varepsilon,r\varepsilon-1)\,\tau_{12}q_\mathbf{VI}\,\big(\frac{\varepsilon}{\varepsilon-1},1-r\varepsilon\big)$ as in
\begin{figure}[H]
    \centering
    \scalebox{0.4}{\includesvg[width=1000pt]{whiskered213}}
    \caption{A harmless 2-path followed by the whiskering of $\Q_{213}$}
    \label{fig:whiskered213}
\end{figure}
Such a 2-path gives trivial 2-holonomy under $\lim_{\varepsilon\to0}$\footnote{Whenceforth we call these \textbf{harmless 2-paths}.}. Modifying the latter 1-path to
\begin{equation}
(\varepsilon,r\varepsilon-1)\,[\tau_{12}q_\mathbf{V}\circ\iota]\,\tau_{12}q_\mathbf{V}\,\tau_{12}q_\mathbf{VI}\,\tau_{12}p_\mathbf{I}\,[\tau_{12}p_\mathbf{I}\circ\iota]\,\left(\frac{\varepsilon}{\varepsilon-1},1-r\varepsilon\right)
\end{equation}
obviously gives trivial 2-holonomy. Now we modify to the 1-path
\begin{equation}
(\varepsilon,r\varepsilon-1)[\tau_{12}q_\mathbf{V}\circ\iota]\,\tau_{12}q_\mathbf{IV}\,\tau_{12}p_\mathbf{III}\,\tau_{12}p_\mathbf{II}\,[\tau_{12}p_\mathbf{I}\circ\iota]\left(\frac{\varepsilon}{\varepsilon-1},1-r\varepsilon\right)\big(c_\mathbf{VI}(r+1),1\big),
\end{equation}
picking up (under $\lim_{\varepsilon\to0}$, we suppress this qualifier from now on) $W^{p_\mathbf{I}}W^{\Q_{213}}W^{p_\mathbf{V}}W^{c_\mathbf{VI}}$.
\item We modify to the 1-path $\big(c_\mathbf{I}\circ\iota,\varepsilon-1\big)\big(c_\mathbf{II}(r+1),\varepsilon-1\big)[\tau_{13}q_\mathbf{V}^{a=-1}\circ\iota]\tau_{13}q_\mathbf{IV}^{a=-1}\tau_{13}p_\mathbf{III}^{a=-1}\tau_{13}p_\mathbf{II}^{a=-1}$ along a harmless 2-path as in
\begin{figure}[H]
    \centering
    \scalebox{0.4}{\includesvg[width=1000pt]{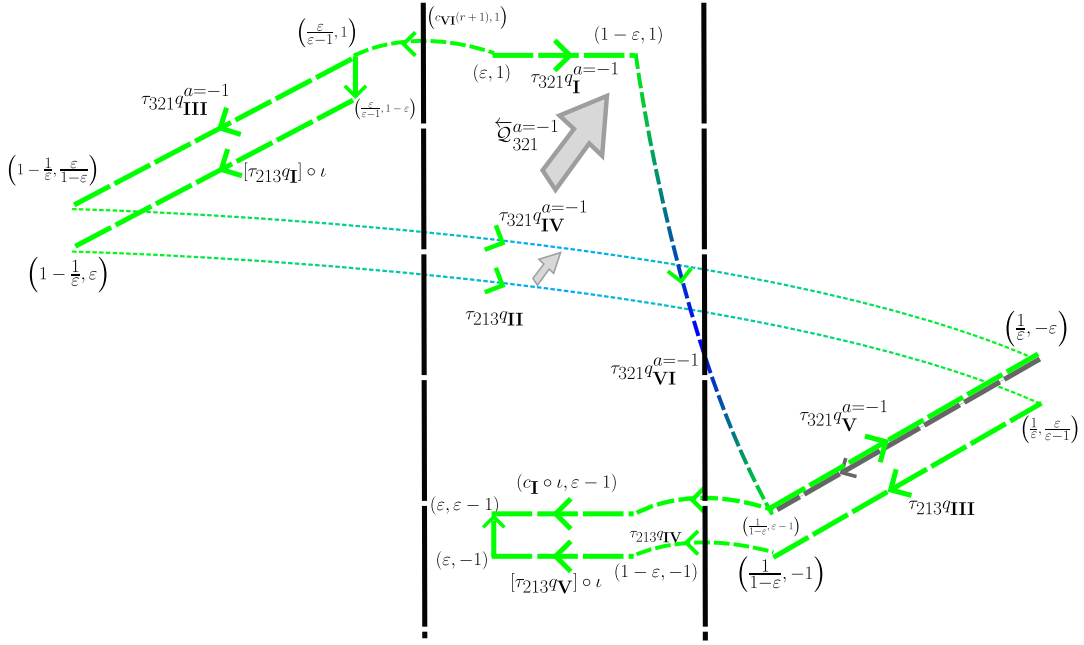}}
    \caption{Harmless 2-path followed by whiskered $\Q_{321}$ and collapsing a loop}
    \label{fig:whiskered321}
\end{figure}
We modify to the 1-path $\big(c_\mathbf{I}\circ\iota,\varepsilon-1\big)\big(c_\mathbf{II}(r+1),\varepsilon-1\big)[\tau_{13}q_\mathbf{V}^{a=-1}\circ\iota]\tau_{13}q_\mathbf{V}^{a=-1}\tau_{13}q_\mathbf{VI}^{a=-1}\tau_{13}p_\mathbf{I}^{a=-1}$ picking up 2-holonomy $-W^{p_\mathbf{I}}e^{i\pi t_{23}}\left(W^{p_\mathbf{III}}\right)^{-1}W^{\Q_{321}}$. We collapse $[\tau_{13}q_\mathbf{V}^{a=-1}\circ\iota]\tau_{13}q_\mathbf{V}^{a=-1}$ to the constant loop on $\left(\frac{1}{1-\varepsilon},\varepsilon-1\right)$ picking up trivial 2-holonomy.
\item We modify to the 1-path $\big(c_\mathbf{I}\circ\iota,\varepsilon-1\big)\,q_\swarrow\big(c_\mathbf{II},1\big)\tau_{13}p_\mathbf{I}^{a=-1}$ picking up $W^{p_\mathbf{I}}\underline{\R}W^{c_\mathbf{I}}$ as in
\begin{figure}[H]
    \centering
    \scalebox{0.15}{\includesvg[width=1000pt]{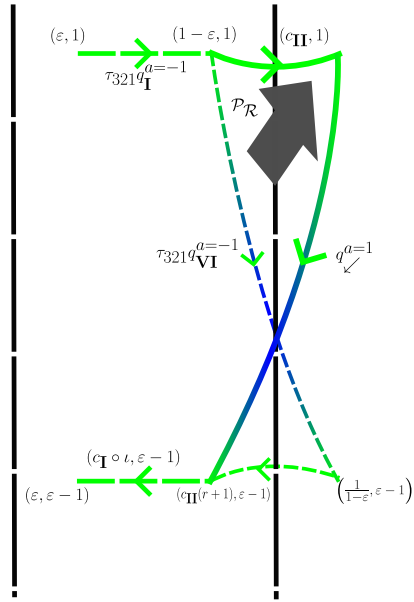}}
    \caption{Whiskered right congruence with $a=1$}
    \label{fig:rightor}
\end{figure}
\item We modify the 1-path $q_\swarrow$ to the 1-path $(\varepsilon,r\varepsilon-1)\tau_{(12)3}q_\mathbf{V}^{a=-1}\tau_{(12)3}q_\mathbf{VI}^{a=-1}\left(\frac{1}{1-\varepsilon},1-r\varepsilon\right)$ along a harmless 2-path as in
\begin{figure}[H]
    \centering
    \scalebox{0.4}{\includesvg[width=1000pt]{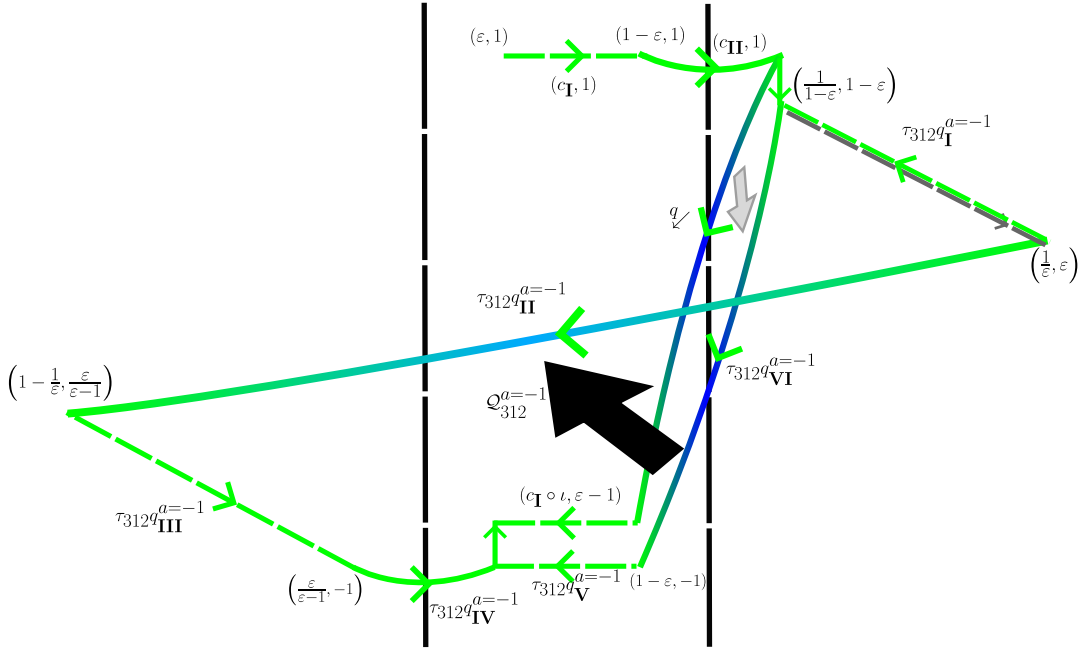}}
    \caption{Harmless 2-path followed by whiskered $\Q_{312}$}
    \label{fig:whiskered312}
\end{figure}
We modify $(\varepsilon,r\varepsilon-1)\tau_{(12)3}q_\mathbf{V}^{a=-1}\tau_{(12)3}q_\mathbf{VI}^{a=-1}\tau_{(12)3}p_\mathbf{I}^{a=-1}[\tau_{(12)3}p_\mathbf{I}^{a=-1}\circ\iota]\left(\frac{1}{1-\varepsilon},1-r\varepsilon\right)(c_\mathbf{II}c_\mathbf{I},1)$ to $(\varepsilon,r\varepsilon-1)\tau_{(12)3}q_\mathbf{IV}^{a=-1}\tau_{(12)3}p_\mathbf{III}^{a=-1}\tau_{(12)3}p_\mathbf{II}^{a=-1}[\tau_{(12)3}p_\mathbf{I}^{a=-1}\circ\iota]\left(\frac{1}{1-\varepsilon},1-r\varepsilon\right)(c_\mathbf{II}c_\mathbf{I},1)$ picking up the 2-holonomy $W^{\Q_{312}}W^{p_\mathbf{III}}e^{i\pi t_{23}}W^{c_\mathbf{I}}$.
\sk

Finally, we compare the sum of those previous 5 steps to the one where we simply modify the 1-path $(c_\mathbf{VI},\varepsilon-1)\,q_\mathbf{VI}$ to the 1-path $(c_\mathbf{VI},\varepsilon-1)[q_\mathbf{V}\circ\iota]\,q_\mathbf{IV}\,p_\mathbf{III}p_\mathbf{II}[p_\mathbf{I}\circ\iota]$
\begin{figure}[H]
\centering
\scalebox{0.4}{\includesvg[width=1000pt]{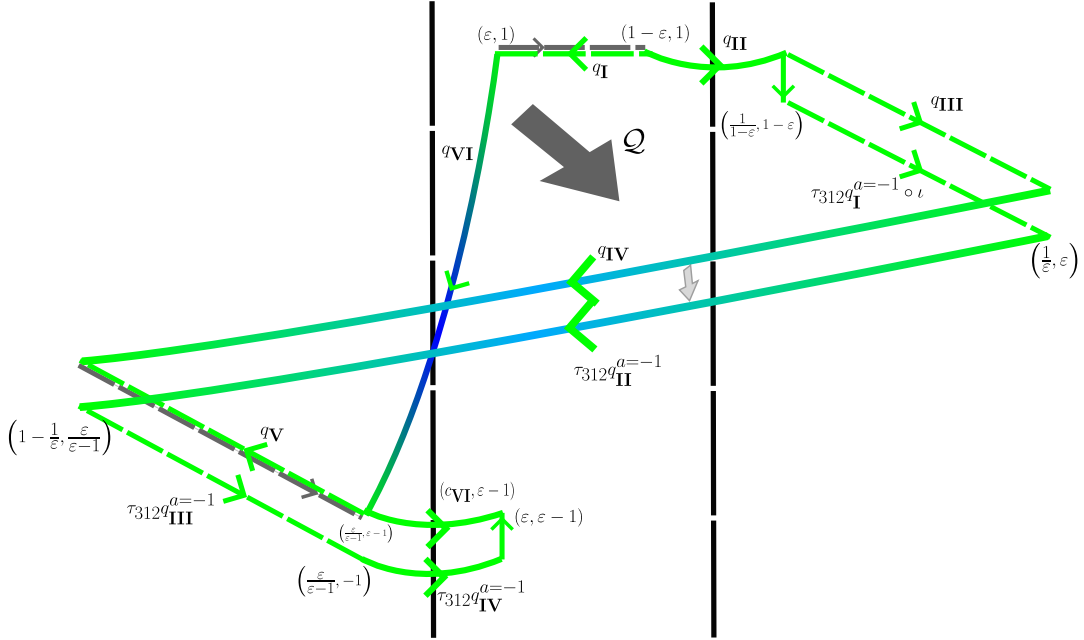}}
\caption{Whiskered $\Q$ followed by a harmless 2-path}
\label{fig:whiskered123}
\end{figure}
picking up $W^{c_\mathbf{VI}}\left(W^{p_\mathbf{V}}\right)^{-1}W^\Q W^{c_\mathbf{I}}$ then applying a final harmless 2-path as above.
\end{enumerate}
\end{constr}

\newpage
\section*{Acknowledgments}
I would like to thank Alexander Schenkel and Robert Laugwitz for many illuminating discussions over the past year. I would also like to thank João Faria Martins for a helpful conversation circa beginning of January.


\appendix
\section{Iterated integrals}
We provide this appendix for two reasons: the first is that we want to recall some basic facts in \cite[Appendix XIX.11]{Kassel} surrounding iterated integrals and single-variable multiple polylogarithms because such facts will be appealed to in Section \ref{sec:Constructing Hex} for deriving the lowest order terms of various 2-holonomies. The second reason is that we wish to present a summary of the material in \cite[Theorem 20]{BRW} for the construction of Drinfeld's Knizhnik-Zamolodchikov associator series as the limit (as $\varepsilon\to0$) of the non-singular factor of a straightforward parallel transport because such a result will be appealed to throughout Sections \ref{sec:Constructing Hex} and \ref{sec:Breen}.


\subsection{Single-variable multiple polylogarithms}
\begin{defi}
The \textbf{Riemann zeta function} is defined as 
\begin{equation}\label{eq:Riemann zeta as infinite sum}
\zeta(s):=\sum_{n=1}^\infty\frac{1}{n^s}
\end{equation}
for $\Re(s)>1$ and analytically continued elsewhere.
\end{defi}
\begin{ex}
The explicit value 
\begin{equation}\label{eq:zeta(2)=pi^2/6}
\zeta(2)=\frac{\pi^2}{6}\quad,
\end{equation}
is known as \textbf{Euler's solution to the Basel problem}. 
\end{ex}
The Riemann zeta function admits the following generalisation.
\begin{defi}
The \textbf{polylogarithm} is defined as
\begin{equation}\label{eq:definition of polylogarithm}
\Li_s(z):=\sum_{n=1}^\infty\frac{z^n}{n^s}
\end{equation}
for $(s\in\bbC,|z|<1)$ and analytically continued elsewhere.
\end{defi}
\begin{rem}
The polylogarithm \eqref{eq:definition of polylogarithm} is indeed an extension of the Riemann zeta function \eqref{eq:Riemann zeta as infinite sum} because $\Li_s(1)=\zeta(s)$ but it is \textit{also} an extension of the natural logarithm function (ergo the name),
\begin{equation}\label{eq:Li_1(z)=-ln(1-z)}
\Li_1(z):=\sum_{n=1}^\infty\frac{z^n}{n}=-\ln(1-z)\quad.
\end{equation}
\end{rem}
In addition to the polylogarithm, there exists another generalisation of the Riemann zeta function.
\begin{defi}\label{def:MZV}
Given $k\in\bbN^*$, the \textbf{multiple zeta function} is defined as 
\begin{equation}\label{eq:multiple zeta function}
\zeta(s_1,\ldots,s_k):=\sum_{n_1>n_2>\cdots>n_k\geq1}^\infty\frac{1}{n_1^{s_1}\cdots n_k^{s_k}}
\end{equation}
for $\{\sum_{j=1}^i\Re(s_j)>i\,|\,1\leq i\leq k\}$ and analytically continued elsewhere. If $s_1\in\bbN^{\geq2}$ and $s_2,\ldots,s_k\in\bbN^*$ then we call \eqref{eq:multiple zeta function} a \textbf{multiple zeta value} (MZV). 
\end{defi}
Finally, we recall the generalisation of both the polylogarithm \eqref{eq:definition of polylogarithm} and the multiple zeta function \eqref{eq:multiple zeta function}.
\begin{defi}\label{def:multiple polylogarithm}
Given $k\in\bbN^*$, the \textbf{multiple polylogarithm} is defined as 
\begin{equation}\label{def:single-variable multiple polylogarithm}
\Li_{s_1,\ldots,s_k}(z_1,\ldots,z_k):=\sum_{n_1>n_2>\cdots>n_k\geq1}^\infty\frac{z_1^{n_1}\cdots z_k^{n_k}}{n_1^{s_1}\cdots n_k^{s_k}}
\end{equation}
for $\{(s_i\in\bbC,|z_i|<1)\,|\,1\leq i\leq k\}$ and analytically continued elsewhere. Setting $z_1=z$ and $z_2=\ldots=z_k=1$, we refer to 
\begin{equation}\label{eq:single-variable multiple polylogarithm}
\Li_{s_1,\ldots,s_k}(z):=\Li_{s_1,\ldots,s_k}(z,\{1\}^{k-1}):=\sum_{n_1>n_2>\cdots>n_k\geq1}^\infty\frac{z^{n_1}}{n_1^{s_1}\cdots n_k^{s_k}}
\end{equation}
as a \textbf{single-variable multiple polylogarithm}.
\end{defi}
\begin{rem}
The single-variable multiple polylogarithm \eqref{eq:single-variable multiple polylogarithm} is clearly a generalisation of the polylogarithm \eqref{eq:definition of polylogarithm} and an extension of the multiple zeta function \eqref{eq:multiple zeta function}. Furthermore, something analogous to \eqref{eq:Li_1(z)=-ln(1-z)} holds, i.e.,
\begin{equation}\label{Li_(1,...,1)(z)}
\Li_{\{1\}^k}(z)=\frac{\big(-\ln(1-z)\big)^k}{k!}\quad.
\end{equation}
\end{rem}
As in \cite[(11.18) and (11.19) of Appendix XIX.11]{Kassel}, we record the following integral expressions which are trivial to prove:
\begin{equation}\label{eq:(11.18/19) of Kassel}
\Li_{s_1+1,s_2\ldots,s_k}(z)=\int_0^z\frac{\Li_{s_1,s_2\ldots,s_k}(r)}{r}\dd r\qquad,\qquad\Li_{1,s_1,\ldots,s_k}(z)=\int_0^z\frac{\Li_{s_1,\ldots,s_k}(r)}{1-r}\dd r\quad.
\end{equation}

\begin{rem}\label{rem:iterated integral defined inductively}
Given $\bbC$-valued 1-forms $\{\omega_i:=f_i(s)\dd s\}_{1\leq i\leq n}$ defined over the interval $[a,b]\subset\bbR\,$, define the iterated integral $\int_a^b\omega_1\cdots\omega_n$ inductively by
\begin{subequations}
\begin{equation}
\int_a^b\omega_1=\int_a^bf_1(s)\dd s
\end{equation}
and
\begin{equation}\label{eq:iterated integral notation}
\int_a^b\omega_1\cdots\omega_n=\int_a^bf_1(s)\left(\int_a^s\omega_2\cdots\omega_n\right)\dd s\quad.
\end{equation}
\end{subequations}
Consider the particular 1-forms:
\begin{equation}\label{eq:particular 1-forms}
\Omega_0:=\frac{\dd s}{s}\qquad,\qquad\Omega_1:=\frac{\dd s}{s-1}\quad,
\end{equation}
then, as in \cite[(11.14) of Appendix XIX.11]{Kassel}, a simple argument by induction on $k\in\bbN^*$ gives us:
\begin{subequations}\label{subeq:int_a^b Omega^k}
\begin{alignat}{2}
\int_a^b\Omega_0^k=&\frac{1}{k!}\left(\ln\frac{b}{a}\right)^k\quad&&,\label{eq:int_a^b Omega_0^k}\\
\int_a^b\Omega_1^k=&\frac{1}{k!}\left(\ln\frac{1-b}{1-a}\right)^k\quad&&.\label{eq:int_a^b Omega_1^k}
\end{alignat}
\end{subequations}
For $g\in\bbN^*$ and $p_1,q_1,\ldots,p_g,q_g\in\bbN^*$, we can use \eqref{eq:(11.18/19) of Kassel} to notice that
\begin{equation}\label{eq:iterated integral expression for multiple zeta}
\int_0^1\Omega_0^{p_1}\Omega_1^{q_1}\cdots\Omega_0^{p_g}\Omega_1^{q_g}=(-1)^{\sum_{i=1}^gq_i}\zeta(p_1+1,\{1\}^{q_1-1},\ldots,p_g+1,\{1\}^{q_g-1})
\end{equation}
thus \eqref{eq:zeta(2)=pi^2/6} gives us
\begin{equation}\label{eq:int_0^1 Omega_0 Omega_1=-pi^2/6}
\int_0^1\Omega_0\Omega_1=-\zeta(2)=-\frac{\pi^2}{6}\quad.
\end{equation}
\end{rem}


\subsection{BRW's construction of Drinfeld's KZ associator series}\label{subsec:BRW construction of Phi}
\begin{constr}\label{con:formal linear ODEs in general}
Following \cite[Subsection 1.2]{BRW}, given a complex vector space $V$ and the unital associative $\bbC$-algebra $\End(V)$, we fix $Y\in\C^\infty([0,1],\bbC)\otimes_\bbC\End(V)$ and consider the following formal linear ODE in the formal power series $\omega\in\big(\C^\infty([0,1],\bbC)\otimes_\bbC\End(V)\big)[[\hbar]]$,
\begin{equation}\label{eq:formal linear ODE}
\frac{\dd\omega}{\dd s}=\hbar Y\omega\qquad,\qquad\omega(0)=1\quad.
\end{equation}
Denote the integral from $0$ to $s\in[0,1]$ by $I_{s0}$ and the left multiplication by $Y$ as $L_Y$ then \eqref{eq:formal linear ODE} gives us
\begin{equation}
\omega(s)-1=\hbar\big(I_{s0}\circ L_Y\big)(\omega)\qquad\implies\qquad\Big(\Id-\hbar\big(I_{s0}\circ L_Y\big)\Big)(\omega)=1\quad.
\end{equation}
The factor of $\hbar$ is what allows us to invert the LHS operator as follows,
\begin{align}
\omega(s)=&\Big(\Id-\hbar\big(I_{s0}\circ L_Y\big)\Big)^{-1}(1)\nn\\=&\sum_{r=0}^\infty\hbar^r\big(I_{s0}\circ L_Y\big)^{\circ r}(1)\nn\\=&1+\sum_{r=1}^\infty\hbar^r\int_0^sY(s_1)\left[\int_0^{s_1}Y(s_2)\left(\cdots\int_0^{s_{r-1}}Y(s_r)\dd s_r\cdots\right)\dd s_2\right]\dd s_1\label{eq:formal ODE solution as iterated integrals}
\end{align}
\end{constr}
Following \cite[Subsection 2.2]{BRW}, given two elements $A,B\in\End(V)$, we define the (flat) formal \textbf{Knizhnik-Zamolodchikov connection} 
\begin{equation}
\Gamma(A,B):=\hbar\left(\frac{A}{x}+\frac{B}{x-1}\right)\dd x
\end{equation}
on $(0,1)\subset\bbR$. The parallel transport from $0<\varepsilon<\tfrac{1}{4}$ to $\frac{3}{4}<1-\varepsilon<1$ along the affine path $c_\mathbf{I}(s):=\varepsilon+s(1-2\varepsilon)$ is given by \eqref{eq:formal ODE solution as iterated integrals} which, using the juxtaposition notation for iterated integrals \eqref{eq:iterated integral notation}, now becomes
\begin{equation}\label{eq:c_epsilon PT}
W^{c_\mathbf{I}}=1+\sum_{r=1}^\infty\hbar^r\sum_{i_1,\ldots,i_r=0}^1\int_\varepsilon^{1-\varepsilon}\frac{\mathrm{d}u_1}{u_1-i_1}\cdots\frac{\mathrm{d}u_r}{u_r-i_r}\A_{i_1}\cdots\A_{i_r}
\end{equation}
where $\A_0:=A$ and $\A_1:=B$. Using \eqref{eq:particular 1-forms}, we may rewrite the parallel transport \eqref{eq:c_epsilon PT} more compactly as 
\begin{equation}\label{eq:super compact expression for W^(c_I)}
W^{c_\mathbf{I}}=\sum_{r=0}^\infty\hbar^r\sum_{i_1,\ldots,i_r=0}^1\int_\varepsilon^{1-\varepsilon}\prod_{m=1}^r\Omega_{i_m}A^{1-i_m}B^{i_m}
\end{equation}

\begin{rem}\label{rem:at least diverging terms}
If $[A,B]=0$ then, instead, we can simply exponentiate the integration of the connection,
\begin{equation}\label{eq:troublesome terms exposed}
W^{c_\mathbf{I}}=\exp\left[\hbar\int_\varepsilon^{1-\varepsilon}\left(\frac{A}{x}+\frac{B}{x-1}\right)\dd x\right]=e^{\hbar\ln(\varepsilon)B}e^{\hbar\ln(1-\varepsilon)(A-B)}e^{-\hbar\ln(\varepsilon)A}
\end{equation}
and this clearly diverges as $\varepsilon\to0$.
\end{rem}
Remark \ref{rem:at least diverging terms}, particularly \eqref{eq:troublesome terms exposed}, makes it clear that we need to extract \textit{at least} the terms $e^{\hbar\ln(\varepsilon)B}$ and $e^{-\hbar\ln(\varepsilon)A}$ in order to acquire something that is \textit{potentially} finite. In other words, we want to show that 
\begin{equation}\label{eq:potentially finite Drinfeld}
\Phi(A,B):=\lim_{\varepsilon\to0}\Phi^\varepsilon(A,B):=\lim_{\varepsilon\to0}\left(e^{-\hbar\ln(\varepsilon)B}W^{c_\mathbf{I}}e^{\hbar\ln(\varepsilon)A}\right)
\end{equation}
is actually finite.
\begin{constr}
We factorise a continuous piecewise smooth reparametrisation of the affine path $c_\mathbf{I}(s):=\varepsilon+s(1-2\varepsilon)$ as the concatenation of two ``exponential half paths". Explicitly, introduce:
\begin{equation}
c_1(s):=\tfrac{1}{2}(2\varepsilon)^{1-s}\qquad,\qquad c_2(s):=1-\tfrac{1}{2}(2\varepsilon)^s
\end{equation}
and 
\begin{equation}
\gamma(s):=\begin{cases}
    \frac{\tfrac{1}{2}(2\varepsilon)^{1-2s}-\varepsilon}{1-2\varepsilon}\quad\,\,,\qquad&0\leq s\leq\tfrac{1}{2}\\
    \frac{1-\tfrac{1}{2}(2\varepsilon)^{2s-1}-\varepsilon}{1-2\varepsilon}\,\,,\qquad&\tfrac{1}{2}\leq s\leq1
    \end{cases}
\end{equation}
then it is straightforward to check that $c_\mathbf{I}\circ\gamma=c_2c_1$. Furthermore, the \textbf{interval inversion} 
\begin{subequations}\label{eq:interval inversion}
\begin{alignat}{2}
\iota:[0,1]&\longrightarrow[0,1]\qquad&&,\\
s&\longmapsto1-s\qquad&&,
\end{alignat} 
\end{subequations}
is such that
\begin{subequations}
\begin{alignat}{2}
c_1=&\iota\circ c_2\circ\iota\qquad&&,\\
\iota^*\left(\Gamma(A,B)\right)&=\Gamma(B,A)\qquad&&.\label{eq:nice fact about iota}
\end{alignat}
\end{subequations}
These facts (together with the standard functoriality properties of parallel transport in Remark \ref{rem:hol axioms}) give us 
\begin{equation}
^{\Gamma(A,B)}W^{c_\mathbf{I}}={}^{\Gamma(A,B)}W^{c_2}\left({}^{\Gamma(B,A)}W^{c_2}\right)^{-1}\quad.
\end{equation}
In particular, \cite[Lemma 19]{BRW} tells us that 
\begin{subequations}
\begin{equation}
{}^{\Gamma(A,B)}W^{c_2}=e^{\hbar\ln(\varepsilon)B}e^{\hbar\ln(2)B}\Xi^\varepsilon_{B,A}
\end{equation}
with
\begin{equation}
\Xi^\varepsilon_{B,A}:=1+\sum_{r=1}^\infty\hbar^r\sum_{\ell_1,\ldots,\ell_r=0}^\infty\hbar^{\sum_{i=1}^r\ell_i}\,\I^\varepsilon_{\ell_1\ldots\ell_r}\ad^{\ell_1}_B(A)\cdots\ad^{\ell_r}_B(A)
\end{equation}
and where
\begin{equation}\label{eq:epsilon iterated integral for associator}
\I^\varepsilon_{\ell_1\ldots\ell_r}:=\frac{1}{\ell_1!\cdots\ell_r!}\int_0^{-\ln(2\varepsilon)}\frac{\tau_1^{\ell_1}}{2e^{\tau_1}-1}\dd\tau_1\cdots\frac{\tau_r^{\ell_r}}{2e^{\tau_r}-1}\dd\tau_r\qquad.
\end{equation}
\end{subequations}
The proof of \cite[Lemma 19]{BRW} shows that \eqref{eq:epsilon iterated integral for associator} is indeed finite under the limit $\lim_{\varepsilon\to0}$. Therefore, \eqref{eq:potentially finite Drinfeld} tells us that\footnote{Notice as well that this form makes it obvious that \begin{equation}\label{eq:Drinfeld swap is inverse}
\Phi(A,B)=\Phi(B,A)^{-1}\quad.
\end{equation}} 
\begin{subequations}\label{subeq:BRW's explicit Phi formula}
\begin{equation}\label{eq:associator as factor of PT}
\Phi(A,B)=e^{\hbar\ln(2)B}\Xi_{B,A}\big(\Xi_{A,B}\big)^{-1}e^{-\hbar\ln(2)A}
\end{equation}
with
\begin{equation}
\Xi_{B,A}:=1+\sum_{r=1}^\infty\hbar^r\sum_{\ell_1,\ldots,\ell_r=0}^\infty\hbar^{\sum_{i=1}^r\ell_i}\,\I_{\ell_1\ldots\ell_r}\ad^{\ell_1}_B(A)\cdots\ad^{\ell_r}_B(A)
\end{equation}
and where
\begin{equation}\label{eq:iterated integral for associator}
\I_{\ell_1\ldots\ell_r}:=\frac{1}{\ell_1!\cdots\ell_r!}\int_0^\infty\frac{\tau_1^{\ell_1}}{2e^{\tau_1}-1}\dd\tau_1\cdots\frac{\tau_r^{\ell_r}}{2e^{\tau_r}-1}\dd\tau_r\qquad.
\end{equation}
\end{subequations}
\end{constr}

\end{document}